\documentclass{amsart}
\usepackage[utf8]{inputenc}
\usepackage{amssymb}
\usepackage{amsmath}
\usepackage{amsthm}
\usepackage{graphicx}
\usepackage{biblatex}
\usepackage{mathdots}
\usepackage{hyperref}
\usepackage[normalem]{ulem}
\usepackage[dvipsnames]{xcolor}
\usepackage{ytableau}
\usepackage[margin=1in]{geometry}
\usepackage{relsize}

\usepackage{thm-restate}

\makeatletter
\newcommand*\bigcdot{\mathpalette\bigcdot@{.7}}
\newcommand*\bigcdot@[2]{\mathbin{\vcenter{\hbox{\scalebox{#2}{$\m@th#1\bullet$}}}}}
\makeatother

\usepackage{tikz}

\usepackage{caption}
\usepackage{subcaption}

\title{Generalized Pitman--Stanley polytope: vertices and faces}

\author[W. T. Dugan]{William T. Dugan}
\address[W.\ T.\ Dugan]{Department of Mathematics and Statistics, University of Massachusetts, Amherst, MA, 01003, United States} 
\email{wtdugan@math.umass.edu}
\urladdr{\url{https://sites.google.com/view/william-dugan/home}}

\author[M. Hegarty]{Maura Hegarty}
\address[M.\ Hegarty]{Operations Research Center, Massachusetts Institute of Technology, Cambridge, MA, 02139, United States} 
\email{mshegart@mit.edu}

\author[A. H. Morales]{Alejandro H. Morales}
\address[A.\ H.\ Morales]{Department of Mathematics and Statistics, University of Massachusetts, Amherst, MA, 01003, United States} 
\email{ahmorales@math.umass.edu}
\urladdr{\url{http://ahmorales.combinatoria.co/}}

\author[A. Raymond]{Annie Raymond}
\address[A.\ Raymond]{Department of Mathematics and Statistics, University of Massachusetts, Amherst, MA, 01003, United States} 
\email{raymond@math.umass.edu}
\urladdr{\url{https://people.math.umass.edu/~raymond/}}

\theoremstyle{definition}
\numberwithin{equation}{section}
\newtheorem{theorem}{Theorem}[section]
\newtheorem{corollary}[theorem]{Corollary}

\newtheorem{proposition}[theorem]{Proposition}
\newtheorem{lemma}[theorem]{Lemma}

\newtheorem{definition}[theorem]{Definition}
\newtheorem{example}[theorem]{Example}
\newtheorem{remark}[theorem]{Remark}

\DeclareMathOperator{\PS}{PS}
\DeclareMathOperator{\sgn}{sgn}

\DeclareMathOperator{\rev}{rev}

\DeclareMathOperator{\SYT}{shSYT}
\usepackage{xcolor}

\newcommand{\defn}[1]{{\color{blue} \it {#1}}}

\newcommand \RR{\mathbb{R}}

\newcommand \NN{\mathbb{N}}

\newcommand \FF{\mathcal{F}}
\newcommand \uu{{\mathbf u}}
\newcommand \vv{{\mathbf v}}
\newcommand \ww{{\mathbf w}}
\newcommand \xx{{\mathbf x}}
\newcommand \aaa{{\mathbf a}}
\newcommand \ccc{{\mathbf c}}
\newcommand \bb{{\mathbf b}}
\newcommand \jj{{\mathbf j}}
\newcommand \kk{{\mathbf k}}

\newcommand \ind{\textup{index}}

\definecolor{Mauras_Green}{rgb}{0.4, 0.8, 0.5}

\begin{document}

\maketitle

\begin{abstract}
    In 1999, Pitman and Stanley introduced the polytope bearing their name along with a study of its faces, lattice points, and volume.  The Pitman-Stanley polytope is well-studied due to its connections to probability, parking functions, the generalized permutahedra, and flow polytopes. Its lattice points correspond to plane partitions of skew shape with entries 0 and 1. Pitman and Stanley remarked that their polytope can be generalized so that lattice points correspond to plane partitions of skew shape with entries $0,1, \ldots , m$. Since then, this generalization has been untouched. We study this generalization and show that it can also be realized as a flow polytope of a grid graph. We give multiple characterizations of its vertices in terms of plane partitions of skew shape and integer flows. For a fixed skew shape, we show that the number of vertices of this polytope is a polynomial in $m$ whose leading term, in certain cases, counts standard Young tableaux of a shifted shape. Moreover, we give formulas for the number of faces, as well as generating functions for the number of vertices.
\end{abstract}

%--------------------------------------Section-------------------------------------%

\section{Introduction}

The eponymous Pitman-Stanley polytope introduced in \cite{Pitman_Stanley_1999} is a well-studied polytope in geometric, algebraic, and enumerative combinatorics. This polytope is defined as follows. For a positive integer $n$ and vectors $\textbf{a}$ and $\textbf{b}$ in $\mathbb{N}^n$, let
\begin{align*}
    \PS_n(\textbf{a},\textbf{b}) := \{\xx\in\mathbb{R}_{\geq 0}^n\quad\big{|}
     \quad b_1+\cdots+b_i \leq x_1+\cdots+x_i\leq a_1+\cdots+a_i \ \text{ for } i=1,\ldots,n\},
\end{align*}
i.e., $\PS_n({\bf a},{\bf b})$ consists of the nonnegative vectors $\xx$ in $\mathbb{R}^n$ that are between vectors ${\bf a}$ and ${\bf b}$ in {\em dominance order}. Recall that, for $\vv,\ww\in \NN^n$, we say that $\vv$ \defn{dominates} $\ww$ and denote it by $\vv \trianglerighteq \ww$ if $\sum_{j=1}^i v_j \geq \sum_{j=1}^i w_j$ for every $i=1,\ldots,n$. Later, Baldoni and Vergne \cite{Baldoni_Vergne_2008} realized $\PS_n({\bf a})$ as a {\em flow polytope} of a certain graph (see Figure~\ref{fig:comparing flows PS}). In addition, Pitman and Stanley related this polytope to the {\em associahedron} \cite[Sec. 6]{Pitman_Stanley_1999}, Postnikov showed that $\PS_n({\bf a}):=\PS_n(\textbf{a},\textbf{0})$ is an example of a {\em generalized permutahedron} \cite{AP}, and Bidkhori related it to {\em lattice path matroids} \cite[\S 4.5]{BidPhd}.  The volume of the polytope  $\PS_n(\textbf{a})$ is of interest in probability \cite[\S 2]{Pitman_Stanley_1999} and is related to {\em parking functions}  (see \cite[Thm. 11]{Pitman_Stanley_1999} and \cite[\S 13.4]{Yan_pf_handbook}), the Tutte polynomial \cite{KonvPak} and {\em Cayley compositions} \cite{KonvPak2}.

An important result of \cite{Pitman_Stanley_1999} is that the polytope $\PS_n({\bf a})$ is combinatorially equivalent to a product of simplices. We denote by $\Delta_d$ the $d$-dimensional simplex. Given  ${\bf a}=(a_1,\ldots,a_n)$  with $a_1>0$, its \defn{signature} $\sgn({\bf a}):=(c_1,\ldots,c_k) \in \mathbb{N}^k$ is the sequence such that  ${\bf a}=
(*, \underbrace{0, \ldots, 0}_{c_1-1}, *,  \underbrace{0, \ldots, 0}_{c_2-1}, \ldots, *,  \underbrace{0, \ldots, 0}_{c_k-1})$
where $*$ and $0$ represent a positive $a_i$ and zero $a_i$ respectively.

\begin{theorem}[{Pitman--Stanley \cite[Thm. 20]{Pitman_Stanley_1999}}] \label{thm:PS faces}
For $\aaa \in \NN^n$ with $a_1>0$ and  $\sgn(\aaa)=(c_1,\ldots,c_k)$, the polytope  $\PS_n({\bf a})$ is combinatorially equivalent to the product of simplices $\Delta_{c_1}\times \cdots \times \Delta_{c_k}$. In particular, if each $a_i>0$, then $\PS_n({\bf a})$ is combinatorially equivalent to an $n$-cube.
\end{theorem}

Recall that a {\em plane partition} of skew shape $\lambda/\mu$ is an array $\pi$ of nonnegative integers of shape $\lambda/\mu$  that is weakly decreasing in rows and columns. Pitman and Stanley also showed that the lattice points of $\PS_n(\textbf{a},\textbf{b})$ are in correspondence with {\em plane partitions} of skew shape $$\theta({\bf a},{\bf b}):=(a_1+\cdots+a_{n},\ldots,a_1+a_2,a_1)/(b_1+\cdots+ b_{n},\ldots,b_1+b_2,b_1)$$ with entries $0$ and $1$. When ${\bf b}=\mathbf{0}$, note that $\theta({\bf a}, {\bf b})$ is a straight shape, and in that case, Pitman and Stanley also mentioned in \cite[Sec. 5]{Pitman_Stanley_1999} how their polytope can be generalized so that lattice points correspond to plane partitions of the same shape with entries $0,1,\ldots,m$ for a fixed positive integer $m$. This generalization is the object of this paper, which is devoted to the enumeration of its vertices and faces, and its follow-up \cite{PSvolume}, which is devoted to its volume and Ehrhart polynomial. The {\em $m$th generalized Pitman--Stanley polytope} is defined as follows.
\begin{align*}
\PS^m_n({\bf a},{\bf b}) := \{ & (x_{ij})_{1\leq i\leq n,1\leq j \leq m}  \in \mathbb{R}_{\geq 0}^{n\times m} \quad\big{|} \quad b_1 + \dots + b_i \leq x_{1m}+x_{2m}+\ldots+x_{im} \leq \cdots \\ &\cdots\leq x_{12}+x_{22}+\ldots+x_{i2}\leq  x_{11}+x_{21}+\ldots+x_{i1} \leq a_1 + \dots + a_i \text{ for } i=1,\ldots,n\},
\end{align*}
i.e., $\PS^m_n({\bf a},{\bf b})$ consists of nonnegative matrices $(x_{ij})$ in $\mathbb{R}^{n\times m}$ whose columns form a multichain of length $m$ between ${\bf a}$ and ${\bf b}$ in dominance order. We show that the polytope $\PS_n^m({\bf a},{\bf b})$ is also a flow polytope of a grid graph $G(n,m)$. See both Figures~\ref{fig:comparing flows PS},\ref{fig:graph skew PS} for $m=1$ and Figure~\ref{fig:graph gen skew PS} for general $m$, respectively. 

\begin{restatable*}[]{theorem}{gPSisflow}
\label{thm: gPS is flow polytope}
The $m$th Pitman--Stanley polytope $\PS_n^m(\textbf{a},\textbf{b})$ is integrally equivalent to the flow polytope $\mathcal{F}_{G(n, m)}(\textbf{a},\textbf{b})$. 
\end{restatable*}

This allows us to use the theory of flow polytopes (e.g., \cite{Baldoni_Vergne_2008,Hille_2003,Mészáros_Morales_2018}) to study the faces of  $\PS_n^m(\textbf{a},\textbf{b})$. In Theorem \ref{thm:bijflowpp}, we give a bijection between the lattice points of $\PS_n^m(\textbf{a},\textbf{b})$ and plane partitions of  shape $\theta({\bf a},{\bf b})$ with entries $0,1,\ldots,m$. Note that Liu--M\'esz\'aros--St.\ Dizier \cite{LMStD1,LMStD2} have realized {\em Gelfand--Tsetlin polytopes}, whose lattice points correspond to (skew) semistandard Young tableaux with bounded entries, as flow polytopes. In the skew case, these tableaux and plane partitions with bounded entries are similar but different objects (e.g., see \cite[Ch. 7]{EC2}) and this difference extends to the associated flow polytopes.

\begin{figure}
    \centering
      \centering
  \begin{subfigure}[b]{0.2\textwidth}
    \centering
     \includegraphics[]{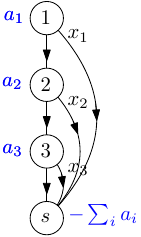}
    \caption{}
      \label{fig:comparing flows PS}
\end{subfigure}
\begin{subfigure}[b]{0.3\textwidth}
  \centering
  {\includegraphics{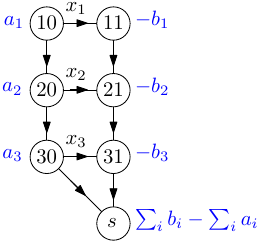}}
  \caption{}
  \label{fig:graph skew PS}
\end{subfigure}
 \begin{subfigure}[b]{0.4\textwidth}
 \includegraphics[]{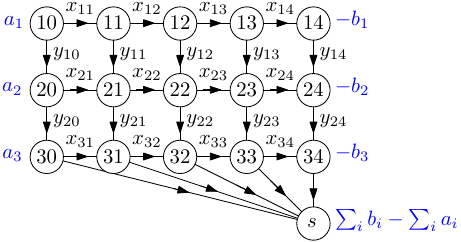}
 \caption{}
  \label{fig:graph gen skew PS}
 \end{subfigure}
     \caption{The first graph realizes $\PS^1_n({\bf a}, \mathbf{0})$ for $n=3$ as shown by Baldoni and Vergne \cite{Baldoni_Vergne_2008}. The second and third graphs are $G(n,1)$ and $G(n,m)$ for $m=4$ which realize the original and generalized Pitman--Stanley polytopes $\PS^1_n({\bf a},{\bf b})$ and $\PS^m_n({\bf a},{\bf b})$ as flow polytopes, respectively for $n=3$.}
    \label{fig:comparing flows PS and gen PS}
\end{figure}

\subsection{Characterization of vertices}

In Theorem~\ref{char: vertices G as forests}, for completeness, we give a characterization of the  vertices of $\PS_n^m(\textbf{a},\textbf{b})$ as flows of $\mathcal{F}_{G(n, m)}(\textbf{a},\textbf{b})$ whose supports are  subgraphs of $G(n,m)$ that are forests. In the context of flow polytopes, this was previously known by Gallo--Sodini \cite{GalloSodini}, and it is also related to a special case of a result of Hille \cite{Hille_2003} characterizing all faces of flow polytopes with nonnegative netflows. This gives an explicit characterization of the flows on the grid corresponding to vertices in terms of flows that \emph{split} and \emph{merge} in particular ways (see Corollary~\ref{conj:vert flows skew}).

Since the lattice points of $\PS_n^m(\textbf{a},\textbf{b})$ are in bijection with plane partitions, in Theorem~\ref{thm:vertexplanepartitions} we characterize the plane partitions corresponding to vertices, which we call \defn{vertex plane partitions} (see Definition~\ref{def:vertex plane partitions general}).

\subsection{Enumeration of vertices}

After giving a flow and plane partition interpretation of the vertices of $\PS_n^m(\textbf{a},\textbf{b})$, we turn to enumerating such vertices. Let $v^{(n,m)}(\aaa,\bb)$ be the number of vertices of $\PS_n^m(\textbf{a},\textbf{b})$ and let $v_\textup{unsplit}^{(n,m)}(\aaa,\bb)$  be the number of \emph{unsplittable flows}. In the straight shape case, i.e., $\textbf{b}=\textbf{0}$, this number $v^{(n,m)}_\textup{unsplit}(\textbf{a}):=v_\textup{unsplit}^{(n,m)}(\textbf{a},\textbf{0})$ gives the number of vertices of the polytope $\PS_n^m(\textbf{a},\textbf{0})$, i.e., it is equal to $v^{(n,m)}(\textbf{a}):=v^{(n,m)}(\textbf{a}, \textbf{0})$. In Theorem~\ref{thm:fixingfirst},  we give a recurrence for $v_\textup{unsplit}^{(n,m)}({\bf a},{\bf b})$ by fixing the flows on the first column of horizontal edges in $G(n,m)$. As a corollary, we obtain the following recurrence for the number $v^{(n,m)}({\bf a})$ of vertices in terms of a generalization of the dominance order called {\em $z$-domination} denoted by $\trianglerighteq_z$(see Definition~\ref{def:zdomination}) and where $\chi(\aaa)$ is the $0/1$-vector with the same support as $\aaa$.

\begin{restatable*}[]{corollary}{corfixingfirst}
Let $\aaa \in \NN^n$, then 
$$v^{(n,m)}({\bf a})=\sum_{\substack{\jj\in \{0,1\}^n:\, \chi(\aaa) \,\trianglerighteq_1 
 \,\jj}} v^{(n,m-1)}(\jj).$$  \end{restatable*}

We also give a recurrence in Theorem~\ref{thm:fixinglast} by fixing the flows on the last column of horizontal edges in $G(n,m)$. This further allows us to show that the number $v^{(n,m)}(\textbf{a})$ of vertices is a polynomial in $m$ and to calculate generating functions like $\sum_{m\geq 0} v^{(n,m)}(\aaa) x^m$ via the {\em transfer-matrix method} (see Corollary~\ref{cor: gf and polynomiality case b=0}). In addition, we show that as a polynomial in $m$, the number $v^{(n,m)}(\textbf{a},\bb)$ of vertices of $\PS_n^m(\aaa,\bb)$ has the following nonnegative expansion.

\begin{restatable*}[]{corollary}{vertexpositivityotherbasis}
 \label{lem:positivity other basis}
For $\aaa,\bb\in \NN^n$, we have that the number $v^{(n,m)}(\aaa,\bb)$ of vertices of $\PS_n^m(\aaa,\bb)\equiv \mathcal{F}_{G(n, m)}(\textbf{a},\textbf{b})$
\begin{equation} \label{eq: vertex skew case other basis}
v^{(n,m)}(\aaa,\bb)=\sum_{k=1}^{|\theta(\aaa,\bb)|} p_{\aaa,\bb,k} \binom{m+1}{k},
\end{equation}
where $p_{\aaa,\bb,k}$ is the number of vertex plane partitions of shape $\theta(\aaa,\bb)$ with entries at most $k-1$ such that each of $\{0,1,\ldots, k-1\}$ appears at least once in the plane partition. 
\end{restatable*}

The vertex plane partitions counted by the above result are of independent interest. For instance, when $\aaa\in\{0,1\}^n$ and $\bb=\mathbf{0}$, then $\theta(\aaa,\bb)$ is a partition $\lambda$ inside the $n$-staircase $\delta_n$, and the leading term $p_{\aaa, \mathbf{0},|\lambda|}$ of \eqref{eq: vertex skew case other basis}  counts vertex plane partitions with distinct entries called  \defn{standard vertex plane partitions}. 
Moreover, such standard vertex plane partitions are in bijection with {\em standard shifted Young tableaux} (SYT) of shape $\lambda'$.

\begin{restatable*}[]{theorem}{leadingtermvertices}
\label{thm:leading term v straight shape}
For $\aaa \in \{0,1\}^n$ and $\bb={\bf 0}$, the leading term of $|\lambda|!\cdot v^{(n,m)}(\aaa,{\bf 0})$ is   $p_{\aaa,{\bf 0},|\lambda|}=\SYT(\lambda')$, where $\lambda=\lambda(\aaa)$.
\end{restatable*}

 In the case $\aaa = {\bf 1}$, these tableaux are exactly all SYT of shifted staircase shape $(n,n-1,\ldots,1)$
(see Corollary~\ref{lem:top coeff SYT of shifted shape}), and thus this coefficient has the following product formula
\cite[\href{https://oeis.org/A003121}{A003121}]{oeis}:
\[
p_{{\bf 1},{\bf 0},\binom{n+1}{2}} \,=\, \frac{\binom{n+1}{2}! \cdot 1!\cdot 2!\cdots (n-1)!}{1!\cdot 3! \cdots (2n-1)!}.
\]

\subsection{Enumeration of $d$-faces}
Lastly, we give in Theorem~\ref{thm::recursive_formula_on_faces} a recurrence for the number of $d$-faces of $\PS_n^m({\bf a})$ by fixing the flows on the first column of horizontal edges in $G(n,m)$. This opens the door to further study the number and lattice structure of faces of $\PS_n^m(\aaa,\bb)$ in future work.

\subsubsection*{Outline} 
The paper is organized as follows. Section \ref{Background} covers background necessary to understand the rest of the paper. Section \ref{The generalized Pitman-Stanley polytope} introduces the $m$th Pitman-Stanley polytope $\PS_n^m(\aaa,\bb)$, its lattice points, and a realization as a flow polytope. We then give characterizations for the vertices of $\PS_n^m(\aaa,\bb)$ in terms of flows and plane partitions in Section~\ref{sec: char vertices}. Then in Section~\ref{sec:enumeration vertices}, we study recurrences, generating functions (with Tables of examples in Appendix~\ref{appendix: tables}), and identities for the number $v^{(n,m)}(\aaa,\bb)$ of vertices of $\PS_n^m(\aaa,
\bb)$. In Section~\ref{sec: enumeration d faces}, we give recurrences for the number of $d$-faces of $\PS_n^m(\aaa)$. Finally, in Section~\ref{sec: final remarks}, we discuss our final remarks and future work.

Examples and code related to this work are available on this  \href{https://cocalc.com/ahmorales/generalizedPitmanStanleyfaces/generalizedPSfaces}{link}\footnote{{https://cocalc.com/ahmorales/generalizedPitmanStanleyfaces/generalizedPSfaces}}.

%--------------------------------------Section-------------------------------------%

\section{Background and preliminaries}\label{Background}
This section will define important concepts used in this article. These concepts include the Pitman-Stanley polytope and plane partitions in Subsection \ref{sec:plane partitions}, flow polytopes in Subsection \ref{Flow polytopes}, and vertices of flow polytopes in Subsection~\ref{subsec:vertices FG}.

\subsection{Plane Partitions} \label{sec:plane partitions}

\begin{definition}
Given a skew shape $\lambda/\mu$ and a nonnegative integer $m$, a \defn{plane partition of skew shape $\lambda/\mu$}  is an array $T$ of nonnegative integers in the diagram of shape $\lambda/\mu$  that is weakly decreasing in rows and columns. We usually restrict the entries to be in $\{0,1,\ldots, m\}$ for some nonnegative integer $m$.

\end{definition}

\begin{example} The five plane partitions of shape $(2, 1)$ with $m=1$ are
\begin{align*}
    \ytableausetup{smalltableaux}   
&    \begin{ytableau}
        0 & 0 \\
        0  
    \end{ytableau}\,,  & 
    &\begin{ytableau}
        1 & 0 \\
        0  
    \end{ytableau}\,, & 
    &\begin{ytableau}
        1 & 1 \\
        0  
    \end{ytableau}\,, & 
    &\begin{ytableau}
        1 & 0 \\
        1  
    \end{ytableau}\,,  &
    &\begin{ytableau}
        1 & 1 \\
        1 
    \end{ytableau}.
\end{align*}

The six plane partitions of skew shape $(3,1)/(1,0)$ with $m=1$ are
\begin{align*}
    \ytableausetup{smalltableaux}   
&    \begin{ytableau}
        \none & 0 & 0 \\
        0  
    \end{ytableau}\,,  & 
    &\begin{ytableau}
        \none & 1 & 0 \\
        0  
    \end{ytableau}\,, & 
    &\begin{ytableau}
        \none & 1 & 1 \\
        0  
    \end{ytableau}\,, & 
    &\begin{ytableau}
        \none& 0 & 0 \\
        1  
    \end{ytableau}\,, &
    &\begin{ytableau}
       \none& 1 & 0 \\
        1  
    \end{ytableau}\,, &
    &\begin{ytableau}
        \none& 1 & 1 \\
        1 
    \end{ytableau}.
\end{align*}

\end{example}

\subsection{Shifted standard tableaux} \label{sec:shifted SYT}

A partition $\nu=(\nu_1,\ldots,\nu_l)$ is \defn{strict} if $\nu_1>\nu_2>\cdots$. For a strict partition $\nu$, the \defn{shifted shape} of $\nu$ is an array with $\nu_i$ boxes in row $i$, indented to the right to start at $(i,i)$. Given a cell $(i,j)$ in the shifted shape, its \defn{shifted hook-length $h^*(i,j)$} is the length of the \defn{shifted hook}: $\{(i,j') \mid j'\geq j\} \cup \{(i',j) \mid i'>i\} \cup \{(j+1,j') \mid j'\geq j+1\}$.

\begin{definition}
    A \defn{standard tableau} of shifted shape $\nu$ is an array $T$ with each of the entries $1,2,\ldots,|\nu|$ increasing in rows and columns. Denote the number of such tableaux by $\SYT(\nu)$.
\end{definition}

\begin{example} The two standard tableaux of shifted shape $(3,2,1)$ are
\begin{align*}
    \ytableausetup{smalltableaux}   
&    \begin{ytableau}
        1 & 2 & 3 \\
        \none  & 4 & 5\\
        \none & \none & 6
    \end{ytableau},  & 
    &\begin{ytableau}
        1 & 2 & 4 \\
        \none  & 3 & 5\\
        \none & \none & 6
    \end{ytableau}. 
\end{align*}
\end{example}

The number of standard tableaux of shifted shape $\nu$ is given by a celebrated product formula known as the shifted hook-length formula by Thrall \cite{Thrall}. 

\begin{theorem}[{see \cite[Ex. 3.21]{sagan2013symmetric}}] \label{thm: hlf shifted shapes}
For a strict partition $\lambda$, the number of standard tableaux of shifted shape is
\[
\SYT(\nu) \,=\, \frac{|\nu|!}{\prod_{(i,j)\in \nu} h^*(i,j)}.
\]
\end{theorem}

\subsection{Flow polytopes}\label{Flow polytopes}

\begin{definition}
\label{def::flow_polytope}
Let $G$ be a connected directed acyclic graph with vertices $\{1,2,\ldots,{N+1}\}$ where $N+1$ is a sink and such that if $(i,j)\in E(G)$, then $i<j$ and let  ${\bf c}=(c_1,\ldots,c_{N},-\sum_i c_i)$ where $c_i \in \mathbb{Z}$. A \defn{{\em $\ccc$-flow}} on $G$ is a tuple $(f(e))_{e\in E(G)}$ of nonnegative real numbers for which each non-sink vertex has netflow $c_i$. That is, 
$$\sum_{e=(i,j)\in E(G)} f(e) \,-\, \sum_{e=(k,i)\in E(G)} f(e) = c_i$$
for $i=1,\ldots, N$. The netflow on the sink  $N+1$ is $-\sum_{i=1}^N c_i$.
The {\em flow polytope} $\mathcal{F}_G(\mathbf{c}) \subset \mathbb{R}^{|E(G)|}$ is defined as the set of $\mathbf{c}$-flows on $G$, and $\mathbf{c}$ is called the \defn{netflow vector}. 
\end{definition}

 The dimension of the flow polytope $\mathcal{F}_G(\ccc)$ is  $|E|-|V|+1$  \cite[\S 1.1]{Baldoni_Vergne_2008}, the first Betti number $\defn{\beta_1(G)}$ of $G$.  Given a graph $G$, a subgraph $H$ of $G$ is denoted by $H \subseteq  G$. 

Two lattice polytopes $P\subset\RR^m$ and $Q\subset\RR^n$ are \defn{integrally equivalent} if there exists an affine transformation $\Phi:\RR^m\rightarrow\RR^n$ whose restriction to $P$ is a bijection $\Phi:P\rightarrow Q$ that preserves the lattice. We denote this by $P\equiv Q$.

Given a graph $G$, let $G^r$ be the graph with the same vertices as $G$ and edges $E(G^r)=\{(i,j) \mid (N+2-j,N+2-i) \in E(G)\}$. That is, $G^r$ is the graph obtained by reversing the edges and relabelling the vertices $i\mapsto N+2-i$ and switching the netflow vector accordingly. By reversing the flows, one shows that the flow polytopes of $G$ and $G^r$ are integrally equivalent.

\begin{proposition}[{\cite[Prop. 2.3]{Mészáros_Morales_2018}}] \label{prop: symm fp reversing}
For a connected directed acyclic graph $G$ on the vertex set $[N+1]$ and ${\bf c} = (c_1,\ldots,c_N,-\sum_i c_i)$ in $\mathbb{Z}^{N+1}$, then
\[
\mathcal{F}_G\Bigl(c_1,\ldots,c_N,-\sum_i c_i\Bigr) \,\equiv\, \mathcal{F}_{G^r}\Bigl(\sum_i c_i,-c_N,\ldots,-c_1\Bigr).
\]
\end{proposition}

\subsection{Vertices of flow polytopes in terms of forests}\label{subsec:vertices FG}

Hille studied flow polytopes and their faces in \cite{Hille_2003} in the context of {\em quivers}. The following notion of valid flows\footnote{In Hille \cite{Hille_2003}, these are called {\em regular} flows.} is important to understand the faces of $\mathcal{F}_G({\bf c})$. 

\begin{definition}[${\bf c}$-valid flows] \label{def: a valid flows}
Let $G, {\bf c},$ and $\mathcal{F}_G(\textbf{c})$ be as in Definition \ref{def::flow_polytope}. A subgraph $H$ of $G$ is \defn{${\bf c}$-valid} if the edges of $H$ are the support of some ${\bf c}$-flow on $G$. 
\end{definition}

The following characterization of vertices of $\mathcal{F}_G({\bf c})$ appeared in earlier work of Gallo and Sodini \cite[Thm. 3.1]{GalloSodini}. A self-contained proof of this result is in Appendix \ref{sec:appendixA}.

\begin{theorem}[{\cite[Thm. 3.1]{GalloSodini}}] \label{char: vertices G as forests}
Vertices of $\mathcal{F}_{G}({\bf c})$ are in one-to-one correspondence with  ${\bf c}$-valid forests of $G$.
\end{theorem}

There is the following description of the faces of $\mathcal{F}_G(\textbf{c})$ due to Hille \cite{Hille_2003} for the case $c_i\geq 0$. 

\begin{theorem}[{Hille \cite[part of Theorem 3.2]{Hille_2003}}]
\label{thm:Hille}
Let $G, {\bf c},$ and $\mathcal{F}_G({\bf c})$ be as in Definition \ref{def::flow_polytope} with $c_i\geq 0$. Then for any integer $d \geq 0$, the $d$-dimensional faces of $\mathcal{F}_G({\bf c})$ are of the form $\mathcal{F}_H({\bf c})$ where $H$ is a ${\bf c}$-valid connected subgraph of $G$ such that the first Betti number $\beta_1(H) = d$, after forgetting the directions of edges of $H$. 
\end{theorem}

%--------------------------------------Section-------------------------------------%

\section{The $m$th Pitman-Stanley polytope as a flow polytope}\label{The generalized Pitman-Stanley polytope}

Recall the definition of the $m$th Pitman--Stanley polytope $\PS_n^m({\bf a},{\bf b})$. This generalization is mentioned in \cite[p. 24]{Pitman_Stanley_1999}\footnote{Note that in \cite[p. 24]{Pitman_Stanley_1999} there is a typo: the variables $v_{ij}$ there should be defined as $v_{ij} :=x_{1j}+x_{2j}+\ldots+x_{ij}$ in our notation. In addition, in our convention, we reordered the $v_{ij}$'s to make the construction of the flow polytopes more intuitive: rather than having $v_{i1}\leq v_{i2}\leq\cdots\leq v_{im}$ as in \cite{Pitman_Stanley_1999}, we have $v_{im}\leq \cdots\leq v_{i2} \leq v_{i1}$.}.  

Just as for the Pitman-Stanley polytope $\PS_n^1({\bf a},{\bf b})$, the lattice points of the $m$th Pitman-Stanley polytope $\PS_n^m({\bf a},{\bf b})$ can be understood as integral flows in a particular graph or as plane partitions. The underlying directed graph is the following grid graph with a sink.

\begin{definition} \label{PS_flow_grid_graph}
For positive integers $n$ and $m$, let \defn{$G(n,m)$} be a directed graph on the vertex set $V=\{(i, j)\,|\,1\leq i\leq n,\,\, 0\leq j\leq m\}\,\cup\, \{s\}$ and the following directed edges: 
 \begin{itemize}
     \item  $((i, j),(i,j+1))$ and $((i,j),(i+1,j))$ for $i\in\{1, 2, \dots, n-1\}$ and $j\in\{0, 1, \dots, m-1\}$,
     \item $((n, j),(n, j+1))$ and $((n,j),s)$ where $j\in\{1, 2, \dots, m-1\}$,
     \item $((i, m),(i+1, m))$ where $i\in\{1, 2, \dots, n-1\}$, and
     \item  $((n, m),s)$.
\end{itemize}
 \end{definition}

 \begin{definition}\label{PS_flow_def}
 For positive integers $n$ and $m$, vectors ${\bf a}= (a_1, a_2, \ldots, a_n)$ such that ${\bf a}\in\NN^n$ and ${\bf b} = (b_1, b_2, \ldots, b_n)$ such that $\bf b\in\NN^n$, and the graph $G(n,m)$ as defined above, denote by $\defn{\FF_{G(n,m)}({\bf a},{\bf b})}$ the flow polytope on the graph $G(n,m)$ with the following netflows:
\begin{itemize}
    \item vertices $\{(1, 0), (2, 0), (3, 0), \dots, (n, 0)\}$ have corresponding netflow $(a_1, a_2, \dots, a_n)$,
    \item vertices $\{(1, m), (2, m), (3, m), \dots, (n, m)\}$ have corresponding netflow $(-b_1, -b_2, \dots, -b_n)$,
    \item the sink vertex $s$ has netflow $-\sum_{i=1}^n a_i+\sum_{i=1}^n b_i$,
    \item and all other vertices $(i, j)$ for $i\in\{1, 2, \dots, n\}$ and $j\in\{1, 2, 3, \dots, m-1\}$ have netflow $0$.
\end{itemize}
 \end{definition}

 \begin{remark} \label{rem: dominance order a and b}
 Note that in order for the polytope  $\FF_{G(n,m)}(\aaa,\bb)$ to be nonempty, it is necessary that $\aaa\trianglerighteq \bb$ in dominance order. Moreover, the same holds for $\PS_n^m(\aaa,\bb)$. Also, note that in the literature of flow polytopes,  $\mathcal{F}_{G(n,m)}(\aaa,\bb)$ would be written as $\mathcal{F}_{G(n,m)}(\aaa,0^{n(m-1)},-\bb,-\sum_{i=1}^n (a_i-b_i))$. For simplicity, we adopt the former convention and we call a flow in this polytope an $(\aaa,\bb)$-flow.
 \end{remark}

In \cite[Example 16]{Baldoni_Vergne_2008}, Baldoni and Vergne showed that the original Pitman--Stanley polytope $\PS_n^1({\bf a})$ is integrally equivalent to a flow polytope of a graph $G'(n,1)$. Next, we show that the $m$th Pitman--Stanley polytope $\PS_n^m(\textbf{a},\textbf{b})$ is also integrally equivalent to the flow polytope $\mathcal{F}_{G(n, m)}(\bf a, \bf b)$. Our construction recovers that of Baldoni and Vergne when $m=1$ and ${\bf b}={\bf 0}$. Moreover, $\PS_n^m(\textbf{a},\textbf{0})$ is integrally equivalent to $\mathcal{F}_{G'(n, m)}(\bf a, \mathbf{0})$ where $G'(n,m)$ is a generalization of the graph used by Baldoni and Vergne (see Figure~\ref{fig:flowGprime34} for the generalization and Figure~\ref{fig:comparing flows PS} for their original graph).

\gPSisflow

\begin{example} The graph $G(3,4)$ for the flow polytope $\PS_3^4(\aaa, \bb) \equiv \mathcal{F}_{G(3, 4)}(\bf a,\bf b)$ is illustrated in Figure~\ref{fig:flowG34skew}.

\begin{figure}[h!]
      \centering
  \begin{subfigure}[b]{0.25\textwidth}
  \includegraphics[scale=0.75]{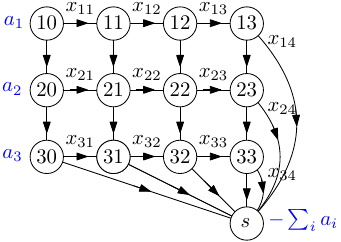}
   \caption{}
 \label{fig:flowGprime34}
  \end{subfigure}
  \quad 
\begin{subfigure}[b]{0.35\textwidth}
\includegraphics[scale=0.75]{grid_graph_G34_blob_labels.pdf}
\caption{}
\label{fig:flowG34skew}
\end{subfigure}
 \begin{subfigure}[b]{0.35\textwidth}
 \includegraphics[scale=0.75]{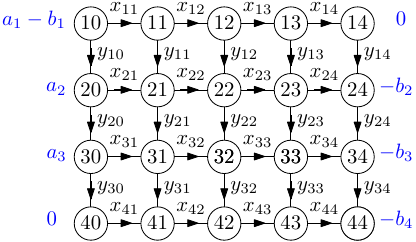}
 \caption{}
 \label{fig:flowH34}
 \end{subfigure}
 \caption{Illustration of the graphs $G'(3,4)$, $G(3,4)$, and $H(3,4)$ for $\mathcal{F}_{G(3,4)}({\bf a},{\bf b})$, $\mathcal{F}_{G(3,4)}({\bf a})$ and $\mathcal{F}_{H(3,4)}({\bf a'},{\bf b'})$, respectively for the graph $H(m,n)$ from Definition~\ref{def:hnm}. The netflow of each internal vertex  is zero.}
 \label{fig:flowG34}
\end{figure}

\end{example}

\begin{proof}[Proof of Theorem~\ref{thm: gPS is flow polytope}]
Let $\Phi: \FF_{G(n,m)}({\bf a},{\bf b}) \to \PS_n^m({\bf a},{\bf b})$ be defined as follows: $f((i,j-1),(i,j)) \mapsto x_{i,j}$ for  $i \in \{1,2,\ldots,n\}$ and $j \in \{1,2,\ldots,m\}$,
that is, $x_{i,j}$ is the flow on the horizontal edge $((i,j-1),(i,j))$ of $G(n,m)$. We denote by $y_{i,j}$ the flow on the vertical edge $((i,j),(i+1,j))$ for $i\in \{1,2,\ldots,n-1\}$ and $j\in \{0,1,\ldots,m\}$. See Figure~\ref{fig:flowG34}. We show $\Phi$ gives the desired integral equivalence.

For the first column $j=0$ of vertices in $G(n,m)$, for vertex $(1,0)$ we have the netflow equations $x_{1,1}+y_{1,0}=a_1$ and thus $x_{1,1} \leq a_1$. For $i=2,\ldots,n$, by adding the netflow relations for vertex $(1,0)$ with the netflow relation $x_{j,1}+y_{j,0}-y_{j-1,0}=a_j$ of vertex $(j,0)$ for $j=2\ldots,i-1$ we obtain that 
\[
(x_{1,1}+\cdots + x_{i-1,1}) +  y_{i-1,0}=  a_1+\dots+a_{i-1}.
\]
Since for vertex $(i,0)$ we have that $a_{i} + y_{i-1,0} = x_{i, 1} + y_{i, 0}$, then we conclude that 
\begin{equation} \label{eq: dominance first col}
x_{1,1}+\cdots + x_{i,1} \leq a_1+\cdots + a_{i}.
\end{equation}

Next, if we fix the flows $x_{1,1},\ldots,x_{1,n}$ on the edges $((1,0),(1,1)),\ldots,((n,0),(n,1))$, we project to a flow polytope $\mathcal{F}_{G(n,m-1)}({\bf x}_1,{\bf b})$, where ${\bf x}_1 = (x_{1,1},\ldots,x_{n,1})$. Iterating the argument above gives that 
\begin{equation} \label{eq: middle ones}
x_{1,m}+\cdots + x_{i,m} \leq \cdots \leq x_{1,2}+\cdots + x_{i,2}\leq x_{1,1}+\cdots + x_{i,1}.
\end{equation}
for $i=1,\ldots,n$. For the last column $j=m$, an analogous argument to the one of the first column shows that  
\begin{equation} \label{eq: last column}
b_1+\cdots +  b_i \leq x_{1,m}+\cdots + x_{i,m}. 
\end{equation}
for $i=1,\ldots,n$. Combining \eqref{eq: dominance first col},\eqref{eq: middle ones}, and \eqref{eq: last column} shows that $\Phi$ is well-defined. 

The affine map $\Phi$ is a bijection since the flows $y_{i,j}$ on the vertical edges and on the edges $((n,j),s)$  are determined by the flows $x_{i,j}$ on the horizontal edges. Lastly, since $\Phi$ restricted to the flows $x_{i,j}$ on the horizontal edges is the identity map, it preserves the lattice. 
\end{proof}

\begin{remark} \label{rem:sums of vertical edges}
Using the netflow relations, one can check that the sum of the flows of the vertical edges in the $i$th row of $G(n,m)$ for $i=1,\ldots,n$ is the following:
\begin{equation} \label{eq: sum of vertical edges}
y_{i,0}+y_{i,1}+\cdots + y_{i,m} = (a_1+\cdots+a_i)-(b_1+\cdots+b_i).
\end{equation}
\end{remark}

\begin{remark}
Since the graph $G(n,m)$ is planar, by the work of Liu--M\'esz\'aros--St. Dizier \cite{LMStD1}, the flow polytope $\mathcal{F}_{G(n,m)}(\aaa,\bb)$ is integrally equivalent to a {\em marked order polytope} of a poset consiting of a product of chains. This will be further explored in \cite{PSvolume}.    
\end{remark}
 
From now on, we use $\mathbf{x}$  to refer to the flows on the horizontal  edges of $G(n,m)$, which determine the flows on the rest of the graph. Recall that $\theta(\aaa,\bb):=\lambda(\aaa)/\mu(\bb)$ where $\lambda=\lambda(\aaa):=(a_1+a_2+\ldots+a_n, \ldots, a_1+a_2, a_1)$ and $\mu=\mu(\bb):=(b_1+b_2+\ldots+b_n, \ldots, b_1+b_2, b_1)$, and that $\lambda'$ denotes the conjugate partition of $\lambda$.
 
\begin{theorem}\label{thm:bijflowpp}
There is a bijection $\Psi$
 between plane partitions of shape $\theta(\aaa,\bb)$ with entries at most $m$ and integral flows for $\mathcal{F}_{G(n,m)}(\bf a, \bf b)$.
 \end{theorem}

\begin{proof}
Let $\pi=(\pi_{ij})$ be a plane partition of shape $\theta(\aaa,\bb)$ with entries at most $m$. Consider a column $j=1,\ldots, \lambda_1$ with (non-empty) entries  $\pi_{\mu'_j+1,j}, \pi_{\mu'_j+2,j}, \ldots, \pi_{\lambda'_j,j}$. (In the case when $\bb=\mathbf{0}$, then $\mu'_j=0$ for every column $j$.) This corresponds to the \defn{trajectory} of one unit of flow in the graph $G(n,m)$ that starts at vertex $(n+1-\lambda'_j,0)$, goes to $(n+1-\lambda'_j,\pi_{\lambda'_j,j})$, then to $(n+1-\lambda'_j+1,\pi_{\lambda'_j,j})$, then to $(n+1-\lambda'_j+1,\pi_{\lambda'_j-1,j})$, then to $(n+1-\lambda'_j+2,\pi_{\lambda'_j-1,j})$, and so on. At the end, two cases happen: if $\mu'_j=0$, it keeps going until it reaches $(n,\pi_{1,j})$ from where it goes to the sink $s$, or if $\mu'_j\geq 1$, until it reaches $(n+1-\mu'_j,\pi_{\mu'_j+1,j})$ from where it goes to the vertex $(n+1-\mu'_j,m)$ with negative netflow $-b_{n+1-\mu'_j}$. In other words, the entries of a column of $\pi$ give the columns of $G(n,m)$ where a unit of flow starting at  vertex $(n+1-\lambda'_j,0)$ goes down in different rows. Let $\Psi(\pi):=f$ be the integral flow described above.

  \begin{figure}[h!]
  \centering
\ytableausetup{mathmode, boxsize=1.5em}
\begin{ytableau}
\none & \none & \none & \textcolor{green}{6} & \textcolor{orange}{4} & \textcolor{Purple}{1} & \textcolor{Turquoise}{0} \\
\none & \textcolor{blue}{7} & \textcolor{yellow}{6} & \textcolor{green}{2} & \textcolor{orange}{0} \\
\textcolor{red}{4} & \textcolor{blue}{1} \\
\end{ytableau}\quad $\leftrightarrow$\quad  \raisebox{-0.6\height}{\includegraphics[scale=0.4]{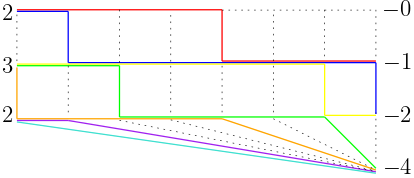}}
    \caption{Constructing an integral flow on $G(n,m)$ from a plane partition for $n=3$, $m=7$, $\aaa=(2,3,2)$ and $\bb=(0,1,2)$.}\label{fig:examplebijection1}
    \end{figure}

For example, consider the skew plane partition of shape $(7,5,2)/(3,1,0)$ in Figure~\ref{fig:examplebijection1}. For the second column, we have $\mu'_2=1$ and $\lambda'_2=3$, so we have entries $\pi_{2,2}=7$ and $\pi_{3,2}=1$. This gives us the trajectory of a unit flow that starts at vertex $(3+1-3,0)=(1,0)$ of $G(n,m)$, keeps going until $(1,1)$, and then goes down to $(2,1)$ where it keeps going until $(2,7)$ and goes down to $(3,7)$. Since $\mu'_2\neq 0$, it stops there at this vertex with negative netflow.  For the fourth column, we have $\mu'_4=0$ and $\lambda'_4=2$ with entries $\pi_{1,4}=6$ and $\pi_{2,4}=0$. This gives us the trajectory of a unit flow that starts at vertex $(2,0)$ of $G(n,m)$, keeps going until $(2,2)$, goes down to $(3,2)$, runs until $(3,6)$, and since $\mu'_4=0$, goes down to $s$ from there. We can forget about the colors of the different unit flows, see Figure \ref{fig:examplebijection2}. 

\begin{figure}[h!]
\centering
    \includegraphics[scale=0.4]{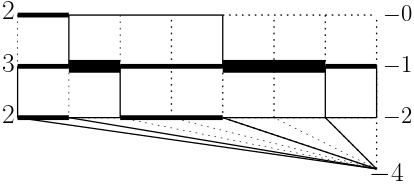}
    \caption{Superposing all units of flow, where the thickest solid lines correspond to a flow of 3 units, the fat ones correspond to a flow of 2 units, and where the thinner solid lines correspond to a flow of 1 unit.}\label{fig:examplebijection2}
\end{figure}

One key observation is that our unit flows, i.e., our colorful trajectories in Figure~\ref{fig:examplebijection1}, are noncrossing. This tells us how to construct a plane partition given some integral flow in $G(n,m)$, thus describing $\Psi^{-1}$. We build our  plane partition column by column (starting with the leftmost), and each column of the plane partition will be built from the bottom up. While there is still some flow, pick a potential trajectory for one unit of flow that starts as high as possible in $G(n,m)$, and amongst those possible trajectories, pick a trajectory that goes as far right as possible in its current row before moving down to the next row. Record the columns in which the trajectory goes down within the plane partition (from bottom to top). If by traveling as far as possible in its current row it reaches a vertex with negative netflow, record an empty entry with an $\times$ instead for that row as well as for any row below in $G(n,m)$. Remove the flow of this trajectory from the flow you started with and update the netflows---note that this is still an integral flow. Keep going until there is no flow left. See Figure~\ref{fig:examplebijection3} for an example.

\begin{figure}[h!]
\[
\begin{array}{llllllll}
    \includegraphics[scale=0.25]{aawmskewbijectionpartitionflow3.png} && \includegraphics[scale=0.25]{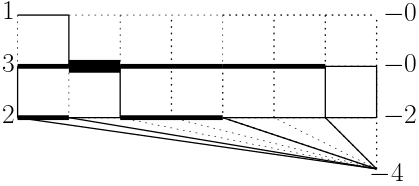} && \includegraphics[scale=0.25]{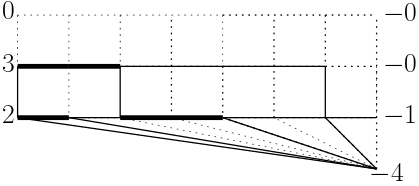}  && \includegraphics[scale=0.25]{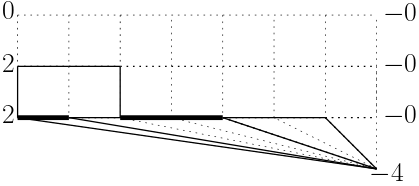}\\
    && \begin{ytableau}
\times  \\
\times \\
4  
\end{ytableau} &&    \begin{ytableau}
\times & \times \\
\times & 7 \\
4 & 1 
\end{ytableau} &&    \begin{ytableau}
\times & \times & \times\\
\times & 7 & 6  \\
4 & 1 
\end{ytableau}
\end{array}
\]

\[  
\begin{array}{lllllll}
        \includegraphics[scale=0.25]{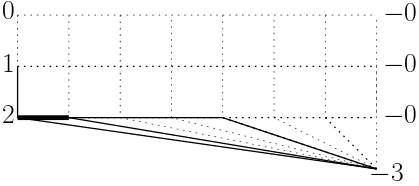} && \includegraphics[scale=0.25]{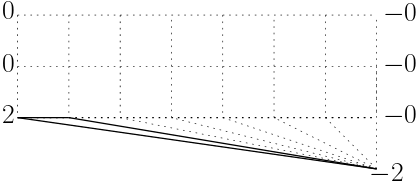} && 
  \includegraphics[scale=0.25]{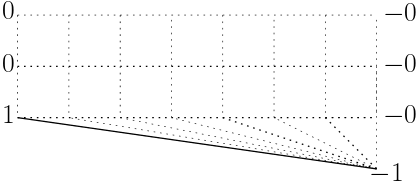} &&
    \includegraphics[scale=0.25]{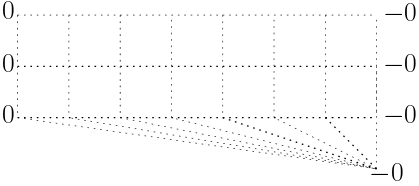}   \\
   \begin{ytableau}
\times & \times & \times & 6\\
\times & 7 & 6 & 2 \\
4 & 1 
\end{ytableau} &&  \begin{ytableau}
\times & \times & \times & 6 & 4\\
\times & 7 & 6 & 2 & 0 \\
4 & 1 
\end{ytableau} &&   \begin{ytableau}
\times & \times & \times & 6 & 4 & 1\\
\times & 7 & 6 & 2 & 0 \\
4 & 1 
\end{ytableau} &&    \begin{ytableau}
\times & \times & \times & 6 & 4 & 1 & 0\\
\times & 7 & 6 & 2 & 0 \\
4 & 1 
\end{ytableau}
\end{array}
\]
\caption{Constructing a plane partition from an integral flow on $G(n,m)$.}\label{fig:examplebijection3}
\end{figure}

Following the procedure outlined above, we obtain a filling of a diagram of shape $\theta(\aaa,\bb)$ where each entry is at most $m$. To show that it is a plane partition, we need to check that the entries are weakly decreasing as we go right or down. The entries are weakly decreasing as we go right in the diagram as we keep sending our unit flows as right as possible before moving down to the next row in the graph. To see that the entries are weakly decreasing as we go down in the diagram, recall that the columns record where our unit flow goes down in the different rows, with the last row of the diagram corresponding to the first row of the graph and vice-versa. The result is obtained since all arcs of $G(n,m)$ are directed to go right or down. Therefore, the diagram yielded by this algorithm is in fact a plane partition and thus $\Psi$ is the desired bijection.
\end{proof}

The tuple of non-intersecting unit flows in $G(n,m)$ obtained from the procedure in the proof above is called the \defn{trajectory decomposition} of an integer $(\aaa,\bb)$-flow. See Figure~\ref{fig:examplebijection1} for an example.

\begin{remark}
Alternatively, there is an equivalent interpretation of $\Psi$ using the rows of the plane partition. In this reformulation, given a plane partition of shape $\theta(\aaa,\bb)$, we obtain a lattice point $\xx$ of $\mathcal{F}_{G(n,m)}(\aaa,\bb)$ where $x_{1j}+\cdots + x_{ij}$ is the number of elements greater than or equal to $j$ in row $n-i+1$. 
Note that downward flows in  this lattice point of $\mathcal{F}_{G(n,m)}(\aaa,\bb)$ are dependent on horizontal flows. Thus, we can recursively fill in these downward flows to construct the complete flow through this graph. See Figure~\ref{fig: other bijection pp to int flows} for an example. This bijection for $m=1$ appeared in \cite[Thm. 12]{Pitman_Stanley_1999}. 

Whereas in the description of $\Psi$ in the proof above, we decomposed the integral flow into noncrossing trajectories, here we decompose the plane partition into an $(m+2)$-chain of nested partitions starting from $\mu$ to $\lambda$.

\begin{figure}[h!]
\centering
\begin{minipage}{6in}
    \ytableausetup{mathmode, boxsize=1.5em}  
  \begin{ytableau}
        \times & \times & \times & 6 & 4 & 1 & 0 \\
        \times & 7 & 6 & 2 & 0   \\
        4 & 1 
    \end{ytableau} \quad $\leftrightarrow$\quad \raisebox{-0.6\height}{\includegraphics[scale=0.5]{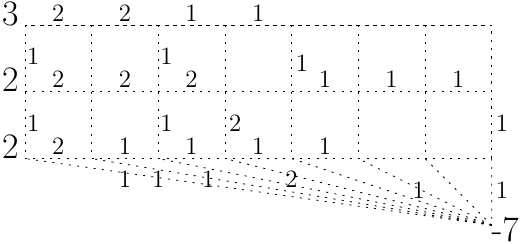}} $\equiv$ \,\,\raisebox{-0.6\height}{
\includegraphics[scale=0.3]{aawmskewbijectionpartitionflow3.png}}
\end{minipage}
\caption{Alternative description of bijection between plane partitions of shape $\theta(\aaa,\bb)$ and integral flows in $\mathcal{F}_{G(n,m)}(\aaa,\bb)$: $x_{1j}+\cdots+x_{ij}$ gives the number of elements greater than or equal to $j$ in row $n-i+1$ of the plane partition.}
\label{fig: other bijection pp to int flows}
\end{figure}
\end{remark}

%--------------------------------------Section-------------------------------------%
\section{Characterization of vertices}
\label{sec: char vertices}

\subsection{Merging flow characterization}

We  use Theorem~\ref{char: vertices G as forests} to give a local criterion (at each vertex of $G$) for whether or not a flow is a vertex of the corresponding flow polytope, as follows.

\begin{definition} \label{def:split and merge}
For an integer $({\bf a}, {\bf b})$-flow on $G(n,m)$, the trajectories of two units of flow \defn{split} (resp. \defn{merge}) if
\begin{itemize}
    \item at a vertex where the two trajectories through which both trajectories go through, they leave (resp. enter) through different edges,
    \item or if at a vertex with negative (resp. positive) netflow through which both trajectories go through, exactly one leaves (resp. enters) through an edge.
\end{itemize} 
An integer $(\aaa,\bb)$-flow ${\bf x}$ in $\mathcal{F}_{G(n,m)}(\aaa,\bb)$ is an \defn{unsplittable flow} if no two of its trajectories split. 
\end{definition}

\begin{example} \label{ex: unsplittable flow grid}
Figure~\ref{fig:bijection vertices straight} has an example of an unsplittable flow of the flow polytope $\mathcal{F}_{G(5,6)}({\bf 1})$. In Figure~\ref{fig:exampleconjectureskewvertices_a} the top two (red and blue) units of flow split, and the second and third (blue and green) units of flow merge.
\end{example}

From the proof of Theorem \ref{thm:bijflowpp}, recall that any integral flow for $\mathcal{F}_{G(n,m)}(\aaa,\bb)$ can be seen as noncrossing trajectories of units of flow. Through this interpretation, we see that the $i$th trajectory from the top ends at the same vertex of $G(n,m)$ in any integral flow for $\mathcal{F}_{G(n,m)}(\aaa,\bb)$. Note that this does not depend on $m$. See Figure~\ref{fig:schematictrajectories}. The next result, which follows from Theorem~\ref{thm:Hille}, gives a characterization of the vertices of  $\mathcal{F}_{G(n,m)}(\aaa,\bb)$ in terms of splits and merges. See Figure~\ref{fig:exampleconjectureskewvertices}  for  examples.

\begin{figure}[h!]
    \centering
    \includegraphics[scale=0.5]{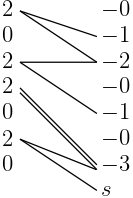}
    \caption{Starting and ending vertices for all trajectories for any flow in $\mathcal{F}_{G(5,m)}((2,0,2,2,0,2,0), (0,1,2,0,1,0,3))$ for any $m$.}
    \label{fig:schematictrajectories}
\end{figure}

\begin{corollary} \label{conj:vert flows skew}
Consider an integral flow for $\mathcal{F}_{G(n,m)}(\aaa,\bb)$ with $\aaa,\bb\in \NN^n$. That flow corresponds to a vertex of $\mathcal{F}_{G(n,m)}(\aaa,\bb)$ if and only if the following conditions are satisfied.
\begin{itemize}
    \item[(i)] The trajectories of any two units of flow starting at the same vertex with positive netflow do not merge. 
    \item[(ii)] The trajectories of any two units of flow ending at the same vertex with negative netflow  do not split. 
    \item[(iii)] The trajectories of any two units of flow starting and ending at different vertices can split at most once and merge at most once.
\end{itemize}
In particular, if $\bb={\bf 0}$, flows corresponding to vertices of $\mathcal{F}_{G(n,m)}(\aaa)$ are unsplittable flows.
\end{corollary}

\begin{proof}
First note that if any of (i), (ii) or (iii) does not hold for some integral flow for $\mathcal{F}_{G(n,m)}(\aaa,\bb)$, then the support of this flow contains a cycle. By Theorem~\ref{char: vertices G as forests}, this implies that this flow does not correspond to a vertex of $\mathcal{F}_{G(n,m)}(\aaa,\bb)$. Contrapositively, this means that if a flow is a vertex, then (i), (ii), (iii) must all hold. 

For the converse, suppose that (i), (ii), (iii) all hold for some integral flow and that the support of this flow contains a cycle between the trajectories of two units of flow. Note that any cycle contains both a merge and a split. Therefore, if the trajectories of those two units started or ended at the same vertex of $G(n,m)$, we would have a contradiction to (i) and (ii) respectively. Finally, note that if the two units started and ended at different vertices, there must be at least one merge before the cycle and a split after the cycle, which means that in total, there are at least two merges and two splits, a contradiction to (iii). Therefore, if (i), (ii), (iii) hold for some integral flow, then the support of this flow contains no cycles, and by Theorem~\ref{char: vertices G as forests}, it corresponds to a vertex.
\end{proof}

\begin{figure}[h!]
      \centering
  \begin{subfigure}[b]{0.24\textwidth}
\includegraphics[scale=0.75]{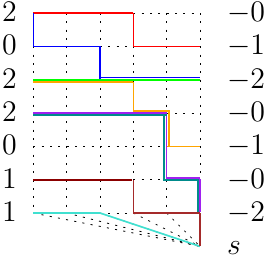}
\caption{}
\label{fig:exampleconjectureskewvertices_a}
\end{subfigure}
  \begin{subfigure}[b]{0.24\textwidth}
\includegraphics[scale=0.75]{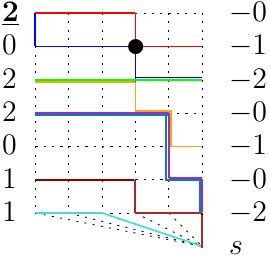}
\caption{}
\label{fig:exampleconjectureskewvertices_b}
  \end{subfigure}
  \begin{subfigure}[b]{0.24\textwidth}
\includegraphics[scale=0.75]{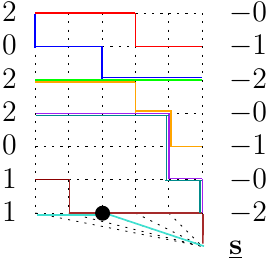}
\caption{}
\label{fig:exampleconjectureskewvertices_c}
  \end{subfigure}
  \begin{subfigure}[b]{0.24\textwidth}
\includegraphics[scale=0.75]{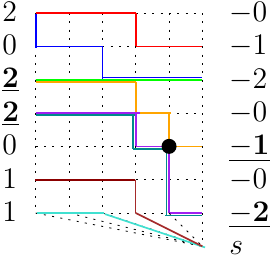}
\caption{}
\label{fig:exampleconjectureskewvertices_d}
  \end{subfigure}
\caption{According to Corollary~\ref{conj:vert flows skew}, the flow on the left is a vertex, whereas the other three are not. The second has a vertex with a positive netflow where two trajectories start that merge. The third has a vertex with negative netflow where two trajectories end that split. The last flow has two trajectories that start and end at different vertices that split and merge twice.}
    \label{fig:exampleconjectureskewvertices}
\end{figure}

\subsection{Plane partition characterization} 

\begin{definition} \label{def:vertex plane partitions general}
Let $n$ and $m$ be nonnegative integers, and $\aaa,\bb\in \NN^n$. A plane partition $\pi=(\pi_{ij})$ of shape $\theta(\aaa,\bb)=\lambda/\mu$ with entries at most $m$ is a \defn{vertex plane partition} if the following conditions hold: 
\begin{itemize}
    \item[(i)] \textbf{For columns $j$ and $j+1$ such that $\mu'_j = \mu'_{j+1}$ and $\lambda'_j=\lambda'_{j+1}$:} $\pi_{i,j}=\pi_{i,j+1}$ for every $i$.
    \item[(ii)] \textbf{For columns $j$ and $j+1$ such that $\mu'_j = \mu'_{j+1}$ and $\lambda'_j> \lambda'_{j+1}$:} If  $\pi_{i,j}=\pi_{i,j+1}$ for some $\mu'_{j+1}+2\leq i \leq \lambda'_{j+1}$, then $\pi_{i-1,j}=\pi_{i-1,j+1}$. Furthermore, for every $\mu'_{j+1}+1\leq i\leq \lambda'_{j+1}$, either $\pi_{i,j+1} <\pi_{i+1,j}$ or $\pi_{i,j+1} = \pi_{i,j}$.
    \item[(iii)] \textbf{For columns $j$ and $j+1$ such that $\lambda'_j=\lambda'_{j+1}$ and $\mu'_j > \mu'_{j+1}$:} If  $\pi_{i,j}=\pi_{i,j+1}$ for some $\mu'_j+1 \leq i \leq \lambda'_j -1$, then $\pi_{i+1,j}=\pi_{i+1,j+1}$. Furthermore, for every $\mu'_j \leq i \leq \lambda'_j-1$ either $\pi_{i,j+1} <\pi_{i+1,j}$ or $\pi_{i+1,j}=\pi_{i+1,j+1}$.
    \item[(iv)] \textbf{For columns $j$ and $j+1$ such that  $\mu'_j \neq \mu'_{j+1}$ and $\lambda'_j\neq\lambda'_{j+1}$:} Let $i_{\min}$ be the smallest row index such that $\pi_{i_{\min},j}\leq \pi_{i_{\min}-1, j+1}$ and let $i_{\max}$ be the biggest row index such that $\pi_{i_{\max}+1,j} \leq \pi_{i_{\max}, j+1}$. If $i_{\min}<i_{\max}$ both exist, then $\pi_{i,j}=\pi_{i,j+1}$ for every row $i_{\min} \leq i \leq i_{\max}$.

\end{itemize}
\end{definition}

\begin{remark} \label{rem: equal columns vertex pp}
Note that in the definition above, if $\mu'_j = \mu'_{j+1}$ and $\lambda'_j=\lambda'_{j+1}$, then the columns $j$ and $j+1$ in the vertex plane partition $\pi$ are equal. In particular, this implies that when ${\bf b}={\bf 0}$, it suffices to consider ${\bf a}\in \{0,1\}^n$.
\end{remark}

\begin{theorem}\label{thm:vertexplanepartitions}
Let $n$ and $m$ be nonnegative integers, then the vertices of $\mathcal{F}_{G(n,m)}(\aaa,\bb)$ are in correspondence with the vertex plane partitions of shape  $\theta(\aaa,\bb)=\lambda/\mu$.
\end{theorem}

\begin{proof}
By Theorems~\ref{thm: gPS is flow polytope} and \ref{thm:bijflowpp}, any lattice point $\mathbf{x}$ in $\mathcal{F}_{G(n,m)}(\aaa,\bb)$ corresponds to both an integral flow $f$ in $G(n,m)$, and a plane partition $\pi:=\Psi^{-1}(f)$ of shape $\lambda/\mu$.
As was explained in the proof of Theorem \ref{thm:bijflowpp}, the flow $f$ can be decomposed into noncrossing trajectories. 

Order the units of flow of $f$ so that the trajectory of unit $u$ never goes below the trajectory of unit $u+1$. As before, let $\theta(\aaa,\bb)=\lambda/\mu$. Thus, note that the trajectory for unit flow $u$ starts at vertex $(\lambda'_1-\lambda'_u+1,0)$ with positive netflow and ends at $(\lambda'_1-\mu'_u+1, m)$---here, we think of sink $s$ as having coordinate $(\lambda'_1+1, m)$. Let $d_{ur}$ be the index of the column in $G(n,m)$ where the unit of flow $u$ descends from row $r$ to row $r+1$ for $\lambda'_1-\lambda'_u+1 \leq r \leq \lambda'_1-\mu'_u$. Observe that $d_{ur}$ is not defined for $r \geq \lambda'_1-\mu'_u+1$ as the trajectory for $u$ ends in row $\lambda'_1-\mu'_u+1$, i.e., it does not descend into the next row. Note that we can  write $\pi$ in terms of the $d_{ur}$'s.

Because the trajectories of the unit flows are noncrossing, we already know that $d_{ur}\geq d_{u+1,r}$ if these both exist. It is possible for neither to exist and if only one exists, it has to be $d_{u+1,r}$. Furthermore, $d_{u,r}\geq d_{u,r+1}$ for any $u$ if both exist. It is possible for neither to exist, and if only one exists, it must be $d_{u,r}$. We look at possible configurations for the following entries of the plane partition where $u$ comes before $u'$ in the order of unit flows; note that the columns for $u$ and $u'$ need not be adjacent in the plane partition.

\begin{center}
\includegraphics[scale=0.5]{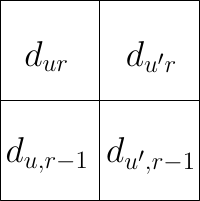}
\end{center}

In what follows, we sometimes denote $d_{u,r-1}, d_{u',r-1}, d_{ur}, d_{u',r}$ as the SW, SE, NW, NE entries of this configuration, respectively.

For a particular row $r$, there are five different scenarios that can arise depending on how the trajectories enter and leave row $r$. Note that $\times$ represents a cell outside of the shape $\lambda/\mu$.

\begin{enumerate}

\item If trajectories $u<u'$ both enter and leave row $r$ together, the possible configurations are

\begin{align*}
    \ytableausetup{smalltableaux}   
&    \begin{ytableau}
        b & b \\
        a & a  
    \end{ytableau}  & 
    &\begin{ytableau}
        b & b \\
        \times & \times  
    \end{ytableau} & 
    &\begin{ytableau}
        \times & \times \\
        a & a  
    \end{ytableau} & 
    &\begin{ytableau}
        \times & \times \\
        \times & \times  
    \end{ytableau} 
\end{align*}
where $a\leq b$.

\item If trajectories $u<u'$ enter separately but leave row $r$ together, i.e., if there is a merge happening in row $r$, the possible configurations are

\begin{align*}
    \ytableausetup{smalltableaux}   
&    \begin{ytableau}
        c & c \\
        b & a  
    \end{ytableau}  & 
    &\begin{ytableau}
        c & c \\
        b & \times  
    \end{ytableau} & 
    &\begin{ytableau}
        \times & \times \\
        b & a  
    \end{ytableau} & 
    &\begin{ytableau}
        \times & \times \\
        b & \times  
    \end{ytableau} 
\end{align*}
where $a< b \leq c$.

\item If trajectories $u<u'$ leave separately but enter row $r$ together, i.e., if there is a split happening in row $r$, the possible configurations are

\begin{align*}
    \ytableausetup{smalltableaux}   
&    \begin{ytableau}
        c & b \\
        a & a  
    \end{ytableau}  & 
    &\begin{ytableau}
        \times & b \\
        a & a  
    \end{ytableau} & 
    &\begin{ytableau}
        c & b \\
        \times & \times  
    \end{ytableau} & 
    &\begin{ytableau}
        \times & b \\
        \times & \times  
    \end{ytableau} 
\end{align*}
where $a\leq b <  c$.

\item If trajectories $u<u'$ enter and leave separately row $r$, but still touch somewhere within that row, i.e., if there is a merge and then a split happening in row $r$, the possible configurations are

\begin{align*}
    \ytableausetup{smalltableaux}   
&    \begin{ytableau}
        d & c \\
        b & a  
    \end{ytableau}  & 
    &\begin{ytableau}
        d & c \\
        b & \times  
    \end{ytableau} & 
    &\begin{ytableau}
        \times & c \\
        b & a  
    \end{ytableau} & 
    &\begin{ytableau}
        \times & c \\
        b & \times  
    \end{ytableau} 
\end{align*}
where $a< b \leq  c <d$.

\item Finally, if trajectories $u<u'$ enter and leave separately row $r$, and never touch each other within that row, the possible configurations are

\begin{align*}
    \ytableausetup{smalltableaux}   
&    \begin{ytableau}
        d & b \\
        c & a  
    \end{ytableau}  & 
    &\begin{ytableau}
        d & b \\
        c & \times  
    \end{ytableau} & 
    &\begin{ytableau}
        \times & b \\
        c & a  
    \end{ytableau} & 
    &\begin{ytableau}
        \times & b \\
        c & \times  
    \end{ytableau} 
\end{align*}
where $a\leq b <  c \leq d$.
\end{enumerate}

We use the considerations above and the conditions listed in Corollary~\ref{conj:vert flows skew} characterizing integral flows that are vertices of $\mathcal{F}_{G(n,m)}(\aaa,\bb)$ to characterize plane partitions that are vertices. 

\begin{itemize}
\item The trajectories of any two units of flow starting and ending at the same vertex must be the same.

In that case, we can only have scenario 1. In the plane partition, we see that any two columns that are between the same outer and inner corners of $\lambda/\mu$ must be equal. 

\item The trajectories of any two units of flow starting at the same vertex with positive netflow do not merge. 

We only need to consider the case when the two trajectories do not end at the same vertex with negative netflow. Starting from the bottom of the two corresponding columns in the plane partition, this means that we may start with some rows as in scenario 1 and then encounter a row from scenario 4 followed by rows from scenario 5. In particular, this means that for any two columns that are between the same outer corners of $\lambda/\mu$, if we have equality in a particular row, then we must have equality in each row below. Furthermore, if a SW entry is smaller or equal to a NE entry, then the SW and SE entries are equal. 

\item The trajectories of any two units of flow ending at the same vertex with negative netflow do not split. 

We only need to consider the case when the two trajectories do not start at the same vertex with positive netflow. Starting from the bottom of the two corresponding columns in the plane partition, this means that we may encounter some rows from scenario 5, then a row from scenario 2, and then rows from scenario 1. In particular, this means that for any two columns that are between the same inner corners of $\lambda/\mu$, if we have equality in a particular row, then we must have equality in each row above. Furthermore, if a SW entry is smaller or equal to a NE entry, then the NW and NE entries are equal. 

\item The trajectories of any two units of flow starting and ending at different vertices can merge and split at most once. 

There are two possibilities here: we can either have a merge followed by a split happening in different rows or both could happen within the same row. That is, from the bottom, we first encounter some rows from scenario 5, then one row from scenario 2, then some rows from scenario 1, and then  one row from scenario 3, and then some rows from scenario 5. Otherwise, we could also start with some rows from scenario 5, and then encounter one row from scenario 4 and then some rows from scenario 5. In particular, this means that for any two columns that are not between the same outer corners and not between the same inner corners of $\lambda/\mu$, if you look at the lowest SW entry that is smaller or equal to a NE entry and at the highest SW entry that is smaller or equal to a NE entry, every row from the one containing the lowest NE entry to the row containing the highest SW entry will be such that both entries in each of those rows are equal.

\end{itemize}

These are exactly the conditions of Definition~\ref{def:vertex plane partitions general}. Thus $\pi=\Psi^{-1}(f)$ is a vertex plane partition if and only if $f$ is a vertex of $\mathcal{F}_{G(n,m)}(\aaa,\bb)$.

\end{proof}

\begin{figure}
    \centering

  \begin{subfigure}[b]{0.6\textwidth}  
    \includegraphics[scale=0.7]{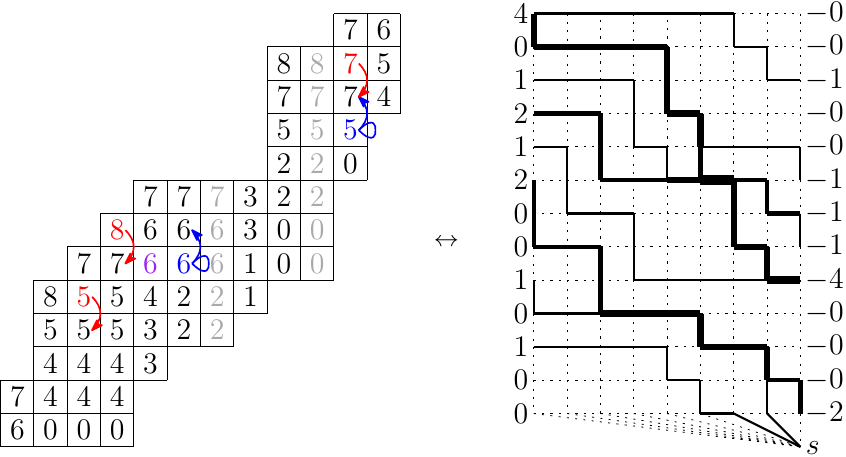}
    \caption{}
\label{fig:bijection vertices skew}
\end{subfigure}
 \begin{subfigure}[b]{0.4\textwidth}     
\hspace{1.02cm} \includegraphics[scale=0.7]{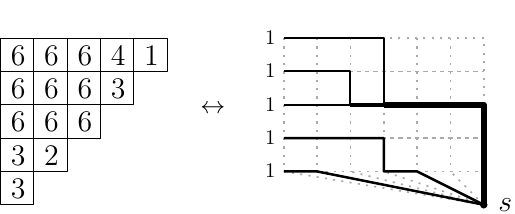}
  \caption{}
  \label{fig:bijection vertices straight}
  
 \end{subfigure}

    \caption{Examples of a  vertex plane partition of (a) skew shape and (b) straight shape and their corresponding flows that are vertices of $\mathcal{F}_{G(n,m)}(\aaa,\bb)$ and $\mathcal{F}_{G(n,m)}(\aaa)$, respectively. In a column $u$ such that $\lambda'_{u-1}=\lambda'_u$, the highest entry that is greater or equal to the entry southwest of it in column $u-1$ is colored in red. In a column $u$ such that $\mu'_{u-1}=\mu'_u$, the lowest entry that is greater or equal to the entry southwest of it in column $u-1$ is colored in blue. Entries in column $u$ that are strictly below a red entry and above a blue entry or the blue entry itself must be equal to the entries directly to their left in column $u-1$.  An entry is colored purple if it is both blue and red; in that case, no equality is forced between entries in columns $u-1$ and $u$. Entries in gray are in columns that must fully be equal to the columns to their left.}
    \label{fig:bijection vertices}
\end{figure}

\begin{example}
Figure~\ref{fig:bijection vertices skew} illustrates the bijection between a flow corresponding to a vertex of the flow polytope $\mathcal{F}_{G(13,8)}(\aaa,\bb)$ for $\aaa=(4,0,1,2,1,2,0,0,1,0,1,0,0)$ and $\bb=(0,0,1,0,0,1,1,1,4,0,0,0,2)$ and a vertex plane partition.
\end{example}

\begin{corollary} \label{cor:vertex plane partitions}
Let $n$ and $m$ be nonnegative integers. The vertices of $\mathcal{F}_{G(n,m)}(\mathbf{1})$ are in correspondence with plane partitions $\pi=(\pi_{ij})$ of staircase shape $\delta_n=(n,n-1,\ldots,1)$ with entries at most $m$ where the following conditions hold: 
\begin{itemize}
    \item either $\pi_{i,j} <\pi_{i+1,j-1}$ or $\pi_{i,j} = \pi_{i-1,j}$,
    \item if $\pi_{i,j}=\pi_{i,j+1}$ then $\pi_{i-1,j}=\pi_{i-1,j+1}$.
\end{itemize}
\end{corollary}

\begin{example}
Figure~\ref{fig:bijection vertices straight} illustrates the bijection between a flow corresponding to a vertex of the flow polytope $ \mathcal{F}_{G(5,6)}(\aaa)$ for $\aaa=(1,1,1,1,1)$ and a vertex plane partition of shape $\delta_5$.
\end{example}

%--------------------------------------Section-------------------------------------%
\section{Enumeration of number of vertices} \label{sec:enumeration vertices}

Let $v^{(n,m)}(\aaa, \bb)$ be the number of vertices of $\mathcal{F}_{G(n,m)}(\aaa,\bb)$, $v^{(n,m)}(\aaa):=v^{(n,m)}(\aaa, \mathbf{0})$, and $v^{(n,m)}_{\textup{unsplit}}(\aaa, \bb)$ be the number of unsplittable flows of $\mathcal{F}_{G(n,m)}(\aaa,\bb)$. Note that by Corollary~\ref{conj:vert flows skew}, we have that $v^{(n,m)}(\aaa)=v_{\textup{unsplit}}^{(n,m)}(\aaa,\mathbf{0})$. In this section, we give different recursions for $v^{(n,m)}(\aaa)$ and $v_{\textup{unsplit}}^{(n,m)}(\aaa,\bb)$. 

For some flow $\xx$ on $G(n,m)$ and a column $0\leq c \leq m-1$ of horizontal arcs $((i,c),(i,c+1))$ for $1\leq i \leq n$ in $G(n,m)$, let $\xx_{\bigcdot c}:=(x_{1c},x_{2c},\ldots, x_{nc})$ be the vector recording the flow on column $c$. Further, we let $v^{(n,m)}_{\textup{unsplit}}(\aaa, \bb, c, \uu)$ for $0\leq c \leq m-1$ be the number of unsplittable flows ${\bf x}$ in $\mathcal{F}_{G(n,m)}(\aaa,\bb)$ such that $\xx_{\bigcdot c}=\uu$. Note that, for $\aaa\in \NN^n$, by partitioning all different possible flows $\xx$ on $G(n,m)$ by their flow $\xx_{\bigcdot c}$ for some column $0\leq c \leq m-1$, we have that $v^{(n,m)}_{\textup{unsplit}}(\aaa,\bb)=\sum_{\uu\in \NN^n} v^{(n,m)}_{\textup{unsplit}}(\aaa, \bb, c, \uu)$. Indeed, if $\aaa\in \NN^n$, then unsplittable flows in $\mathcal{F}_{G(n,m)}(\aaa,\bb)$ are integral. Furthermore, since flow can only go down and right in $G(n,m)$, we can immediately note that $v^{(n,m)}_{\textup{unsplit}}(\aaa, \bb, c, \uu)=0$ if $\aaa$ does not dominate $\uu$. So instead of summing over all vectors in $\NN^n$, we already know that $$v^{(n,m)}_{\textup{unsplit}}(\aaa,\bb)=\sum_{\substack{\uu\in \NN^n: \aaa \,\trianglerighteq\,\uu}} v^{(n,m)}_{\textup{unsplit}}(\aaa, \bb, c, \uu).$$ Knowing where $\aaa$ is 0 and where it is positive allows to further make the number of terms in this sum smaller. To do so, recall that  $\chi(\aaa)$ denotes the $0/1$-vector with the same support as $\aaa$.

\begin{lemma}\label{lem:onlycareaboutchi}
For $\aaa,\aaa',\bb \in \NN^n$, if $\chi(\aaa)=\chi(\aaa')$, then we have that $v^{(n,m)}_{\textup{unsplit}}(\aaa,\bb)=v^{(n,m)}_{\textup{unsplit}}(\aaa',\bb).$ In particular, for $\bb=\mathbf{0}$, $v^{(n,m)}(\aaa)=v^{(n,m)}(\aaa')$. 
\end{lemma}

\begin{proof}
This follows from the fact that if $a_i>0$, for any unsplittable flow of $\mathcal{F}_{G(n,m)}(\aaa,\bb)$, all units of flow starting at $(i,0)$ will follow the same trajectory within $G(n,m)$ since the flow starting at a source vertex is unsplittable. Therefore, it only matters where $a_i$ is zero and where it is positive to understand all unsplittable flows of $\mathcal{F}_{G(n,m)}(\aaa,\bb)$. 
\end{proof}

\begin{lemma}
For positive integers $n$ and $m$ and $0\leq c\leq m-1$, we have that $$v^{(n,m)}_{\textup{unsplit}}(\aaa,\bb)=\sum_{\substack{\uu\in \NN^n:\,  \aaa\,\trianglerighteq\,\uu \textup{ and}\\ \chi(\aaa) \,\trianglerighteq\, \chi(\uu)}} v_{\textup{unsplit}}^{(n,m)}(\aaa, \bb, c, \uu).$$
\end{lemma}

\begin{proof}
We show that if $\chi(\aaa)$ does not dominate $\chi(\uu)$, then $v^{(n,m)}_{\textup{unsplit}}(\aaa,\bb, c,\uu)=0.$ Suppose not. Since $\chi(\aaa)$ does not dominate $\chi(\uu)$, there exists $j\in [n]$ such that $\sum_{i=1}^j \chi(\aaa)_i< \sum_{i=1}^j \chi(\uu)_i$. Thus, for some unsplittable flow counted by $v^{(n,m)}_{\textup{unsplit}}(\aaa,\bb, c, \uu)$ corresponding to some flow $\xx$, this means that $x_{ic}>0$ for more $i$'s among $1\leq i \leq j$ than there are vertices $(1,0)$ through $(j,0)$ with positive netflow in $G(n,m)$. Since flow can only go down and right in $G(n,m)$, and since $\xx$ is unsplittable, this is a contradiction. 
\end{proof}

In this section, we refine our understanding of when $v^{(n,m)}_{\textup{unsplit}}(\aaa,\bb, c, \uu)=0$ to give two different recursions to calculate $v^{(n,m)}_{\textup{unsplit}}(\aaa,\bb)$.  Both are based on the following idea: we can calculate $v^{(n,m)}_\textup{unsplit}(\aaa,\bb,c, \uu)$ by calculating the product of the number of flows on two graphs obtained after deleting column $c$ of horizontal arcs in $G(n,m).$

\begin{lemma}\label{lem:vertexmult}
Fix positive integers $n$, $m$ and let $1 \leq c \leq m-2$. Then $$v^{(n,m)}_{\textup{unsplit}}(\aaa,\bb)=\sum_{\substack{\uu\in \NN^n: \, \aaa \,\trianglerighteq\, \uu \textup{ and}\\ \chi(\aaa) \,\trianglerighteq\, \chi(\uu)}}v^{(n,c)}_{\textup{unsplit}}(\aaa,\uu)\cdot v^{(n,m-c-1)}_{\textup{unsplit}}(\uu,\bb).$$
\end{lemma}

\begin{proof}
We observe that the number of unsplittable flows counted by $v^{(n,m)}_{\textup{unsplit}}(\aaa,\bb, c, \uu)$, illustrated in Figure~\ref{fig: unsplittable col c}, is the product of the number of unsplittable flows on the two graphs illustrated in Figure~\ref{fig: unsplittable col c left}. Note that this is equivalent to the product of the number of unsplittable flows on the respective graphs illustrated in Figure~\ref{fig: unsplittable col c right},
which in turn is equal to $v^{(n,c)}_{\textup{unsplit}}(\aaa,\uu) \cdot v^{(n,m-c-1)}_{\textup{unsplit}}(\uu,\bb)$.
\end{proof}

\begin{figure}
      \centering
  \begin{subfigure}[b]{0.4\textwidth}
  \centering
     \raisebox{30pt}{
     \includegraphics[scale=0.5]{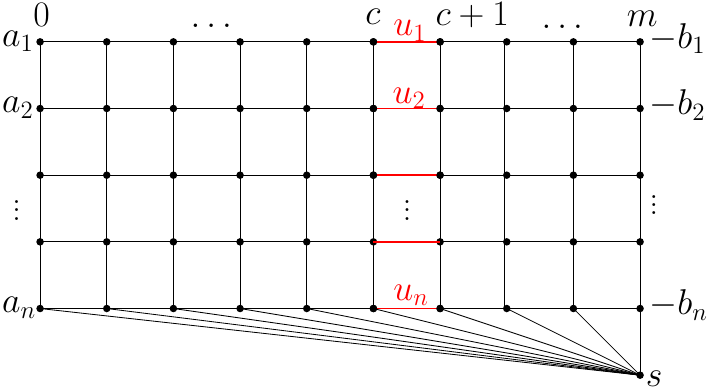}}
    \caption{}
      \label{fig: unsplittable col c}
\end{subfigure}
\begin{subfigure}[b]{0.3\textwidth}
\centering
   \includegraphics[scale=0.5]{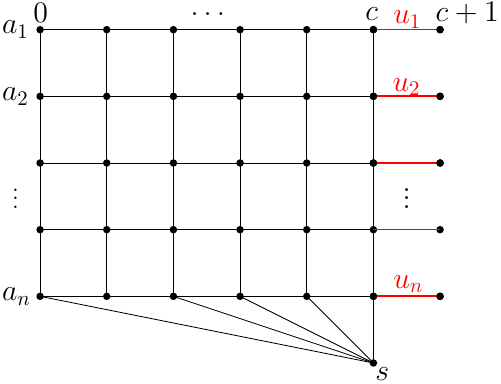}
   \includegraphics[scale=0.5]{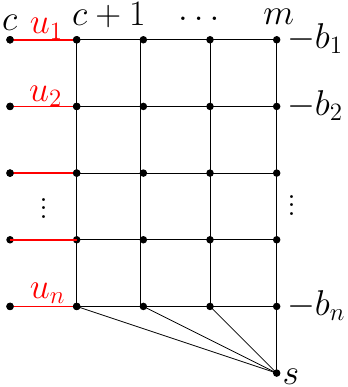}
  
   \caption{}
\label{fig: unsplittable col c left}
\end{subfigure}
 \begin{subfigure}[b]{0.2\textwidth}
 \centering
    \includegraphics[scale=0.5]{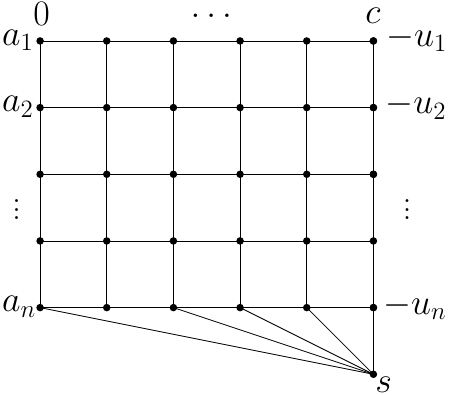}
  \includegraphics[scale=0.5]{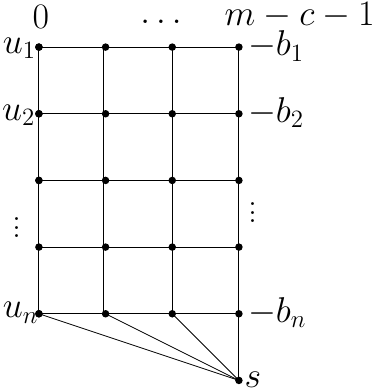}
\caption{}
  \label{fig: unsplittable col c right}
 \end{subfigure}
\caption{Decomposition of unsplittable flows in the proof of Lemma~\ref{lem:vertexmult}.}
\label{fig: unsplit flow decompositon}
\end{figure}

For the cases when $c=0$ and $c=m-1$, to avoid confusion, we keep the graphs in Figure~\ref{fig: unsplittable col c left} instead of those from Figure~\ref{fig: unsplittable col c right}. We name these graphs $H_\ulcorner(n)$ and $H_\urcorner(n)$, and we consider a third graph $H_\top(n)$ that allows us to unify both cases (see Figure~\ref{fig:Hn1}).

\begin{definition} \;
\begin{itemize}
\item Let $H_\ulcorner(n)$ be the directed graph with vertices $(1,0), \ldots, (n,0),(1,1), \ldots, (n,1)$ as well as a vertex called $s$. Arcs go from $(i,0)$ to $(i,1)$ for every $1\leq i\leq n$, and from $(i,0)$ to $(i+1,0)$ for every $1 \leq i\leq n-1$, and finally from $(n,0)$ to $s$.

\item  Let $H_\urcorner(n)$ be the directed graph with vertices $(1,-1), \ldots, (n,-1),(1,0), \ldots, (n,0)$ as well as a vertex called $s$. Arcs go from $(i,-1)$ to $(i,0)$ for every $1\leq i\leq n$, and from $(i,0)$ to $(i+1,0)$ for every $1 \leq i\leq n-1$, and finally from $(n,0)$ to $s$.

\item Let $H_\top(n)$ be the directed graph with vertices $(1,-1), \ldots, (n,-1),(1,0), \ldots, (n,0), (1,1), \ldots, (n,1)$ as well as a vertex called $s$. Arcs go from $(i,-1)$ to $(i,0)$ for every $1\leq i\leq n-1$, from $(i,0)$ to $(i,1)$ for every $1\leq i\leq n$, and from $(i,0)$ to $(i+1,0)$ for every $1 \leq i\leq n$, and finally from $(n,0)$ to $s$.

\end{itemize}
\end{definition}

\begin{definition} \;
\begin{itemize}
\item For $\aaa, \bb \in \NN^n$, denote by $\mathcal{F}_{H_\ulcorner(n)}(\aaa,\bb)$ the flow polytope on the graph $H_\ulcorner(n)$ where vertices $(1, 0), \ldots,$  $(n, 0)$ respectively have netflow $a_1, \dots, a_n$, and vertices $(1,1), \ldots, (n,1)$ respectively have netflow $-b_1, \ldots, -b_n$, and the sink vertex $s$ has netflow $-\sum_{i=1}^n a_i + \sum_{i=1}^n b_i$.

\item For $\aaa, \bb \in \NN^n$, denote by $\mathcal{F}_{H_\urcorner(n)}(\aaa,\bb)$ the flow polytope on the graph $H_\urcorner(n)$ where vertices $(1, -1),$ $ \ldots,(n, -1)$ respectively have netflow $a_1, \dots, a_n$, and vertices $(1,0), \ldots, (n,0)$ respectively have netflow $-b_1, \ldots, -b_n$, and the sink vertex $s$ has netflow $-\sum_{i=1}^n a_i + \sum_{i=1}^n b_i$. 

\item For $\aaa, \bb \in \NN^n$, denote by $\mathcal{F}_{H_\top(n)}(\aaa,\bb)$ the flow polytope on the graph $H_\top(n)$ where vertices $(1, -1),\ldots, (n, -1)$ respectively have netflow $a_1, \dots, a_n$, and vertices $(1,1), \ldots, (n,1)$ respectively have netflow $-b_1, \ldots, -b_n$, and the sink vertex $s$ has netflow $-\sum_{i=1}^n a_i + \sum_{i=1}^n b_i$.  All vertices $(i, 0)$ for $1 \leq i\leq n$  have netflow 0.
\end{itemize}
\end{definition}

\begin{figure}[h!]
    \centering
    \includegraphics[scale=0.7]{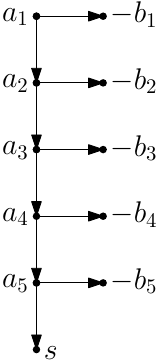} \qquad \includegraphics[scale=0.7]{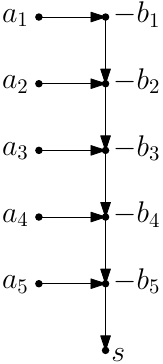} \qquad \includegraphics[scale=0.7]{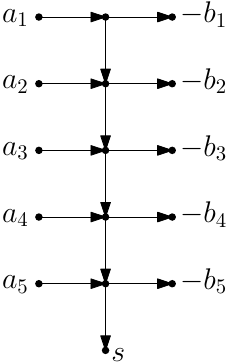}
    \caption{Illustration of graphs $H_\ulcorner(5)$, $H_\urcorner(5)$, and $H_\top(5)$.}
    \label{fig:Hn1}
\end{figure}

By the same reasoning as for Lemma~\ref{lem:vertexmult}, we obtain the following lemma by letting $c=0$ and $c=m-1$.

\begin{lemma}\label{lem:vertexmultfirstlast}
We have $$v^{(n,m)}_{\textup{unsplit}}(\aaa,\bb)=\sum_{\substack{\uu\in \NN^n:\, \aaa \,\trianglerighteq\, \uu \textup{ and}\\ \chi(\aaa) \,\trianglerighteq\, \chi(\uu)}} \left|\left\{\mathbf{x} \mid \mathbf{x} \textup{ is an integral unsplittable flow of }\mathcal{F}_{H_\ulcorner(n)}(\aaa,\uu)\right\}\right|\cdot v^{n,m-1}_{\textup{unsplit}}(\uu,\bb),$$ and we have $$v^{(n,m)}_{\textup{unsplit}}(\aaa)=\sum_{\substack{\uu\in \NN^n: \,\aaa \,\trianglerighteq\, \uu \textup{ and}\\ \chi(\aaa) \,\trianglerighteq\, \chi(\uu)}} v^{(n,m-1)}_{\textup{unsplit}}(\aaa,\uu) \cdot  \left|\left\{\mathbf{x} \mid \mathbf{x} \textup{ is an integral unsplittable flow of }\mathcal{F}_{H_\urcorner(n)}(\uu,\bb)\right\}\right|.$$
\end{lemma}

\begin{remark} \label{rem: 0 or 1 flow for trees}
Since $H_\ulcorner(n)$ and $H_\urcorner(n)$ are trees, their flow polytopes are empty or consist of one integral flow (by Gallo and Sodini's \cite{GalloSodini} characterization of vertices of flow polytopes, see Theorem~\ref{char: vertices G as forests}).  In what follows, we characterize for which vectors $\uu$ the polytopes $\mathcal{F}_{H_\ulcorner(n)}(\aaa,\uu)$ and $\mathcal{F}_{H_\urcorner(n)}(\uu,\bb)$ contain one integral unsplittable flow. 
\end{remark}

Recall from Definition~\ref{def:split and merge} that when we talk about unsplittable flows, if there is a vertex with negative netflow, then the flow cannot also be on an outgoing arc of that vertex.  Because of this, we have the following lemma. 

\begin{lemma}\label{lem:equivalenceultur}
Unsplittable flows in $\mathcal{F}_{H_\ulcorner(n)}(\aaa,\bb)$ are in bijection with the unsplittable flows in $\mathcal{F}_{H_\top(n)}(\aaa,\bb)$ which are in bijection with the unsplittable flows in $\mathcal{F}_{H_\urcorner(n)}(\aaa,\bb)$.
\end{lemma}

\begin{proof}
Consider an unsplittable flow in $\mathcal{F}_{H_\top(n)}(\aaa,\bb)$. Consider the flow on $H_\ulcorner(n)$ where the flow on every arc is equal to the flow on the corresponding arc in $H_\top(n)(\aaa,\bb)$. This gives an unsplittable flow on  $H_\ulcorner(n)$.

Now take an unsplittable flow $\xx$ in $\mathcal{F}_{H_\ulcorner(n)}(\aaa,\bb)$. For each arc of $H_\top(n)$ that has a corresponding arc in $H_\ulcorner(n)$, assign the value in $\xx$ for that arc. For the arcs $((i,-1),(i,0))$, let the flow take value $a_i$. Then since $\xx$ was a feasible unsplittable flow in $\mathcal{F}_{H_\ulcorner(n)}(\aaa,\bb)$, this is a feasible unsplittable flow in $\mathcal{F}_{H_\top(n)}(\aaa,\bb)$.

The same ideas can be used to show the bijection between the unsplittable flows in $\mathcal{F}_{H_\urcorner(n)}(\aaa,\bb)$ and in $\mathcal{F}_{H_\top(n)}(\aaa,
\bb)$, respectively.
\end{proof}

For the recursions, we need to understand unsplittable flows in $\mathcal{F}_{H_\top(n)}(\aaa,\bb)$. This will also allow us to characterize what $\xx_{\bigcdot c+1}$  can be if one is given $\xx_{\bigcdot c}$. The patterns of zeros will end up playing an important role through the following definition. 

\begin{definition}\label{def:zdomination}
 Let $\aaa=(a_1, a_2, \ldots, a_n), \bb=(b_1, b_2, \ldots, b_n) \in \NN^n$ be such that $\aaa\trianglerighteq\bb$, and $\chi(\aaa)\trianglerighteq\chi(\bb)$. Let $Z_\aaa=\{i\in[n]\,|\,a_i=0\}$ and $Z_\bb=\{i\in[n]\,|\,b_i=0\}$. Consider the  matching $M\subseteq Z_\aaa\times Z_\bb$ satisfying the following conditions:
\begin{itemize}
    \item every $i\in Z_\aaa$ is matched to a unique $j\in Z_\bb$,
    \item if $(i,j)\in M$, then $i\geq j$ (which is possible to require since $\chi(\aaa)\trianglerighteq\chi(\bb)$),
    \item if $(i,j), (i',j')\in M$ such that $i<i'$, then $j<j'$, that is, the matching is noncrossing, and
    \item if $(i,j)\in M$, then there exists no $j'\in Z_\bb$ that is unmatched such that $i\leq j' <j$.
\end{itemize}
Note that this defines a unique matching in $Z_\aaa\times Z_\bb$. Algorithmically, $M$ can be built as follows: take the largest unmatched index, say $i'$, in $Z_\aaa$ and match it to the largest index in $Z_\bb$ that is unmatched and that is smaller or equal to $i'$. Do this recursively until all indices in $Z_\aaa$ are matched.  We say that $\aaa$ $z$-dominates $\bb$, and denote it by $\aaa \trianglerighteq_z \bb$, where $$z:=\max\left\{\max\{i-j\mid (i,j)\in M\}, 1\right\}.$$ Note that if we say that $\aaa\trianglerighteq_z\bb$ for some $z\in \NN$, then we are assuming that $\aaa\trianglerighteq\bb$ and that $\chi(\aaa)\trianglerighteq\chi(\bb)$. Note that, because of these assumptions, $z$ exists and is at most $n-1$.
\end{definition}

\begin{remark} \label{rem:def 1 domination} It turns out that 1-domination will play an important role in what comes next, so it is useful to notice the following.  Let $\chi(\aaa)=:\vv$ and $\chi(\bb)=:\ww$. If $\aaa\trianglerighteq_1\bb$, then it follows that $(v_{i+2}, v_{i+3}, \ldots, v_n)$ dominates $(w_{i+2},w_{i+3}, \ldots, w_n)$ for every $i\in [n]$ such that $b_i=0$ and $b_{i+1}>0$.\end{remark}

\begin{example}
Let $\aaa=(2,3,1,0,5,4,3,4,4,0)$ and $\bb=(0,4,1,1,3,0,1,0,0,0)$. Then $\aaa\trianglerighteq_3 \bb$ since $M=\{(10,10),(4,1)\}$ (but $\aaa \not\trianglerighteq_2 \bb$ and $\aaa \not\trianglerighteq_2 \bb$). On the other hand, $(3,2,0,0,2,3,4)\trianglerighteq_1(3,0,1,0,1,0,6)$ since $M=\{(4,4),(3,2)\}$.
\end{example}

\begin{lemma}\label{thm:dominance}
Consider a vector $\aaa \in \NN^n$. For each $\jj\in \{0,1\}^n$ such that $\chi(\aaa)\trianglerighteq_1\jj$, there exists a unique vector $\bb\in \NN^n$ such that $\chi(\bb)=\jj$ and such that $\mathcal{F}_{H_\top(n)}(\aaa, \bb)$ contains an unsplittable flow. Furthermore, in that case, $\mathcal{F}_{H_\top(n)}(\aaa, \bb)$ contains a single unsplittable flow. For all other vectors $\bb\in \NN^n$ outside of these conditions, $\mathcal{F}_{H_\top(n)}(\aaa, \bb)$ has no unsplittable flows. 
\end{lemma}

Before proving this lemma, let us develop some intuition on the relationship between vectors $\aaa$ and $\bb$ in $\mathcal{F}_{H_\top(n)}(\aaa, \bb)$. Certainly, as explained previously, we know that $\mathcal{F}_{H_\top(n)}(\aaa,\bb)$ has no unsplittable flows if $\aaa\not\trianglerighteq\bb$ or if $\chi(\aaa)\not\trianglerighteq\chi(\bb)$. So we can assume that $\aaa\trianglerighteq_z\bb$ for some $z\leq n-1$. Consider some vector $\bb \in \NN^n$ such that there is an unsplittable flow in $\mathcal{F}_{H_\top(n)}(\aaa, \bb)$. The value of each $b_k$ for $k\in[n]$ acts as a switch. If $b_k=0$, then any incoming flow to $(k,0)$ is pushed to $(k+1, 0)$. If $b_k>0$, then any incoming flow to $(k,0)$ is pushed to $(k, 1)$. Because the flow is unsplittable, the flow is sent to one edge or the other, not both. 

Consider $i,j\in [n]$ such that $i<j$, $b_i>0$, $b_j>0$, and $b_k=0$ for all $i<k<j$. Because $b_i$ and $b_j$ are greater than zero, we know there is positive flow on the horizontal arcs $((i,0),(i,1))$ and $((j,0),(j,1))$ and not on the vertical arcs $((i,0),(i+1,0))$ and $((j,0),(j+1,0))$. Therefore, all flow beginning at vertices $(i+1,-1), \ldots, (j,-1)$ must end up going on arc $((j,0),(j,1))$. That means that we have that $b_j=\sum_{k=i+1}^j a_k$ as in Figure \ref{fig:Zero-Pattern-Intuition}.

Similarly, for the smallest $i$ such that $b_i>0$, we must have that $b_i=\sum_{k=1}^i a_k$. Note that by setting $\bb$ in this way, there is a unique unsplittable flow in $\mathcal{F}_{H_\top(n)}(\aaa,\bb)$. Moreover, observe that once we select a zero pattern, this either gives us a unique $\bb$ with this zero pattern for which $\mathcal{F}_{H_\top(n)}(\aaa, \bb)$ contains a single unsplittable flow, or we get that there are no unsplittable flows for any $\bb$ with this zero pattern. See for example Figures~\ref{fig:unsplittablefirstcolumn_no_unsplit} and \ref{fig:unsplittablefirstcolumn_unsplit}. In the proof below, we characterize for which zero patterns we are in the first setting and for which we are in the second setting.

\begin{proof}[Proof of Lemma~\ref{thm:dominance}]

\textbf{Claim 1:} If $\chi(\aaa)\trianglerighteq_z \jj$ such that $z\geq 2$, there does not exists a valid unsplittable flow in $\mathcal{F}_{H_\top(n)}(\aaa, \bb)$. 

Suppose that $\aaa \trianglerighteq_z \bb$ for some $z\geq 2$ and that there is a feasible unsplittable flow in $\mathcal{F}_{H_\top(n)}(\aaa,\bb)$. Let $(i^*,j^*)\in M$ be the matched zeros where $i^*-j^*\geq 2$ and where $i^*$ is as large as possible, i.e., any $i'\in Z_\aaa$ such that $i'>i^*$ is either matched to $i'\in Z_\bb$ or to $i'-1\in Z_\bb$. We know that there cannot be any unmatched zeros among entries $b_{j^*+1}, b_{j^*+2}, \ldots, b_{i^*}$, and any zero among these would need to be matched to some $i'\in Z_\aaa$ such that $i'>i^*$, but they are too far away to be matched together. So $b_{i^*}$ is the only entry amongst $b_{j^*+1}, b_{j^*+2}, \ldots, b_{i^*}$ that can be zero. In particular, since $i^*-j^*\geq 2$, we have that $b_{i^*-1}>0$, which means that there must be positive flow going on the arc $((i^*-1,0),(i^*-1,1))$, and thus, for an unsplittable flow, there can be no flow going down on the arc $((i^*-1,0),(i^*,0))$. 

Let $\overline{i}$ be the greatest index such that $i\in Z_\aaa$ for all $i^*\leq i\leq \overline{i}$. We claim that there exists $i^*\leq j\leq \overline{i}$ such that $b_j>0$. Indeed, if all such $b_j$'s were zero, then one of those zeros would be unmatched since all zeros $i'\in Z_\aaa$ such that $i'>i^*$ are matched to $i'$ or $i'-1$, and $i^*$ is matched to $j^*>i^*$. Note that this holds even if $\overline{i}=n$ or if $\overline{i}=i^*$. Since $b_j>0$, there must be positive flow on arc $((j,0),(j,1))$, but there is no vertex with positive netflow among $(i^*,-1), \ldots, (\overline{i},-1)$ to feed this flow, and flow cannot be coming down from $(i^*-1,0)$ from our previous considerations. Therefore, there cannot be a feasible unsplittable flow in $\mathcal{F}_{H_\top(n)}(\aaa,\bb)$ if $\aaa\trianglerighteq_z\bb$ where $z\geq 2$.

\textbf{Claim 2:} For any $\jj\in \{0,1\}^n$ such that $\chi(\aaa)\trianglerighteq_1\jj$, there exists a unique $\bb\in \NN^n$ such that $\chi(\bb)=\jj$ and for which $\mathcal{F}_{H_\top(n)}(\aaa, \bb)$ contains a single unsplittable flow.

We prove this by induction on $n$. Let $n=1$. Then there are two cases: $\chi(\aaa)=(1)$ or $\chi(\aaa)=(0)$. In the first case, there are two vectors $\jj\in \NN^1$ that are $1$-dominated by $\chi(\aaa)$, namely $\jj=(1)$ and $\jj=(0).$ If $\jj=(1)$, let $\bb=(a_1)$, and there is an unsplittable flow pushing $a_1$ units from $(1,-1)$ to $(1,0)$ to $(1,1)$. If $\jj=(0)$, let $\bb=(0)$, and there is an unsplittable flow pushing $a_1$ units from $(1,-1)$ to $(1,0)$ to the sink $s$. In the second case, there is only one $\jj$ that is $1$-dominated by $\chi(\aaa)$, namely $(0)$. In that case, there is no netflow anywhere, and so this is trivially an unsplittable flow. Thus, when $n=1$, for any $\aaa\in \NN^1$ and for any $\jj\in \NN^1$ that is $1$-dominated by $\chi(\aaa)$, there is exactly one $\bb\in \NN^1$ such that $\chi(\bb)=\jj$ and such that $\mathcal{F}_{H_\top(n)}(\aaa, \bb)$ contains a single unsplittable flow.

Assume that the claim holds up to $n-1$. We show then that the same holds for $n$. Consider some $\aaa, \jj \in \NN^n$ such that $\chi(\aaa)\trianglerighteq_1\jj$. 

First consider the case when $a_n>0$. Then $(\chi(\aaa)_1, \ldots, \chi(\aaa)_{n-1})\trianglerighteq_1(j_1, \ldots, j_{n-1})$. Indeed, since $\chi(\aaa)\trianglerighteq_1\jj$, this means that the matching $M$ defined previously is such that every zero in $\chi(\aaa)$ gets matched to the corresponding entry in $\jj$ or to the one before. Since $\chi(\aaa)_n=1$, the same matching is valid after truncating the two vectors to their first $n-1$ entries. So, by our induction hypothesis, we have that there exists some vector $(b_1, \ldots, b_{n-1})$ such that $\chi((b_1, \ldots, b_{n-1}))=(j_1, \ldots, j_{n-1})$, and such that $\mathcal{F}_{H_\top(n)}((a_1,\ldots, a_{n-1}), (b_1, \ldots, b_{n-1}))$ contains a single unsplittable flow. Suppose that $y$ units went into the sink there (where $y$ can be zero). If $j_n=0$, then $\mathcal{F}_{H_\top(n)}(\aaa, (b_1, \ldots, b_{n-1},0))$ contains a single unsplittable flow where $a_n+y$ units end in the sink. If $j_n=1$, then $\mathcal{F}_{H_\top(n)}(\aaa, (b_1, \ldots, b_{n-1},a_n+y))$ contains a single unsplittable flow where zero units end in the sink.

Now consider the case when $a_n=0$ and $j_n=0$. In the matching $M$, both $n$th 0-entries are matched together, and so we get that $(\chi(\aaa)_1, \ldots, \chi(\aaa)_{n-1})\trianglerighteq_1(j_1, \ldots, j_{n-1})$ just as in the previous case, and we can repeat the same argument that we had there for $j_n=0$.

Finally, consider the case when $a_n=0$ and $j_n=1$. Since $\chi(\aaa)\trianglerighteq_1\jj$, we know that the zero in the $n$th entry of $\aaa$ is matched to a zero in the $n$th or $(n-1)$th entry of $\jj$. Since $j_n=1$, we must have that $j_{n-1}=0$, and that this is the zero that $a_n$ is matched to in matching $M$. Restricting to the first $n-1$ entries of $\aaa$ and $\jj$ could thus potentially yield a different matching.  Let $i^*$ be the biggest index such that $a_{i^*}>0$ so that that $a_{i^*+1}=\ldots = a_n=0$. Since $a_n$ is matched to $j_{n-1}$, from $1$-domination, we have that $a_k$ is matched to $j_{k-1}$ for $i^*+1\leq k\leq n$ in $M$. Then  $(\chi(\aaa)_1, \ldots, \chi(\aaa)_{i^*})\trianglerighteq_1(j_1, \ldots, j_{i^*})$. So, by our induction hypothesis, we have that there exists some vector $(b_1, \ldots, b_{i^*})$ such that $\chi((b_1, \ldots, b_{i^*}))=(j_1, \ldots, j_{i^*})$, and such that $\mathcal{F}_{H_\top(n)}((a_1,\ldots, a_{i^*}), (b_1, \ldots, b_{i^*}))$ contains a single unsplittable flow. Suppose that $y$ units went into the sink there (where $y$ can be zero).  Then, $\mathcal{F}_{H_\top(n)}(\aaa, (b_1, \ldots, b_{i^*},0, \ldots, 0, a_{i^*}+y))$ contains a single unsplittable flow where zero units end in the sink, and $\chi((b_1, \ldots, b_{i^*},0, \ldots, 0, a_{i^*}+y))=\jj$ as desired.
\end{proof}

\begin{figure}[h!]
  \begin{subfigure}[b]{0.3\textwidth}
  \centering
    \includegraphics[scale=0.8]{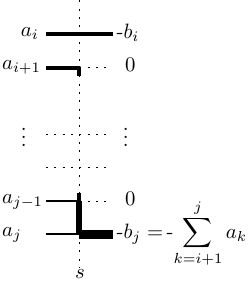}
    \caption{}
     \label{fig:Zero-Pattern-Intuition}
\end{subfigure}
\begin{subfigure}[b]{0.3\textwidth}
\centering
   \includegraphics[scale=0.8]{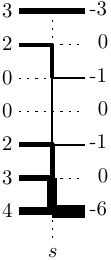}
  \caption{}
  \label{fig:unsplittablefirstcolumn_no_unsplit}
\end{subfigure}
 \begin{subfigure}[b]{0.3\textwidth}
 \centering
\includegraphics[scale=0.8]{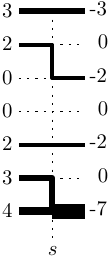}
 \caption{}
  \label{fig:unsplittablefirstcolumn_unsplit}
 \end{subfigure}
     \caption{(a) Schematic of the zero pattern. (b) The vector $\aaa\trianglerighteq_1 \bb$, however, there is no unsplittable flow on $\mathcal{F}_{H_\top(7)}(\aaa, \bb)$. (c) The vector $\bb$ has the same zero pattern as in (b), and it is the only $\bb$ for which there is a unique unsplittable flow in $\mathcal{F}_{H_\top(7)}(\aaa, \bb)$.}
     \label{fig:unsplittablefirstcolumn}
\end{figure}

\begin{corollary}\label{cor:updominance}
Consider $\bb \in \NN^n$ and $\jj\in \{0,1\}^n$ such that $\jj\trianglerighteq_1 \chi(\bb)$.
\begin{enumerate}
\item[(i)] There exists at least one vector $\aaa\in \RR^n$ such that $\chi(\aaa)=\jj$ and such that $\mathcal{F}_{H_\top(n)}(\aaa, \bb)$ contains an unsplittable flow. Furthermore, in that case, $\mathcal{F}_{H_\top(n)}(\aaa, \bb)$ contains a single unsplittable flow. For all other vectors $\aaa\in \RR^n$  outside of these conditions, $\mathcal{F}_{H_\top(n)}(\aaa, \bb)$ has no unsplittable flows.  

\item[(ii)] Moreover, if $b_k\geq \sum_{l=i+1}^k j_l$ for every  $b_i, b_k>0$ such that $1\leq i<k\leq n$ and $b_l=0$ for all $i<l<k$, and $b_i \geq \sum_{l=1}^i j_l$ for the smallest $i$ such that $b_i>0$, then there exists at least one vector $\aaa\in \NN^n$ such that $\chi(\aaa)=\jj$ and such that $\mathcal{F}_{H_\top(n)}(\aaa, \bb)$ contains an unsplittable flow. Furthermore, in that case, $\mathcal{F}_{H_\top(n)}(\aaa, \bb)$ contains a single unsplittable flow. For all other vectors $\aaa\in \NN^n$  outside of these conditions, $\mathcal{F}_{H_\top(n)}(\aaa, \bb)$ has no unsplittable flows. 

\end{enumerate}

\end{corollary}

\begin{proof}
Similarly to the previous argument, consider some vector $\aaa \in \RR^n$. If there is an unsplittable flow in $\mathcal{F}_{H_\top(n)}(\aaa, \bb)$, then we know there must be $b_i$ units of flow on each arc $((i,0),(i,1))$ (and so whenever $b_i=0$, there is no flow on $((i,0),(i,1))$). Consider $i,k\in [n]$ with $i<k$ such that $b_i,b_k>0$ and $b_l=0$ for all $i<l<k$. Note that since the flow is unsplittable and we know there is flow on $((i,0),(i,1))$ and $((k,0),(k,1))$, there can be no flow on arcs $((i,0),(i+1,0))$ and $((k,0),(k+1,0))$. Therefore, all flow entering through vertices with positive netflow among $(i+1,-1), \ldots, (k,-1)$ must end up going on arc $((k,0),(k,1))$. That means that we have $b_k=\sum_{l=i+1}^k a_l$. Respectively, for the smallest $i$ such that $b_i>0$, we must have that $b_i=\sum_{l=1}^i a_l$. As long as $1\leq \sum_{l=i+1}^k j_l =: t$ (respectively $\sum_{l=1}^i j_l:=t$), we can then set $a_l=b_k/t$ for $i+1\leq l \leq k$ (respectively $a_l=b_i/t$ for $1\leq l \leq i$) such that $j_l=1$, and $a_l=0$ for any $l$ such that $j_l=0$. Note that other choices for the $a_l$'s would have worked here too, which is why there can exist more than one vector $\aaa$, and that for each of them there is a unique unsplittable flow in $\mathcal{F}_{H_\top(n)}(\aaa,\bb)$. From Lemma~\ref{thm:dominance}, we know that $\aaa\trianglerighteq_1\bb$, otherwise, there would be no flow. 

Note that if we want $\aaa\in \NN^n$ as in (ii), then we also need $b_k\geq \sum_{l=i+1}^k j_l$ for every such $k$ so that at least one unit is coming from each $a_l$ that needs to be positive. Similarly for the first positive $b_i$, we need $b_i\geq \sum_{l=1}^i j_l$. Finally, in the case $\bb=\mathbf{0}$, observe that since there are no positive $b_i$'s, the condition is vacuously satisfied so every $\jj\in \{0,1\}^n$ yields at least one vector $\aaa\in \mathbb{N}^n$ such that $\mathcal{F}_{H_\top(n)}(\aaa,\bb)$ contains a single unsplittable flow.
\end{proof}

We translate the idea present in the last two results into a definition.

\begin{definition}
\begin{itemize}
\item[]
    \item For any $\aaa\in \NN^n$ and $\jj \in \{0,1\}^n$ such that $\chi(\aaa)\trianglerighteq_1\jj$, let $\ww_{\jj,\aaa, \rightarrow}$ be the unique vector such that $\chi(\ww_{\jj,\aaa,\rightarrow})=\jj$ and $\mathcal{F}_{H_\top(n)}(\aaa,\ww_{\jj,\aaa,\rightarrow})$ contains an unsplittable flow.
    
    \item For any $\bb\in \NN^n$ and $\jj \in \{0,1\}^n$ such that $\jj\trianglerighteq_1\chi(\bb)$, let $W_{\jj,\bb, \leftarrow}$ be the set of vectors $\ww$  such that $\chi(\ww)=\jj$ and $\mathcal{F}_{H_\top(n)}(\ww,\bb)$ contains an unsplittable flow.
\end{itemize}
\end{definition}

Note that whenever $W_{\jj,\bb, \leftarrow}$ contains more than one vector, then the number of unsplittable flows in $\mathcal{F}_{H_\top(n)}(\ww,\bb)$ for all $\ww\in W_{\jj,\bb, \leftarrow}$ is always the same, either $0$ or $1$ (see, e.g., Remark~\ref{rem: 0 or 1 flow for trees}).

\begin{corollary}\label{cor:dominanceallways}The following statements hold. 
\begin{enumerate}
    \item[(i)] Given $\xx_{\bigcdot c}$ for some unsplittable flow $\xx$ for $\mathcal{F}_{G(n,m)}(\aaa,\bb)$ for some $0\leq c\leq m-2$, then $\xx_{\bigcdot c+1}=\ww_{\jj,\xx_{\bigcdot c}, \rightarrow}$ for some $\jj\in \{0,1\}^n$ such that $\chi(\xx_{\bigcdot c})\trianglerighteq_1 \jj$. 
    \item[(ii)] Given $\xx_{\bigcdot c+1}$ for some unsplittable flow $\xx$ for $\mathcal{F}_{G(n,m)}(\aaa,\bb)$ for some $0\leq c\leq m-2$, then $\xx_{\bigcdot c}=\ww$ for some $\ww\in W_{\jj,\xx_{\bigcdot c+1}, \leftarrow}$ for some $\jj\in \{0,1\}^n$ such that $\jj\trianglerighteq_1\chi(\xx_{\bigcdot c+1})$. 
    \item[(iii)] For some unsplittable flow $\xx$ for $\mathcal{F}_{G(n,m)}(\aaa,\bb)$, we have that $\xx_{\bigcdot 0}=\ww_{\jj,\aaa, \rightarrow}$ for some $\jj\in \{0,1\}^n$ such that $\chi(\aaa)\trianglerighteq_1 \jj$.
    \item[(iv)] For some unsplittable flow $\xx$ for $\mathcal{F}_{G(n,m)}(\aaa,\bb)$, we have that $\xx_{\bigcdot m-1}=\ww$ for some $\ww\in W_{\jj,\bb, \leftarrow}$ for some $\jj\in \{0,1\}^n$ such that $\jj\trianglerighteq_1\chi(\bb)$.
\end{enumerate}
\end{corollary}

\begin{proof}
Consider some unsplittable flow $\xx$ in $\mathcal{F}_{G(n,m)}(\aaa,\bb)$. Then restricting to different columns of the graph, we see that this yields unsplittable flows for $\mathcal{F}_{H_\ulcorner}(\aaa,\xx_{\bigcdot 0})$, $\mathcal{F}_{H_\top}(\xx_{\bigcdot c},\xx_{\bigcdot c+1})$ for any $0 \leq c \leq m-2$ and for $\mathcal{F}_{H_\urcorner}(\xx_{\bigcdot m-1},\bb)$. Lemma~\ref{thm:dominance} and Corollary~\ref{cor:updominance} yield the result.
\end{proof}

We are now ready to give the two main recurrences for the number $v^{(n,m)}_{\textup{unsplit}}(\aaa,\bb)$ of unsplittable flows. 

\subsection{Fixing the flow in the first column of $G(n,m)$}

\begin{theorem}
 \label{thm:fixingfirst}
Let $\aaa,\bb\in \NN^n$. Then $$v^{(n,m)}_{\textup{unsplit}}(\aaa,\bb)=\sum_{\substack{\jj\in \{0,1\}^n:\, \chi(\aaa) \,\trianglerighteq_1\, \jj}} v^{(n,m-1)}_{\textup{unsplit}}(\ww_{\jj,\aaa,\rightarrow},\bb).$$  
\end{theorem}

\begin{proof}
From Lemma~\ref{lem:vertexmultfirstlast}, we know that by partitioning our flows according to their vector $\xx_{\bigcdot 0}$, we have
\begin{align*}
    v^{(n,m)}_{\textup{unsplit}}(\aaa,\bb)&=\sum_{\substack{\uu\in \NN^n: \,\aaa \trianglerighteq \uu \textup{ and}\\ \chi(\aaa) \trianglerighteq \chi(\uu)}} \left|\left\{\mathbf{x} \mid \mathbf{x} \textup{ is an integral unsplittable flow of }\mathcal{F}_{H_\ulcorner(n)}(\aaa,\uu)\right\}\right|\cdot v^{n,m-1}_{\textup{unsplit}}(\uu,\bb)\\
    &=\sum_{\substack{\uu\in \NN^n:  \,\aaa \trianglerighteq \uu \textup{ and}\\ \chi(\aaa) \trianglerighteq \chi(\uu)}} \left|\left\{\mathbf{x} \mid \mathbf{x} \textup{ is an integral unsplittable flow of }\mathcal{F}_{H_\top(n)}(\aaa,\uu)\right\}\right|\cdot v^{n,m-1}_{\textup{unsplit}}(\uu,\bb),\\
\end{align*} 
where the second line follows from Lemma~\ref{lem:equivalenceultur}. From Lemma~\ref{thm:dominance}, we get that $$\left|\left\{\mathbf{x} \mid \mathbf{x} \textup{ is an integral unsplittable flow of }\mathcal{F}_{H_\top(n)}(\aaa,\uu)\right\}\right|$$ is one if $\uu=\ww_{\jj,\aaa,\rightarrow}$ for some $\jj\in\{0,1\}^n$ that is $1$-dominated by $\aaa$, and 0 otherwise.
\end{proof}

In particular, when ${\bf b}={\bf 0}$, we get the following recurrence for the number $v^{(n,m)}(\aaa)=v^{(n,m)}_{\textup{unsplit}}(\aaa)$ of vertices of $\mathcal{F}_{G(n,m)}({\bf a})$ since we know that $v^{(n,m-1)}(\ww_{\jj,\aaa,\rightarrow},\mathbf{0})=v^{(n,m-1)}(\jj,\mathbf{0})$ by Lemma~\ref{lem:onlycareaboutchi}.

\corfixingfirst

Through this recursion, we can calculate generating functions with $v^{(n,m)}(\aaa)$ as coefficients for small values of $n$. Note that we know that $v^{(n,0)}(\aaa)=1$ for any $\aaa$.

\begin{example}
In Table~\ref{tab:genfunctions5}, we give generating functions for $v^{(5,m)}(\aaa)$, which simultaneously give the values for $n<5$ since $v^{(n,m)}(a_1, \ldots, a_n)=v^{(n+1,m)}(0,a_1, \ldots, a_n)$. Indeed, adding a row at the top of our graph $G(n,m)$ with netflow of $0$ everywhere in that row doesn't change the possible flows. In Section~\ref{subsec:binomialcoefficients}, we discuss why the denominators of the generating functions are of the form $(1-x)^i$.

\end{example}

\subsection{Fixing the flow in the last column of $G(n,m)$ and the transfer-matrix method}

\begin{theorem}\label{thm:fixinglast}
Let $\aaa\in \NN^n$. Then $$v^{(n,m)}_{\textup{unsplit}}(\aaa, \bb)=\sum_{\substack{\uu\in \NN^n\cap W_{\chi(\uu),\bb,\leftarrow} :\\ \aaa \,\trianglerighteq\, \uu \textup{ and } \chi(\aaa) \,\trianglerighteq\, \chi(\uu)}} v^{(n,m-1)}_{\textup{unsplit}}(\aaa,\uu).$$ 
\end{theorem}

\begin{proof}
From Lemma~\ref{lem:vertexmultfirstlast}, we know that by partitioning our flows according to their vector $\xx_{\bigcdot m-1}$, we have
\begin{align*}
    v^{(n,m)}_{\textup{unsplit}}(\aaa,\bb)&=\sum_{\substack{\uu\in \NN^n:\\ \aaa \,\trianglerighteq\, \uu \textup{ and } \chi(\aaa) \,\trianglerighteq\, \chi(\uu)}} v^{(n,m-1)}_{\textup{unsplit}}(\aaa,\uu) \cdot \left|\left\{\mathbf{x} \mid \mathbf{x} \textup{ is an integral unsplittable flow of }\mathcal{F}_{H_\urcorner(n)}(\uu,\bb)\right\}\right|\\
    &=\sum_{\substack{\uu\in \NN^n \cap W_{\chi(\uu),\bb,\leftarrow}:\\ \aaa \,\trianglerighteq\, \uu \textup{ and } \chi(\aaa) \,\trianglerighteq\, \chi(\uu)}} v^{(n,m-1)}_{\textup{unsplit}}(\aaa,\uu),\\
\end{align*} 
where the second line follows from the fact that $\mathcal{F}_{H_\urcorner(n)}(\uu,\bb)$ contains exactly one flow if $\uu\in W_{\chi(\uu),\bb,\leftarrow}$ by Corollary~\ref{cor:dominanceallways}~(iv). 
\end{proof}

 In particular, when $\bb=\mathbf{0}$, we get the following recurrence for the number $v^{(n,m)}(\aaa)$ of vertices of $\mathcal{F}_{G(n,m)}(\aaa)$  since $\mathcal{F}_{H_\urcorner(n)}(\uu,\mathbf{0})$ contains exactly one flow for every vector $\uu$.

\begin{corollary}
Let $\aaa\in \NN^n$, then $$v^{(n,m)}(\aaa)=\sum_{\substack{\uu\in \NN^n: \aaa \,\trianglerighteq\, \uu \\ \textup{ and } \chi(\aaa) \,\trianglerighteq\, \chi(\uu)}} v^{(n,m-1)}_{\textup{unsplit}}(\aaa,\uu).$$ 
\end{corollary}

To refine our understanding of the sum in Theorem~\ref{thm:fixinglast}, we will introduce a directed graph such that the entries of the powers of its adjacency matrix yield $v^{(n,m-1)}_{\textup{unsplit}}(\aaa,\bb)$ based on the following lemma.

\begin{lemma}\label{lem:sequences}
We have that $v^{(n,m-1)}_{\textup{unsplit}}(\aaa,\bb)$ is equal to the number of sequences $\jj_1, \ldots, \jj_{m-1}$ where $\jj_i\in \{0,1\}^n$ for every $1 \leq i\leq m-1$, where $\chi(\aaa)\trianglerighteq_1 \jj_1$, $\jj_{m-1}\trianglerighteq_1 \chi(\bb)$, and where $\jj_i \trianglerighteq_1  \jj_{i+1}$ for every $1\leq i \leq m-2$. 
\end{lemma}

\begin{proof}
Let $\xx$ be an unsplittable flow on $G(n,m-1)$. This unsplittable flow can be understood as a sequence of its column vectors $\mathbf{x}_{\bigcdot 0}, \ldots, \mathbf{x}_{\bigcdot m-2}$. From Corollary~\ref{cor:dominanceallways}, we know that $\xx_{\bigcdot c}\trianglerighteq_1 \mathbf{x}_{\bigcdot c+1}$, and furthermore, that $\xx_{\bigcdot c+1}$ is the unique vector with this particular zero pattern for which there is a feasible flow, i.e., $\xx_{\bigcdot c+1}=\ww_{\chi(\xx_{\bigcdot c+1}),\xx_{\bigcdot c},\rightarrow}$ for $0 \leq c \leq m-3$, and $\xx_{\bigcdot 0}=\ww_{\chi(\xx_{\bigcdot 0}),\aaa,\rightarrow}$ and $\bb=\ww_{\chi(\bb),\xx_{\bigcdot m-2},\rightarrow}$.

We thus see that these sequences are in bijection with the sequences $\chi(\mathbf{x}_{\bigcdot 0})=\jj_1, \ldots, \chi(\mathbf{x}_{\bigcdot m-2})=\jj_{m-1}$. Further, from the definition of $1$-dominance, we have $\chi(\xx_{\bigcdot c}) \trianglerighteq_1 \chi(\xx_{\bigcdot c+1})$ for every $0\leq c \leq m-3$, $\chi(\aaa) \,\trianglerighteq_1\,\chi(\xx_{\bigcdot 0})$ and $\chi(\xx_{\bigcdot m-2})\,\trianglerighteq_1\,\chi(\bb)$.
\end{proof}

We can now better understand the number of such sequences by creating a graph $D_n$ that records $1$-dominance.

\begin{definition}
Let $D_n$ be the directed graph where the vertex set consists of $\{0,1\}^n$ and there is an arc from $\jj\in \{0,1\}^n$ to $\kk\in \{0,1\}^n$ if $\jj \trianglerighteq_1  \kk$. Note that $\jj\trianglerighteq_1 \jj$ for every $\jj \in \{0,1\}^n$, so $D_n$ has loops at every vertex. Further, let $A_n$ be the adjacency matrix of $D_n$, i.e., $(A_n)_{\jj,\kk}$ is $1$ if $\jj \trianglerighteq_1  \kk$ and 0 otherwise.
\end{definition}

By the {\em transfer-matrix method}  \cite[Thm. 4.7.2]{EC1}, we have that for $\aaa,\bb\in \{0,1\}^n$, the $(\aaa,\bb)$-entry of $(A_n)^m$ records the number of sequences of $\aaa,\jj_1,\jj_2,\ldots, \jj_{m-1},\bb$ where each entry $1$-dominates the next. Further observe that by using, say, the lexicographic order for the sets in $\{0,1\}^n$, we see that $A_n$ is an upper triangular matrix because such an order ensures that no vector is dominated by any vector before it. By abuse of notation, let $\ind({\bf a}):=\sum_{i=1}^n a_i2^{i-1}$ be the index of ${\bf a}\in \{0,1\}^n$ in lexicographic order.

\begin{example}
Figure~\ref{fig:D3} shows $D_3$.  
We order the vectors in $\{0,1\}^3$ lexicographically. This yields the following upper triangular matrices $A_3$, $(A_3)^2$, and $(A_3)^3$:
\begin{figure}[h!]
    \centering
    \includegraphics[scale=0.35]{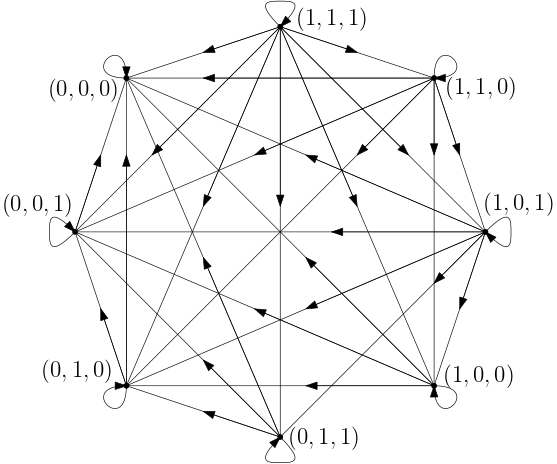}
    \caption{The graph $D_3$.}
    \label{fig:D3}
\end{figure}

$${\small \begin{pmatrix}
1&1&1&1&1&1&1&1\\
&1&1&1&\textcolor{red}{0}&1&1&1\\
&&1&1&1&1&1&1\\
&&&1&\textcolor{blue}{0}&1&1&1\\
&&&&1&1&1&1\\
&&&&&1&1&1\\
&&&&&&1&1\\
&&&&&&&1
\end{pmatrix},
\quad \begin{pmatrix}
1&2&3&4&3&6&7&8\\
&1&2&3&\textcolor{red}{1}&4&5&6\\
&&1&2&2&4&5&6\\
&&&1&\textcolor{blue}{0}&2&3&4\\
&&&&1&2&3&4\\
&&&&&1&2&3\\
&&&&&&1&2\\
&&&&&&&1
\end{pmatrix}, \quad  \begin{pmatrix}
1&3&6&10&7&19&26&34\\
&1&3&6&\textcolor{red}{3}&11&16&22\\
&&1&3&3&9&14&20\\
&&&1&\textcolor{blue}{0}&3&6&10\\
&&&&1&3&6&10\\
&&&&&1&3&6\\
&&&&&&1&3\\
&&&&&&&1
\end{pmatrix}.}
$$
Note how the entry corresponding to $((1,1,0),(0,1,1))$ in red is 0 in $A_3$ since $(1,1,0)\trianglerighteq_2 (0,1,1)$, but that it is positive in $(A_3)^2$ since $(1,1,0)\trianglerighteq_1 (1,0,1) \trianglerighteq_1 (0,1,1)$. (We'll see shortly that this $2$-domination and the entry becoming positive in $(A_3)^2$ is not a coincidence!) However, observe how the entry corresponding to $((1,0,0),(0,1,1))$ in blue is 0 in $A_3, (A_3)^2$ and $(A_3)^3$. This is because, even though $(1,0,0)$ is lexicographically greater than $(0,1,1)$ and thus comes before it within our order, we do not have that $(1,0,0)\trianglerighteq (0,1,1)$, and so there is no $z\in \NN$ such that this entry becomes positive in $(A_3)^z$.
\end{example}

\begin{lemma}\label{lem:zdominationingraph}
If the $(\aaa,\bb)$-entry of $(A_n)^m$ is positive, then so is the $(\aaa,\bb)$-entry of $(A_n)^{m+1}$. Furthermore, the $(\aaa,\bb)$-entry of $(A_n)^m$ is positive if and only if $\aaa$ $z$-dominates $\bb$ for some $z\leq m$.
\end{lemma}

\begin{proof}
We first note that if the $(\aaa,\bb)$-entry of $(A_n)^m$ is positive, then this means there is at least one directed path of length $m$ from $\aaa$ to $\bb$ in $D_n$. Given that each vertex of $D_n$ has a loop, we can easily extend this path of length $m$ to a path of length $m+1$ by going through one loop at any one of the vertices encountered along that path. Thus, the $(\aaa,\bb)$-entry of $(A_n)^{m+1}$ will also be positive.

We now show that the $(\aaa,\bb)$-entry of $(A_n)^m$ is positive if and only if $\aaa \trianglerighteq_z \bb$ for some $z\leq m$. We do so by induction on $m$. When $m=1$, this follows from the definition of $A_n$. Now assume that the statement holds up to $m-1$, and we will show the same for $m$. Certainly, by the preceding remark, we know that every $(\aaa,\bb)$-entry of $(A_n)^m$ where $\aaa \trianglerighteq_z \bb$ for some $z\leq m-1$ is positive.

Now, consider an $(\aaa,\bb)$-entry of $(A_n)^m$ where $\aaa \trianglerighteq_m \bb$ and let $M$ be the matching of the zeros of $\aaa$ and $\bb$ as described in Definition~\ref{def:zdomination}. Let $M_m := \{(i,j) \in M \mid i-j=m\}$. Note that since $\aaa \trianglerighteq_m\bb$, $M_m\neq \emptyset$: there must exist $a_i=0$ and $b_j=0$ that are matched in $M$ such that $i-j=m$. Note that any matched pair of zeros in $M\backslash M_m$ is at most $m-1$ away from one another. Furthermore, for some $(i,j)\in M_m$, if $b_{j + 1}=0$, then $b_{j+1}$ cannot be unmatched in $M$ otherwise $a_{i}$ would be matched to $b_{j+1}$ instead of to $b_j$. Since $\aaa\trianglerighteq_m\bb$, $b_{j+1}$ cannot be matched to an entry after $a_{i+1}$ as such an entry would be at least $m+1$ away. Thus, if $b_{j+1}=0$ for some $j$ such that $(i,j)\in M_m$, we must have that $(i+1, j+1)\in M_m$ as well.

Let $\bb'$ be the vector such that $b'_k=b_k$ for every $k\not\in \{j \mid (i,j)\in M_m\}\cup \{j+1 \mid (i,j)\in M_m\}$. For every maximally consecutive sequence of $m$-matched zeros in $\bb$, say $j, j+1, \ldots, j+d$ where $(i, j), (i+1, j+1), \dots, (i+d, j+d)\in M_m$, and thus where $b_{j+d+1} =1$ by the previous comments, let $b'_{j}=1$ and $b'_{j+1}=\ldots=b'_{j+d+1}=0$. (Note that we allow $d$ here to be 0.) See Figure~\ref{fig:mdominationmiddle} for an example, and let $M'$ be the matchings of the zeros of $\aaa$ and $\bb'$. Note that if $(i,j)\in M\backslash M_m$, then $(i,j)\in M'$ as well. Moreover, for $(i,j)\in M_m$, we have $(i,j+1)\in M'$. Thus $\aaa \trianglerighteq_{m-1}\bb'$. Furthermore, $\bb' \trianglerighteq_1 \bb$  by the construction.

\begin{figure}
    \centering
    \includegraphics[scale=0.4]{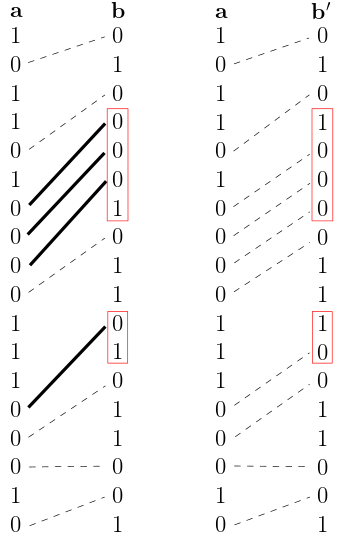}
    \caption{An example where $\aaa$ $3$-dominates $\bb$, and where $\aaa$ $2$-dominates $\bb'$ constructed from $\bb$. The thick lines represent the matched $i_k$ and $j_k$ in $M_3$, i.e., such that $a_{i_k}=b_{j_k}=0$ and $i_k-j_k=3$. The red boxes represent entries $j, j+1, \ldots, j+d+1$ from the proof of Lemma~\ref{lem:zdominationingraph}.}
    \label{fig:mdominationmiddle}
\end{figure}

By our induction hypothesis, we know that entry $(\aaa,\bb')$ in $(A_n)^{m-1}$ is positive. By the definition of $A_n$, the $(\bb',\bb)$ entry in $A_n$ is positive. By multiplying $(A_n)^{m-1}$ with $A_n$, we see that the $(\aaa,\bb)$ entry in $(A_n)^m$ is positive. Equivalently, taking a path of length $m-1$ from $\aaa$ to $\bb'$ and adding the path of length $1$ from $\bb'$ to $\bb$ yields a path of length $m$ from $\aaa$ to $\bb$.

Finally, consider an $(\aaa,\bb)$-entry of $(A_n)^m$ where $\aaa \trianglerighteq_z \bb$ with $z\geq m+1$. In this case, we need to show that the entry is 0. Suppose not. Then there must exist some $\bb'$ such that entry $(\aaa,\bb')$ in $(A_n)^{m-1}$ and entry $(\bb',\bb)$ in $A_n$ are both positive. Therefore, by the induction hypothesis, we know that $\aaa\trianglerighteq_{m-1}\bb'$ and $\bb' \trianglerighteq_1 \bb$. We create a matching $N$ of zeros in $\aaa$ to zeros in $\bb$ that contains $(i,j)$ if $(i,k)$ is in the matching for $\aaa$ and $\bb'$ and $(k,j)$ is in the matching for $\bb'$ and $\bb$. Then we know that $i-k\leq m-1$ and $k-j\leq 1$, and so $i-j\leq m$. Note that $N$ might not be a proper matching for $\aaa$ and $\bb$: there could be unmatched zeros. But in that case, in $N$, we can match zeros in $\aaa$ to even later entries in $\bb$, and so we will still have that $i-j\leq m$. We thus see that $\aaa \trianglerighteq_m\bb$, a contradiction. 
\end{proof}

\begin{theorem}\label{thm:matrixvertices}
For positive integers $n$ and $m$, we have that $v^{(n,m)}_{\textup{unsplit}}(\aaa,\bb)=
((A_n)^{m+1})_{\chi(\aaa),\chi(\bb)}$.
\end{theorem}

\begin{proof}
This follows from  Lemma~\ref{lem:sequences} and the definition of $D_n$ and the observation that follows it.
\end{proof}

\begin{remark}
Note that one can think of $v^{(n,0)}_{\textup{unsplit}}(\aaa,\bb)=(A_n)_{\chi(\aaa),\chi(\bb)}$ as the number of unsplittable flows on  $\mathcal{F}_{H_{\top}(n)}(\aaa,\bb)$ which is always 0 or 1.
\end{remark}

We also can write $v^{(n,m)}_{\textup{unsplit}}(\aaa,\bb)$ as a sum over the entries of the $m$th power of $A_n$.

\begin{corollary} \label{cor:matrix vertices sum}
    We have that $$v^{(n,m)}_{\textup{unsplit}}(\aaa,\bb)=
\sum_{\substack{\uu\in \NN^n\cap W_{\chi(\uu),\bb,\leftarrow}:\\ \aaa \,\trianglerighteq\,  \uu \textup{ and } \chi(\aaa) \,\trianglerighteq\,  \chi(\uu)}} ((A_n)^{m})_{\chi(\aaa),\chi(\uu)}.$$ In particular, $v^{(n,m)}(\aaa)=
\sum_{{\bf j} \in \{0,1\}^n} ((A_n)^{m})_{\chi(\aaa),{\bf j}}$.
\end{corollary}

\begin{proof}
    This follows from Theorems~\ref{thm:fixinglast} and \ref{thm:matrixvertices}.
\end{proof}

\begin{example}

In Table~\ref{tab:genfunctions5ab}, we give the closed forms for the generating functions of $v^{(5,m)}_{\textup{unsplit}}({\bf 1}, \bb)$.
\end{example}

\subsection{Nonnegative combinations of binomial coefficients and generating functions}\label{subsec:binomialcoefficients} \label{subsec:generating functions}

We now explain the presence of $(1-x)^i$ as the denominator of the generating functions given in Appendix~\ref{appendix: tables}, which indicates that for some fixed $n$, $\aaa$ and $\bb$, $v^{(n,m)}(\aaa)$ and $v^{(n,m)}_{{\textup{unsplit}}}(\aaa,\bb)$  are polynomial functions in $m$ and can be written as linear combinations of $\binom{m+j}{i}$'s. We also give  nonnegative expansions of the number $v^{(n,m)}_\textup{unsplit}(\aaa,\bb)$ of unsplittable flows in the basis $\binom{m}{k}$ and for the number $v^{(n,m)}(\aaa,\bb)$ of {\em vertices}  in the basis $\binom{m+1}{k+1}$.

We obtain the following result from Theorem~\ref{thm:matrixvertices} and the transfer-matrix method. Given a matrix $B$, we denote by $(B:j,i)$  the submatrix obtained from $B$ by removing the row and column indexed by $j$ and $i$, respectively. Moreover, we let $I_n$ be the $2^n \times 2^n$ identity matrix.

\begin{corollary} \label{cor: gf and polynomiality}
The number $v^{(n,m)}_{\textup{unsplit}}(\aaa,\bb)$ of unsplittable flows has the following generating function,
\[
1+\sum_{m\geq 0} v^{(n,m)}_{\textup{unsplit}}({\bf a},\bb) x^{m+1} \,=\, \frac{(-1)^{\ind(\chi(\aaa))+\ind(\chi(\bb))}Q_{{\bf a},{\bf b}}(x)}{(1-x)^{2^n}},
\]
where $Q_{{\bf a},{\bf b}}(x)$ is the determinant of the submatrix $(I_n-xA_n:\chi(\bb),\chi(\aaa))$.  In particular, $v^{(n,m)}_{\textup{unsplit}}(\aaa,\bb)$ is a polynomial in $m$ of degree at most $2^n-1$. 
\end{corollary}

\begin{proof}
By Theorem~\ref{thm:matrixvertices} we have that 
\[
1+\sum_{m\geq 0} v^{(n,m)}_{\textup{unsplit}}({\bf a}, \bb) x^{m+1} \,=\,  1+\sum_{m\geq 0} ((A_n)^{m+1})_{\chi(\aaa),\chi(\bb)}x^{m+1}.
\]
The transfer-matrix method \cite[Thm. 4.7.2]{EC1} allows to evaluate the generating function on the right-hand side in terms of determinants as $(-1)^{\ind(\chi(\aaa))+\ind(\chi(\bb))}Q_{\aaa,\bb}(x)/\det(I_n-xA_n)$. Thus,
\[
1+\sum_{m\geq 0} v^{(n,m)}_{\textup{unsplit}}({\bf a}, \bb) x^{m+1} \,=\, \frac{(-1)^{\ind(\chi(\aaa))+\ind(\chi(\bb))}Q_{\aaa,\bb}(x)}{\det(I_n-xA_n)}.
\]
Since the matrix $A_n$ is upper triangular with ones along the diagonal, then $\det(I_n-xA_n)=(1-x)^{2^n}$ as desired. 
\end{proof}

\begin{corollary} \label{cor: gf and polynomiality case b=0}
The number $v^{(n,m)}(\aaa)$ of vertices of $\mathcal{F}_{G(n,m)}(\aaa)$ has the following generating function,
\[
\sum_{m\geq 0} v^{(n,m)}({\bf a}) x^m \,=\, \frac{\sum_{{\bf j}\in \{0,1\}^n} (-1)^{\ind(\chi(\aaa))+\ind({\bf j})} Q_{\aaa,\bf j}(x)}{(1-x)^{2^n}},
\]
where $Q_{\aaa,\bf j}(x)$ is the determinant of the submatrix $(I_n-xA_n:{\bf j},\chi(\aaa))$.  In particular, $v^{(n,m)}(\aaa)$ is a polynomial in $m$ of degree at most $2^n-1$. 
\end{corollary}

\begin{proof}
By Corollary~\ref{cor:matrix vertices sum} and exchanging the order of summations, we have that 
\[
\sum_{m\geq 0} v^{(n,m)}({\bf a}) x^m \,=\, \sum_{\jj\in \{0,1\}^n} \sum_{m\geq 0}  \, x^m ((A_n)^{m})_{\chi(\aaa),\jj}.
\]
The transfer-matrix method \cite[Thm. 4.7.2]{EC1} allows to evaluate each generating function on the right-hand side with fixed ${\bf j}$ in terms of determinants as $(-1)^{\ind(\chi(\aaa))+\ind({\bf j})}Q_{\aaa,\bf j}(x)/\det(I_n-xA_n)$. Thus,
\[
\sum_{m\geq 0} v^{(n,m)}({\bf a}) x^m \,=\, \frac{\sum_{{\bf j}\in \{0,1\}^n} (-1)^{\ind(\chi(\aaa))+\ind({\bf j})} Q_{\aaa,\bf j}(x)}{\det(I_n-xA_n)}.
\]
Since the matrix $A_n$ is upper triangular with ones along the diagonal, then $\det(I_n-xA_n)=(1-x)^{2^n}$ as desired. Note that in Tables \ref{tab:genfunctions5} and \ref{tab:genfunctions5ab}, the powers of the denominators are smaller due to cancellation.
\end{proof}

Another way to see that $v_{\textup{unsplit}}(\aaa,\bb)$  and $v^{(n,m)}(\aaa)$ are polynomials in $m$ is by giving an alternative nonnegative expansion in terms of the polynomial basis $\binom{m+1}{k}$.
For every $\binom{m+1}{k}$ ways of choosing $1\leq k \leq m+1$ columns of the $m+1$ columns in graph $G(n,m)$, count the number of unsplittable flows that have flow descending in exactly the chosen columns. This count will be some nonnegative integer, and thus we see that $v^{(n,m)}(\aaa)$ is a nonnegative combination of $\binom{m+1}{k}$ for $1 \leq k \leq m+1$.

Here is a different interpretation involving a nonnegative combination of binomial coefficients.

\begin{lemma}\label{lem:nonnegativecountpaths}
We have that
$$v^{(n,m)}_{\textup{unsplit}}(\aaa,\bb)=\sum_{k=1}^{2^n-1} ((A_n-I_n)^k)_{\chi(\aaa),\chi(\bb)} \binom{m+1}{k},$$
where $((A_n-I_n)^k)_{\chi(\aaa),\chi(\bb)}\geq 0$ for all $\aaa,\bb$.
\end{lemma}

\begin{proof} 
We have seen that $v^{(n,m)}_{\textup{unsplit}}(\aaa,\bb)=\sum_{j\in\{0,1\}^n}((A_n)^{m+1})_{\chi(\aaa),\jj}$, i.e., one counts the number of paths of length $m+1$  from $\aaa$ to $\bb$ in $D_n$. Recall that $D_n$ has a loop at every vertex, and so paths from $\aaa$ to $\bb$ can be extended by going through loops of the vertices in the path. Let $D^0_n$ be the graph $D_n$ but with its loops removed. Then the $(\aaa,\bb)$-entry of  $(A_n-I_n)^{k}$  counts the number of paths of length $k$  from $\aaa$ to $\bb$ in $D^0_n$ or equivalently the number of paths of length $k$  from $\aaa$ to $\bb$ in $D_n$ that do not include going over loops. 

Every path of length $k$ from $\aaa$ to $\bb$ in $D_n^0$ can be extended to a path of length $m+1\geq k$ from $\aaa$ to $\bb$ in $D_n$ by choosing how to distribute $m+1-k$ extra arcs as loops at the $k+1$ vertices in the original path. Finally, note that the longest path in $D_n^0$ is upper bounded by $2^n-1$, which would only happen if there was a path going through every single vertex. This yields the result.
\end{proof}

\begin{example}
Consider $n=3$. To compute the number of paths in $D_3^0$, look at the following matrices.
$$A_3-I_3=
\begin{pmatrix}
0&1&1&1&1&1&1&1\\
0&0&1&1&0&1&1&1\\
0&0&0&1&1&1&1&1\\
0&0&0&0&0&1&1&1\\
0&0&0&0&0&1&1&1\\
0&0&0&0&0&0&1&1\\
0&0&0&0&0&0&0&1\\
0&0&0&0&0&0&0&0\\
\end{pmatrix} \qquad
(A_3-I_3)^2=
\begin{pmatrix}
0&0&1&2&1&4&5&6\\
0&0&0&1&1&2&3&4\\
0&0&0&0&0&2&3&4\\
0&0&0&0&0&0&1&2\\
0&0&0&0&0&0&1&2\\
0&0&0&0&0&0&0&1\\
0&0&0&0&0&0&0&0\\
0&0&0&0&0&0&0&0\\
\end{pmatrix}$$

$$(A_3-I_3)^3=
\begin{pmatrix}
0&0&0&1&1&4&8&13\\
0&0&0&0&0&2&4&7\\
0&0&0&0&0&0&2&5\\
0&0&0&0&0&0&0&1\\
0&0&0&0&0&0&0&1\\
0&0&0&0&0&0&0&0\\
0&0&0&0&0&0&0&0\\
0&0&0&0&0&0&0&0\\
\end{pmatrix} \qquad (A_3-I_3)^4=
\begin{pmatrix}
0&0&0&0&0&2&6&14\\
0&0&0&0&0&0&2&6\\
0&0&0&0&0&0&0&2\\
0&0&0&0&0&0&0&0\\
0&0&0&0&0&0&0&0\\
0&0&0&0&0&0&0&0\\
0&0&0&0&0&0&0&0\\
0&0&0&0&0&0&0&0\\
\end{pmatrix}$$

$$(A_3-I_3)^5=
\begin{pmatrix}
0&0&0&0&0&0&2&8\\
0&0&0&0&0&0&0&2\\
0&0&0&0&0&0&0&0\\
0&0&0&0&0&0&0&0\\
0&0&0&0&0&0&0&0\\
0&0&0&0&0&0&0&0\\
0&0&0&0&0&0&0&0\\
0&0&0&0&0&0&0&0\\
\end{pmatrix}
\qquad (A_3-I_3)^6=
\begin{pmatrix}
0&0&0&0&0&0&0&2\\
0&0&0&0&0&0&0&0\\
0&0&0&0&0&0&0&0\\
0&0&0&0&0&0&0&0\\
0&0&0&0&0&0&0&0\\
0&0&0&0&0&0&0&0\\
0&0&0&0&0&0&0&0\\
0&0&0&0&0&0&0&0
\end{pmatrix}$$

Therefore, we have
$v^{(3,m)}((1,1,1),(0,0,0))=1\binom{m+1}{1}+6\binom{m+1}{2}+13\binom{m+1}{3}+14\binom{m+1}{4}+8\binom{m+1}{5}+2\binom{m+1}{6}$. 
\end{example}

We give an expansion of $v^{(n,m)}(\aaa,\bb)$ in terms of vertex plane partitions.

\vertexpositivityotherbasis

\begin{proof}

We know from Theorem~\ref{thm:vertexplanepartitions} that the number $v^{(n,m)}(\aaa,\bb)$ of vertices of $\mathcal{F}_{G(n,m)}(\aaa,\bb)$ counts the number of vertex plane partitions of shape $\theta(\aaa,\bb)$.  We can partition the set of these vertex plane partitions according to how many different labels they use. Consider all the vertex plane partitions using $k$ labels. You can further partition these by the $k$ labels among $\{0,1,\ldots, m\}$ that they use. There are $\binom{m+1}{k}$ such classes, and each class contains the same number of vertex plane partitions. Since $\theta(\aaa,\bb)$ has $|\theta(\aaa,\bb)|$ cells, there can be at most that many different labels within a plane partition, and so the result holds.
\end{proof}

\begin{remark}
Note that $|\theta(\aaa,\bb)|$ is not necessarily the degree in $m$ of the polynomial $v^{(n,m)}(\aaa,\bb)$ since for example some columns of the vertex plane partition might be forced to be equal (see Remark~\ref{rem: equal columns vertex pp}). However, for example if $\bb={\bf 0}$ and $\aaa \in \{0,1\}^n$, then the degree is indeed $|\lambda|=|\theta(\aaa,{\bf 0})|$.    
\end{remark}

\begin{example}
Take a path in $D^0_4$, say going through $(1,1,1,1)$, $(1,1,1,0)$, $(0,1,0,1)$, $(0,0,1,1)$, and $(0,0,0,0)$. This path corresponds to some vertex counted by $v^{4,3}(1,1,1,1)$. It also corresponds to a flow in $\mathcal{F}_{G(4,3)}(1,1,1,1)$ and to a vertex plane partition. See Figure~\ref{fig:binomialpresence0}.

\begin{figure}[h!]
    \centering
    \includegraphics[scale=0.7]{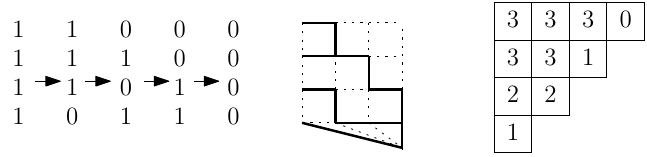}
    \caption{Path of length 4 in $D^0_4$ and the corresponding flow in $\mathcal{F}_{G(4,3)}((1,1,1,1),(0,0,0,0))$  and vertex plane partition.}
    \label{fig:binomialpresence0}
\end{figure}

It will also generate $\binom{10}{4}$ vertices counted by $v^{(4,9)}(1,1,1,1)$. For example, $(1,1,1,1)$ and $(0,0,0,0)$ could be visited twice, and $(0,1,0,1)$, $(0,0,1,1)$ thrice each so that the total length of the path would become $10$, or instead, $(0,0,0,0)$ could be visited a total number of seven times (see the first column of Figures~\ref{fig:binomialpresenceA} and \ref{fig:binomialpresenceB} respectively). These correspond to flows in $\mathcal{F}_{G(4,9)}(1,1,1,1)$ represented in the second column of Figures~\ref{fig:binomialpresenceA} and \ref{fig:binomialpresenceB} respectively, which can be thought of as the original flow in Figure~\ref{fig:binomialpresence0} that has been stretched out, i.e., where the descents happen in different columns of the graph.

Using notation from Theorem~\ref{thm:vertexplanepartitions}, there were four columns in which the descents happened in Figure~\ref{fig:binomialpresence0}: $d_{44}=0$, $d_{11}=d_{33}=1$, $d_{12}=d_{22}=2$, $d_{13}=d_{14}=d_{23}=d_{24}=d_{34}=3$. In  Figure~\ref{fig:binomialpresenceA}, we change these to $d_{44}=1=0+1$, $d_{11}=d_{33}=2=1+1$, $d_{12}=d_{22}=5=2+3$, $d_{13}=d_{14}=d_{23}=d_{24}=d_{34}=8=3+5$, and we keep them the same as in Figure~\ref{fig:binomialpresence0} for Figure~\ref{fig:binomialpresenceB}. Note that we are simply adding the number of loops seen by the time we come to a particular element.   The descents can be recorded in the vertex plane partition, and so we see that deciding where to extend the path or where to stretch out the flow corresponds to deciding which labels to use. The original labels $0,1,2,3$ in Figure~\ref{fig:binomialpresence0} are sent respectively to $1,2,5,8$ in the Figure~\ref{fig:binomialpresenceA}, and to $0,1,2,3$ in Figure~\ref{fig:binomialpresenceB}.  

\begin{figure}[h!]
  \begin{subfigure}[b]{0.5\textwidth}
    \centering
    \includegraphics[scale=0.7]{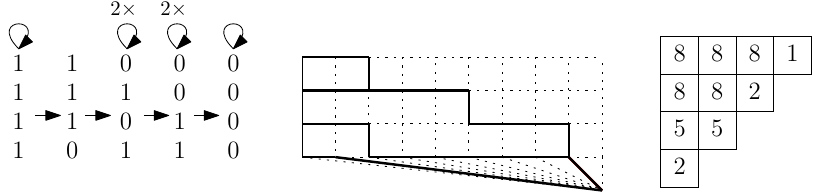}
    \caption{}
    \label{fig:binomialpresenceA}
\end{subfigure}

\begin{subfigure}[b]{0.5\textwidth}
    \includegraphics[scale=0.7]{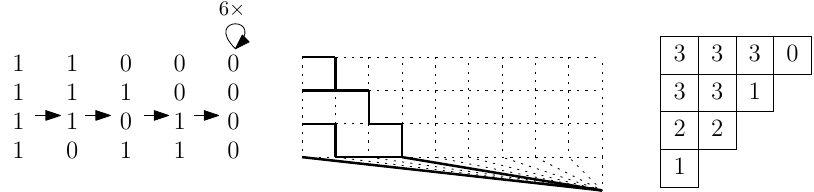}
    \caption{}
    \label{fig:binomialpresenceB}
\end{subfigure}
    \caption{Two examples of extending a path of length 4 in $D^0_4$ to be a path of length $10$ in $D_4$ versus choosing where the descents occur in the flow versus changing the labels in vertex plane partitions.}
    \label{fig:binomialpresence}
\end{figure}
\end{example}

\begin{corollary}
For $\aaa,\bb\in \NN^n$, we have that the generating function for the number $v^{(n,m)}(\aaa,\bb)$ of vertices of $\mathcal{F}_{G(n,m)}(\aaa,\bb)$ satisfies
\begin{equation} \label{eq: gf vertices skew case}
\sum_{m\geq 0} v^{(n,m)}(\aaa,\bb) x^m = \frac{\sum_{i=0}^d c_ix^i}{(1-x)^{d+1}},
\end{equation}
where $d$ is the degree in $m$ of $v^{(n,m)}(\aaa,\bb)$ and $c_i = \sum_{k=i-1}^d p_{\aaa,\bb,k} \binom{d-k}{i+1-k}(-1)^{i+1-k}$.
\end{corollary}

\begin{proof}
By Corollary~\ref{lem:positivity other basis}, we have that for fixed $n$, $\aaa$, and $\bb$, the number $v^{(n,m)}(\aaa,\bb)$ is a polynomial in $m$ of degree $d\leq |\theta(\aaa,\bb)|$. By standard facts of generating functions \cite[Thm. 4.1.1]{EC1}, this implies that the generating function on the left-hand side of \eqref{eq: gf vertices skew case} has the form $C(x)/(1-x)^{d+1}$  where $C(x)$ is a polynomial of degree less than $d$. By standard facts from linear algebra (see \cite[Thm. 3.18]{CCD}), we have that 
\[
v^{(n,m)}(\aaa,\bb) = \sum_i c_i \binom{m+d-i}{d},
\]
where $c_i$ are the coefficients of $C(x)$. Corollary~\ref{lem:positivity other basis} gives an expansion of $v^{(n,m)}(\aaa,\bb)$ in the polynomial basis $\binom{m+1}{k}$ for $k \in \{0,\ldots,m+1\}$. Using the change of bases $\binom{m+1}{k}=\sum_{i=0}^d (-1)^{i+1-k} \binom{d-k}{i+1-k} \binom{m+d-i}{d}$ between this basis and the basis $\binom{m+d-i}{d}$ gives the desired result. 
\end{proof}

\begin{example}
In Table~\ref{tab:genfunctions5abvert}, we give the closed forms for the generating functions of $v^{(5,m)}({\bf 1},\bb)$.
\end{example}

\subsection{The leading term of $v^{(n,m)}(\aaa,\bb)$}

The leading coefficient of  the polynomial $v^{(n,m)}(\aaa,\bb)$ of degree $d$ is of independent interest. By equation \eqref{eq: vertex skew case other basis}, up to a factor of $1/d!$, it counts vertex plane partitions with distinct entries (when those exist). We give a name to these.

\begin{definition} \label{def: standard vertex plane partitions}
Let $n$ be a nonnegative integer, and $\aaa,\bb\in \NN^n$. A \defn{standard vertex plane partition} of shape  $\theta(\aaa,\bb)=\lambda/\mu$ is a vertex plane partition of shape $\lambda/\mu$ with distinct entries $0,1,\ldots |\lambda/\mu|-1$. They are counted by $p_{\aaa,\bb,|\lambda/\mu|}$.
\end{definition}

From Definition~\ref{def:vertex plane partitions general}, we can give a characterization of standard vertex plane partitions.

\begin{proposition} \label{prop:char standard vertex pp}
Let $n$ be a nonnegative integer, and $\aaa,\bb\in \NN^n$. A plane partition of shape $\theta(\aaa,\bb)=\lambda/\mu$ with distinct entries $0,1,\ldots |\lambda/\mu|-1$ is a vertex plane partition if and only if the following conditions hold: 
\begin{itemize}
\item[(i)] There is no index $j$ such that $\mu'_j=\mu'_{j+1}$ and $\lambda'_j=\lambda'_{j+1}$.
    \item[(ii)] {For columns $j$ and $j+1$ such that $\mu'_j = \mu'_{j+1}$:} for every $i=\mu'_{j+1}+2,\ldots,\lambda'_{j+1}$,  $\pi_{i,j+1} <\pi_{i+1,j}$.
    \item[(iii)] {For columns $j$ and $j+1$ such that $\lambda'_j=\lambda'_{j+1}$:} for every  $i=\mu'_{j+1}+2,\ldots,\lambda'_{j+1}$, $\pi_{i,j+1} <\pi_{i+1,j}$.
    \item[(iv)] {For columns $j$ and $j+1$ such that  $\mu'_j \neq \mu'_{j+1}$ and $\lambda'_j\neq\lambda'_{j+1}$:} there is at most one row index $i$ for $\mu'_j \leq i \leq \lambda'_j-1$ such that  $\pi_{i,j}< \pi_{i-1, j+1}$.
\end{itemize}    
\end{proposition}

\begin{proof}
This characterization of standard vertex plane partitions follows from Definition~\ref{def:vertex plane partitions general} and the fact that the entries in such plane partitions are distinct.
\end{proof}

As we show next, in the case of $\aaa \in \{0,1\}^n$ and $\bb={\bf 0}$, the number $p_{\aaa,{\bf 0}, |\lambda|}$ counts standard tableaux of shifted shape (see Section~\ref{sec:shifted SYT}). Given the shape $\lambda(\aaa)$, let \defn{$\nu$} be the shifted shape with parts given by the columns of $\lambda$, i.e. $\nu$ is the {\em conjugate} $\lambda'$ viewed as a shifted shape (see Figure~\ref{fig:ex_shifted}). Recall that $\SYT(\nu)$ denotes the number of standard tableaux of shifted shape $\nu=\lambda'$. This number of tableaux can be computed via the shifted hook-length formula (see Theorem~\ref{thm: hlf shifted shapes}).

\leadingtermvertices

\begin{proof}
The leading term of $|\lambda|!\cdot v^{(n,m)}(\aaa,{\bf 0})$ is the number $p_{\aaa,{\bf 0},|\lambda|}$ of standard vertex plane partitions of shape $\lambda=\lambda(\aaa)$.
By Proposition~\ref{prop:char standard vertex pp}, standard vertex plane partitions have entries that are strictly decreasing along rows and columns, and where southwest-northeast diagonals are strictly decreasing as well. We claim that these standard vertex plane partitions are in bijection with shifted standard Young tableaux (shifted SYT) of shifted shape $\lambda'$. Since $\aaa \in \{0,1\}^n$, then $\lambda(\aaa)$ has columns of distinct lenghts and so $\lambda'$ is a strict partition.  Given a standard vertex plane partition $\pi$ of shape $\lambda$, let $T$ be the shifted tableau of shifted shape $\lambda'$ whose entries in the $i$th row are the complement (i.e., $|\lambda|-c$ where $c$ is the entry) of the entries in the $i$th column of $\pi$. In the schematic to represent both of these objects in Figure~\ref{fig:schematicshifted}, we see that there is a bijection if we send $A$ to $A'$, $B$ to $B'$, $C$ to $C'$, and take complements. See Figure~\ref{fig:ex_shifted} for a concrete example. This map is a bijection.
\end{proof}

In the case of $\aaa={\bf 1}$ and $\bb={\bf 0}$, the number of standard vertex plane partitions has a product formula.

\begin{corollary} \label{lem:top coeff SYT of shifted shape}
We have that $p_{\mathbf{1},\mathbf{0},\binom{n+1}{2}} = \binom{n+1} {2}!\cdot(1!\cdot2!\cdots(n-1)!)/(1!\cdot 3! \cdots (2n-1)!).$
\end{corollary}

\begin{proof}
For $\aaa={\bf 1}$ and $\bb={\bf 0}$, the shape $\lambda({\bf 1})=\delta_n$ and its associated shifted shape is also the shifted staircase $\delta^*_n$. By Theorem~\ref{thm:leading term v straight shape} we have that $p_{{\bf 1},{\bf 0}, \binom{n+1}{2}} = \SYT(\delta^*_n)$. Since it is known (e.g., see \cite[Ex. 3.12.21]{BS} and \cite[\href{https://oeis.org/A003121}{A003121}]{oeis}) that this number of shifted SYT  is $\binom{n+1}{2}!\cdot(1!\cdot2!\cdots(n-1)!)/(1!\cdot 3! \cdots (2n-1)!)$, the result holds.
\end{proof}

 \begin{figure}[h!]
     \begin{subfigure}[b]{0.4\textwidth}
     \centering
     \includegraphics[scale=0.6]{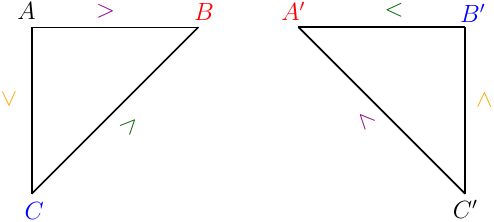}
     \caption{}
     \label{fig:schematicshifted}
     \end{subfigure}
     \begin{subfigure}[b]{0.5\textwidth}
     \centering
 %    \raisebox{5pt}{
%          \includegraphics[scale=0.8]{pp2shsytc.pdf}}
          \raisebox{5pt}{
          \includegraphics[scale=0.8]{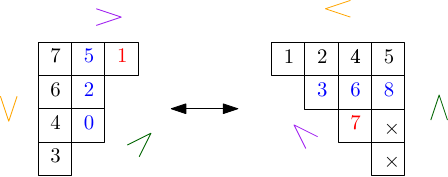}}
     \caption{}
     \label{fig:ex_shifted}
     \end{subfigure}
     \caption{(a) Schematic of a standard vertex plane partition and a shifted standard Young tableau of staircase shape. (b) Example of correspondence between a standard vertex plane partitions for $\aaa=(1,1,0,1)$ and a shifted SYT.}
\end{figure}

\begin{example}
An example of the bijection between vertex plane partitions with distinct entries and shifted SYT of staircase shape is illustrated in Figure~\ref{fig:ex_shifted}.
\end{example}

\begin{remark}
In the case of general $\bb$, the number $p_{\aaa,{\bf b},|\lambda/\mu|}$ of standard vertex plane partitions of shape $\lambda/\mu=\theta({\bf a},{\bf b})$ equals the number of  linear extensions of a certain poset with the additional restriction of condition~(iv) from Proposition~\ref{prop:char standard vertex pp}. See Figure~\ref{fig: poset of std vertex plane partitions}.
\end{remark}

\begin{figure}
   \begin{subfigure}[b]{0.22\textwidth}
   \includegraphics[scale=0.6]{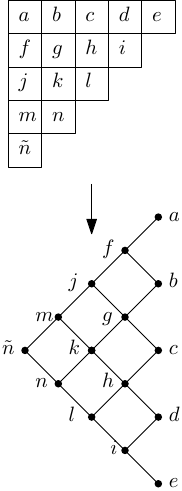}
        \caption{}
    \label{}
    \end{subfigure}   
  \begin{subfigure}[b]{0.22\textwidth}
   \raisebox{27pt}{\includegraphics[scale=0.6]{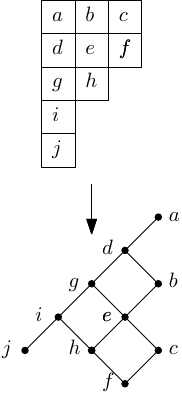}}
        \caption{}
    \label{}
    \end{subfigure}  
    \begin{subfigure}[b]{0.22\textwidth}
        \raisebox{20pt}{\includegraphics[scale=0.6]{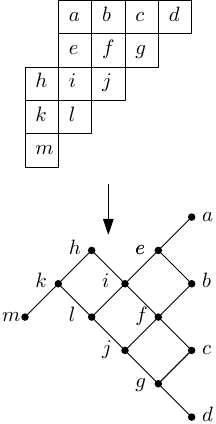}}
        \caption{}
    \label{}
    \end{subfigure}   
    \begin{subfigure}[b]{0.22\textwidth}
        \raisebox{30pt}{\includegraphics[scale=0.6]{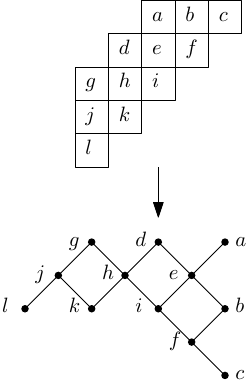}}
    \caption{}
    \label{}
    \end{subfigure}
\caption{Schematic of the poset associated to the shapes (a) $\theta({\bf 1},{\bf 0})$, (b) $\theta((1,0,1,1,0),{\bf 0})$, (c) $\theta({\bf 1},(0,0,0,1,0))$, and (c) $\theta({\bf 1},(0,0,0,1,1))$. The number of standard vertex plane partitions of these shapes counts certain  linear extensions of these posets (with the restrictions in Proposition~\ref{prop:char standard vertex pp}).}
\label{fig: poset of std vertex plane partitions}

\end{figure}

\section{Enumeration of $d$-faces} \label{sec: enumeration d faces}

In this section, we use a similar recurrence to the one in Lemma~\ref{lem:vertexmult},  where we fixed the flow in a column of $G(n,m)$, to give a recurrence for the number of higher-dimensional faces of $\PS_n^m({\bf a})\equiv \mathcal{F}_{G(n,m)}({\bf a})$. Throughout this section, we work under the assumption that ${\bf b} = {\bf 0}$.

 Let $f_d^{(n,m)}({\bf a})$ be the number of $d$-dimensional faces  of $\mathcal{F}_{G(n,m)}({\bf a})$. By Theorem~\ref{thm:Hille}, such faces correspond to $\aaa$-valid connected subgraphs of $G(n,m)$ with Betti number $d$.  By abuse of notation, in this section we refer to faces by their corresponding $\aaa$-valid subgraphs. For ${\bf u}\in \{0,1\}^n$ and $0 \leq c \leq m-1$, let $f_{d,\text{split}}^{(n,c)}({\bf a},{\bf u})$ be the number of {\bf a}-valid connected subgraphs $H$ of $G(n,c+1)$ with $\beta_1(H)=d$ and support ${\bf u}$ in the $c$th column of horizontal edges of $G(n,c+1)$. Note that $f_{d,\text{split}}^{(n,m)}({\bf a},{\bf 0})=f_d^{(n,m)}(\aaa)$ since fixing the flow to be zero on the horizontal edges of the last column of $G(n,m+1)$ brings us back to $G(n,m)$.  Next, we define an analogue of a split in a flow (see Definition~\ref{def:split and merge}) for such subgraphs.

\begin{definition}
Given a subgraph $H$ of $G(n,m)$, a \defn{split} in $H$ is a vertex with outdegree equal to $2$.    
\end{definition}

\begin{lemma}
\label{lem::f_vector_recurrence}
Fix positive integers $n,m$, and let $0 \leq c \leq m-1$. Then
\begin{equation} \label{eq:f_vector recurrence}
    f_d^{(n,m)}({\bf a}) = \sum_{{\bf u} \in \{0,1 \}^n} \sum_{i = 0}^d f_i^{(n, c)}\left( {\bf a},{\bf u} \right)\cdot f_{d - i, \text{split}}^{(n, m-c-1)}({\bf u}).
\end{equation}
\end{lemma}

\begin{proof}
 We will follow the same approach as in the proof of Lemma~\ref{lem:vertexmult}. Let ${\bf u} \in \{0,1\}^n$ and consider a $d$-dimensional face of $\mathcal{F}_{G(n,m)}({\bf a})$ indexed by an ${\bf a}$-valid subgraph $H$ with support ${\bf u}$ in the $c$th column of edges. Since $\beta_1(H)=d$, then $H$ has $d$ {\em fundamental cycles}, i.e., the elements of the bases of cycles of the {\em cycle space} of a graph.

 Since ${\bf b} = {\bf 0}$, it follows that $s$ is the only sink of $H$ and hence that every fundamental cycle of $H$ corresponds to a split in $H$. As illustrated in Figures~\ref{fig: unsplittable col c} and~\ref{fig: unsplittable col c left}, any such subgraph $H$ of $G(n,m)$ having $d$ splits can be decomposed uniquely into a pair of valid subgraphs, $(H_1, H_2)$, where $H_1\subseteq G(n,c)$ and $H_2\subseteq G(n, m-c - 1)$  (see Figure~\ref{fig: unsplittable col c right}). (Recall that if $c=0$, we can take $H_\ulcorner(n)$ instead of $G(n,0)$.)  Moreover, if the number of splits in $H_1$ is $i$ and the number of splits in $H_2$ is $d-i$ for some $i=0,\ldots,d$, then there are $f_i^{(n, c)}\left( {\bf a},{\bf u} \right)\cdot f_{d - i, \text{split}}^{(n, m-c-1)}({\bf u})$ such pairs $(H_1,H_2)$. Note that the number of $\aaa$-valid subgraphs of $G(n,c+1)$ with support $\uu$ in column $c$ that have $i$ splits is equal to the number of $i$-dimensional faces in $\mathcal{F}_{G(n,c+1)}(\aaa,\bf{0})$ where the flow in the $c$th column is $\uu$, which is $f_i^{(n, c)}\left( {\bf a},{\bf u}\right)$ by definition.  Summing over all possible $i$'s and all possible support vectors ${\bf u}$ across column $c$ gives the result.  
\end{proof}

\begin{remark}
Note that some of the summands in \eqref{eq:f_vector recurrence} might be zero.    
\end{remark}

In the next theorem, we obtain a more explicit recurrence by fixing $c = 0$ in the previous lemma. To do this, we will also need the following notion of breaking one vector into the block structure of another.

\begin{definition}
Let ${\bf j}$ and ${\bf a}$ be vectors of the same length $\ell$, where $\aaa$ has signature $\sgn({\bf a})=(b_1-1,b_2-1,\ldots,b_{\ell}-1)$. Then the \emph{block structure of ${\bf j}$ according to ${\bf a}$}, denoted $\sgn_{{\bf a}}({\bf j})$, is the tuple $(B_1, \ldots , B_{\ell})$ where $B_1$ consists of the first $b_1$ elements of ${\bf j}$, $B_2$ consists of the next $b_2$ elements of ${\bf j}$ and so on.
\end{definition}

\begin{example}
 Let ${\bf a} = (1,1,0,1,0,0,1,0,1)$ and suppose ${\bf j} = (1,0,0,0,1,1,1,0,0)$. Since the signature of ${\bf a}$ is $\sgn({\bf a})=(0,1,2,1,0)$, then the block structure of ${\bf j}$ according to  ${\bf a}$ is the tuple $\sgn_{{\bf a}}({\bf j}) = ( (1), (0,0), (0,1,1), (1,0), (0) )$. 
\end{example}

\begin{theorem}
\label{thm::recursive_formula_on_faces}
Let $n$ and $m$ be positive integers, ${\bf a} \in \mathbb{N}^n$, and $d\in\{0,1,\ldots,nm\}$. Then for $m> 1$, we have
\begin{equation} \label{eq:recurrence faces first column PS}
    f_d^{(n,m)}({\bf a}) = \sum_{{\bf j} \in \{0,1 \}^n } \sum_{k = 0}^d \binom{\beta_{{\bf j}, {\bf a}}}{k} f_{d - k - \gamma_{{\bf j}, {\bf a}}}^{(n, m-1)}({\bf j}),
\end{equation}
where $\beta_{{\bf j}, {\bf a}}$ is the number of nonzero blocks of $\sgn_{{\bf a}}({\bf j})$, $\gamma_{{\bf j}, {\bf a}} := b_{{\bf j}} - \beta_{{\bf j}, {\bf a}}$, and  $b_{{\bf j}}$  is the total number of $1$'s appearing in ${\bf j}$. For $m=1$, we have 
\begin{equation} \label{eq:recurrence faces base case}
f_{d}^{(n,1)}({\bf j}) = [x^{d+
k}] \prod_{i=1}^k ((x+1)^{c_i+1}-1)
\end{equation}
where $(c_1,\ldots,c_k):=\sgn({\bf j})$.
\end{theorem}

\begin{proof}
 Setting $c = 0$ in Lemma \ref{lem::f_vector_recurrence}, we obtain
\begin{align} \label{eq:faces d simensional case first sum}
    f_d^{(n,m)}({\bf a}) = \sum_{{\bf j} \in \{0,1 \}^n} \sum_{i = 0}^d f_i^{(n,0)}\left( {\bf a}, {\bf j} \right)\cdot f_{d - i,\text{split}}^{(n, m-1)}({\bf j}).
\end{align}

As mentioned in the proof of Lemma~\ref{lem::f_vector_recurrence}, the index $i$ keeps track of the number of splits contributing to the dimension of each face that are coming from the column indexed by $c=0$ in $G(n,m)$. Instead of indexing the sum in this way, we can sum over a smaller set by observing that much of the information about the number of splits in the graph is contained in the data of the vectors ${\bf j}$ and ${\bf a}$ themselves. 

To make this precise, let $H \subseteq G(n,m)$ be a subgraph contributing to a face counted by  $f_d^{(n,m)}({\bf a})$. Then by equation~\eqref{eq:faces d simensional case first sum}, there exists a vector ${\bf j} \in \{0,1\}^n$ and an integer $0 \leq i \leq d$ for which $H$ contributes to the term $f_{i}^{(n,0)}\left( {\bf a}, {\bf j} \right)\cdot f_{d - i}^{(n, m-1)}({\bf j})$ in the sum. Consider the tuple $\sgn_{{\bf a}}({\bf j})$ of blocks of ${\bf j}$ according to ${\bf a}$. If one of these blocks  has no $1$'s appearing, then it is not contributing any splits to $H$. Otherwise, if a block contains $\ell$ $1$'s, then it must be contributing $\ell - 1$ splits to $H$ since the last  $1$ will have outdegree 1, whereas the other 1's will have outdegree 2.    See Example~\ref{ex::recursive_formula_faces_example}. 

Letting $\beta_{{\bf j}, {\bf a}}$ and $\gamma_{{\bf j}, {\bf a}} $ be defined as in the theorem statement, the preceding argument shows that $\beta_1(H) \geq \gamma_{{\bf j}, {\bf a}} $. Hence the only extra information we need to determine the exact value of $\beta_1(H)$ is how many of the nonzero blocks have edges between one another. Letting $k$ denote the number of edges used to connect nonzero blocks (drawn in green in Figure~\ref{fig::lemma_6_2_example}), we see that there are exactly  $\binom{\beta_{{\bf j}, {\bf a}}}{k}$ ways to connect the nonzero blocks and hence this many ways to choose $H$ so that it has $k + \gamma_{{\bf j}, {\bf a}}$ splits in its first column (in other words, $i = k + \gamma_{{\bf j}, {\bf a}}$).

Summing over this index $k$ instead of $i$ in the right-hand side of equation~\eqref{eq:faces d simensional case first sum} gives the following equivalence (still keeping ${\bf j}$ fixed):
\[
\sum_{i = 0}^d f_i^{(n,0)}\left( {\bf a}, {\bf j} \right)\cdot f_{d - i}^{(n, m-1)}({\bf j}) = \sum_{k = 0}^d \binom{\beta_{{\bf j}, {\bf a}}}{k} f_{d - k - \gamma_{{\bf j}, {\bf a}}}^{(n, m-1)}({\bf j}).
\]

Finally, summing over all possible ${\bf j}$'s gives the desired sum of equation~\eqref{eq:recurrence faces first column PS}.

For the base case $m=1$, $f_d^{(n,1)}({\bf a})$  is the number of $d$-faces of the original Pitman--Stanley polytope $\mathcal{F}_{G(n,1)}({\bf j})$, which is a product of simplices by Theorem~\ref{thm:PS faces}. The formula~\eqref{eq:recurrence faces base case} follows from this result.
\end{proof}

\begin{example}
\label{ex::recursive_formula_faces_example}
 As an example of the formation of cycles in the proof of Theorem~\ref{thm::recursive_formula_on_faces}, suppose $\chi({\bf a}) = (1,0,0,1,0,0,0,0)$ and  ${\bf j} = (0,1,1,0,1,0,1,1)$. Then $\sgn_{\bf a}({\bf j}) = (B_1, B_2) = ((0,1,1), (0,1,0,1,1))$. The block $B_1$  contains two $1$'s and consequently contributes one split to the subgraph (depicted in red in Figure~\ref{fig::lemma_6_2_example}). The block $B_2$ contains three $1$'s and hence contributes two splits to the subgraph (one depicted in turquoise and one depicted in gold in Figure~\ref{fig::lemma_6_2_example}).
 
 Hence in any subgraph whose flow along the first column of the graph has support ${\bf j}$, the first column must be contributing a minimum of three cycles to the whole subgraph. The exact number of splits that the first column will contribute to the whole subgraph will be determined by the number of edges connecting the two blocks. In Figure~\ref{fig::lemma_6_2_example}, the $\binom{2}{k}$ possibilities for connecting the blocks are depicted with edges in green for $k = 0$ (left), $k = 1$ (middle) and $k = 2$ (right).  
 
 \begin{figure}[h]
     \centering
     \includegraphics[width = \textwidth]{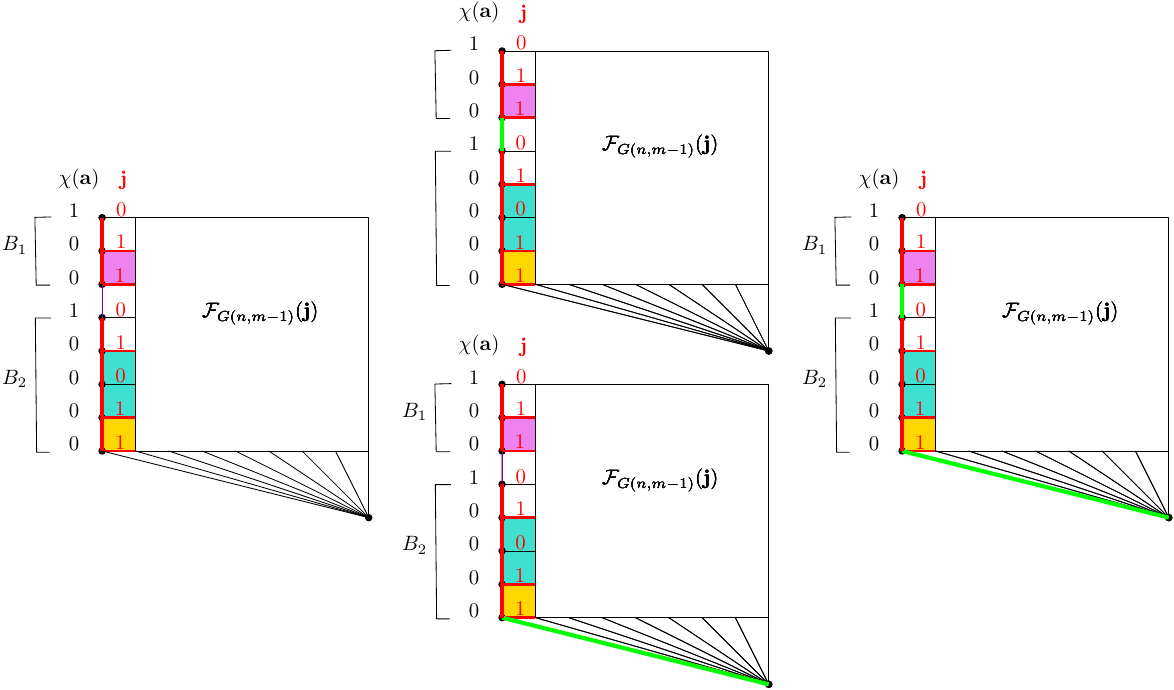}
     \caption{An illustration depicting how cycles are created when counting $d$-faces in the proof of Theorem~\ref{thm::recursive_formula_on_faces}. The first column of the subgraph contributes a minimum of $3$ cycles to the whole subgraph, thus the subgraph on the rest of the graph corresponds to a $(d-3 - k)$-face in $\mathcal{F}_{G(n,m-1)}({\bf j})$ for some $k \in \{0,1,2 \}$.}
     \label{fig::lemma_6_2_example}
 \end{figure}
\end{example}

The case $\chi({\bf a})={\bf 1}$ of the result above yields a simpler recurrence.

\begin{corollary}
For ${\bf a}$ in $\mathbb{N}^n$ with $\chi({\bf a})={\bf 1}$, we have that 
\[
f_d^{(n,m)}({\bf 1}) = \sum_{ {\bf j}  \in \{0,1\}^n} \sum_{k = 0}^{d}
\binom{b_{\bf j}}{k} f_{d-k}^{(n,m-1)}({\bf j}),
\]
where $b_{\bf j}$ is the number of ones appearing in ${\bf j}$. 
\end{corollary}

\begin{example}
For $m=n=2$, we have 
\begin{multline*}
f_d^{(2,2)}(1,1) = \left(f_{d}^{(2,1)}(1,1)+2\cdot f_{d-1}^{(2,1)}(1,1)+f_{d-2}^{(2,1)}(1,1)\right)+\\ +\left(f_{d}^{(2,1)}(1,0)+f_{d-1}^{(2,1)}(1,0)\right) + \left(f_{d}^{(2,1)}(0,1)+f_{d-1}^{(2,1)}(0,1)\right) + f_d^{(2,1)}(0,0).
\end{multline*}
Using this and the base case \eqref{eq:recurrence faces base case}, we obtain the $f$-vector $(1, 10, 21, 18, 7, 1)$ of $\mathcal{F}_{G(2,2)}(1,1)$. 
\end{example}

\section{Final remarks} \label{sec: final remarks}

\subsection{An equivalent formulation on a grid graph}

The graph $G(n,m)$ is almost a grid and its lack of symmetry under reversing the graph can be fixed by defining an integrally equivalent polytope on a grid.

\begin{definition}\label{def:hnm}
Let $\defn{H(n,m)}$ be the directed $(n+1)\times (m+1)$ grid graph with vertex set $V=\{(i,j) \mid 1\leq i \leq n+1, 0 \leq j \leq m\}$ and edges $((i,j),(i,j+1))$ and $((i,j),(i+1,j))$ for $i\in\{1, 2, \dots, n\}$ and $j\in\{0, 1, \dots, m-1\}$. 
\end{definition}

See Figure~\ref{fig:flowH34} for an example of the graph $H(n,m)$.

\begin{definition}
For vectors ${\bf a'}=(a_1,\ldots,a_{n+1})$ and ${\bf b'}=(b_1,\ldots,b_{n+1})$ with $\sum_i a_i=\sum_i b_i$, let \defn{$\mathcal{F}_{H(n,m)}({\bf a'},{\bf b'})$} be the flow polytope on the graph $H(n,m)$ with the following netflows:
\begin{itemize}
    \item vertices $\{(1,1),\ldots,(n+1,0)\}$ have corresponding netflow $a_1,\ldots,a_{n+1}$, and
    \item vertices $\{(1,m),\ldots,(n+1,m)\}$ have corresponding netflow $-b_1,\ldots,-b_{n+1}$.
\end{itemize}
\end{definition}

\begin{remark}\label{rem:othergridpartitions}
By the same argument as in the proof of Theorem \ref{thm:bijflowpp}, the lattice points of $\mathcal{F}_{H(n,m)}({\bf a'},{\bf b'})$ are in bijection with plane partitions of shape $\theta({\bf a'},{\bf b'})$ with entries at most $m$. 
\end{remark}

One nice property of the grid graph $H(n,m)$ is that its reverse $H(n,m)^r$ is isomorphic to itself. We thus obtain the following integrally equivalent flow polytopes. Given a vector ${\bf  a}=(a_1,\ldots,a_n)$, let $\rev({\bf a})=(a_n,a_{n-1},\ldots,a_1)$ denote the reverse vector.

\begin{proposition} \label{prop:rev symmetry grid}
For vectors ${\bf a'}=(a_1,\ldots,a_{n+1})$ and ${\bf b'}=(b_1,\ldots,b_{n+1})$ with $\sum_i a_i=\sum_i b_i$, we have that  $\mathcal{F}_{H(n,m)}({\bf a'},{\bf b'})\equiv \mathcal{F}_{H(n,m)}(\rev({\bf b'}), \rev({\bf a'}))$.
\end{proposition}

\begin{proof}
The result follows from Proposition~\ref{prop: symm fp reversing} and the fact that $H(n,m)^r\cong H(n,m)$.
\end{proof}

\begin{remark}
Combining this with Remark~\ref{rem:othergridpartitions}, we obtain a correspondence between plane partitions with entries at most $m$ of shapes $\theta({\bf a'},{\bf b'})$ and shape $\theta(\rev({\bf b'}),\rev({\bf a'}))$. This correspondence can be seen directly by rotating the shape $\theta((b_{n+1}, b_n, \ldots, b_0),(a_{n+1}, a_n, \ldots, a_0))$ by $180$ degrees.
\end{remark}

The next result shows that the flow polytopes on the graph $G(n,m)$ are integrally equivalent to flow polytopes on  the grid graph $H(n,m)$. See Figures~\ref{fig:flowG34skew}, \ref{fig:flowH34}.

\begin{proposition}
For positive integers $n$ and $m$ and vectors ${\bf a}=(a_1,\ldots,a_n), {\bf b}=(b_1,\ldots,b_n) \in \NN^n$ and ${\bf a'}=(a_1-b_1,\ldots,a_n,0), {\bf b'}=(0,b_2,\ldots,b_n,b_{n+1})$ where $b_{n+1}=\sum_{i=1}^n (a_i-b_i)$, we have $\mathcal{F}_{G(n,m)}({\bf a},{\bf b})\equiv \mathcal{F}_{H(n,m)}({\bf a'},{\bf b'})$.
\end{proposition}

\begin{proof}
Note that both graphs $G(n,m)$ and $H(n,m)$ have an induced subgraph of vertices $(i,j)$ for $1\leq i\leq n, 0 \leq j \leq m+1$ isomorphic to $H(n-1,m)$. Let $\Phi': \mathcal{F}_{G(n,m)}({\bf a},{\bf b}) \to \mathcal{F}_{H(n,m)}({\bf a'},{\bf b'})$ be defined as follows: $f\mapsto f'$ where 
\begin{align*}
f'(e) &= f(e), \text{ for } e \in E(H(n-1,m)), \\
f((n,j-1),s) &= f'((n,j-1),(n+1,j-1))=: y_{n,j-1}, \text{ for } j=1,\ldots,m+1, \\
f'((n+1,0),(n+1,1))&= y_{n,0},\\
f'((n+1,j),(n+1,j+1)) &= f'((n+1,j-1),(n+1,j)) + y_{n,j}, \text{ for } j=1,\ldots,m-1.
\end{align*}
We show that $\Phi'$ gives the desired integral equivalence. Indeed, one can check that the affine map $\Phi'$ is well-defined and a bijection. Lastly, since $\Phi'$ restricted to the flows on the edges of $G(n,m)$ is the identity map, it preserves the lattice. 
\end{proof}

\begin{remark}
We note that we can restrict the vectors ${\bf a'}$ and ${\bf b'}$ to be such that $a_{n+1}=0$ and $b_1=0$ without loss of generality. Indeed, since $\sum_{i=1}^{n+1} a_i=\sum_{i=1}^{n+1} b_i$, the top row of $\theta({\bf a'}, {\bf b'})$ will  always be empty, so it does not matter what value $a_{n+1}$ takes: we may replace $b_{n+1}$ with $b_{n+1}-a_{n+1}$ and set $a_{n+1}=0$. Similarly, if $b_1>0$, then the first $b_1$ columns of the shape $\theta({\bf a'}, {\bf b'})$ are empty, so we may replace $a_1$ with $a_1-b_1$ and set $b_1=0$, and still obtain the same shape (of cells to fill).
\end{remark}

\subsection{Face lattice of $\PS_n^m(\aaa)$}
Recall that the  original Pitman--Stanley polytope $\PS_n^1({\bf a})$ is combinatorially equivalent to a product of simplices (see Theorem~\ref{thm:PS faces}). In Section~\ref{sec:enumeration vertices} and Section~\ref{sec: enumeration d faces}, we gave recurrences for the number of vertices and faces of the generalized Pitman--Stanley polytope $\PS_n^m({\bf a})$. A natural extension of this is to study the face lattice of this polytope. The first named author will study this lattice in future work.

\subsection{Closed forms of generating functions of $v^{(n,m)}(\aaa,\bb)$}

In Section~\ref{subsec:generating functions}, we studied the generating functions for the sequence  $\left(v^{(n,m)}(\aaa,\bb)\right)_{m\geq 1}$ of the number of vertices of $\PS_n^m(\aaa,\bb) \equiv \mathcal{F}_{G(n,m)}(\aaa,\bb)$. These generating functions seem of independent interest.
For example, as illustrated in Table~\ref{tab:genfunctions5abvert}, when $\aaa=(1,\ldots,1)$ and $\bb=(0,1,\ldots,1)$, the closed form of the generating function involves the \defn{Eulerian numbers} $A(n,k)$ that count the number of permutations $w$ of size $n$ with $k-1$ {\em descents} (positions $i$ in $w$ with $w_i>w_{i+1}$). 

\begin{proposition} \label{prop: vertices Eulerian case}
Let $\aaa=(1,\ldots,1) \in \mathbb{N}^n$ and $\bb=(0,1,\ldots,1) \in \mathbb{N}^n$ then the number $v^{(n,m)}({\bf 1},{\bf 0})$ satisfies
\begin{equation} \label{eq: vertices gf Eulerian numbers}
\sum_{m\geq 0} v^{(n,m)}({\bf 1},\bb) x^m \,=\, \frac{\sum_{k=1}^n A(n,k)x^{k-1}}{(1-x)^{n+1}}.
\end{equation}
\end{proposition}

\begin{proof}
For $\aaa=(1,\ldots,1) \in \mathbb{N}^n$ and $\bb=(0,1,\ldots,1) \in \mathbb{N}^n$, the shape $\theta(\aaa,\bb)=\delta_{n}/\delta_{n-1}$ is a diagonal of $n$ boxes. Next, we describe the vertex plane partitions of this shape which characterize the vertices of $\mathcal{F}_{G(n,m)}(\aaa,\bb)$ by Theorem~\ref{thm:vertexplanepartitions}. Since each row and column of shape $\theta(\aaa,\bb)$ has exactly one cell, none of the three restrictions of Definition~\ref{def:vertex plane partitions general} apply, thus the vertex plane partitions are exactly the plane partitions of this shape with entries at most $m$. Since each cell of $\theta(\aaa,\bb)$ is independent then $v^{(n,m)}(\aaa,\bb)=(m+1)^n$. The result then follows by standard properties of Eulerian numbers \cite[Prop. 1.4.4]{EC1}.
\end{proof}

\begin{remark}
\label{rem::integral_equivalence_to_product_of_simplices}
For $\aaa=(1,\ldots,1) \in \mathbb{N}^n$ and $\bb=(0,1,\ldots,1) \in \mathbb{N}^n$, the polytope  $\mathcal{F}_{G(n,m)}(\aaa,\bb)$ is integrally equivalent to $(\Delta_m)^n$, the product of $n$  $m$-simplices $\Delta_m$. Indeed, from equation~\eqref{eq: sum of vertical edges}, one can see that the flows $(y_{i,0},y_{i,1},\ldots,y_{i,m})$ on the vertical edges of the $i$th row of $G(m,n)$ correspond to each simplex $\Delta_{m}$.  See Figure~\ref{fig::integral_equivalence_to_product_of_simplices}. This also explains why in this case $v^{(n,m)}(\aaa,\bb)=(m+1)^n$.
\end{remark}

\begin{figure}[h!]
    \centering
    \includegraphics[width = 0.7\textwidth]{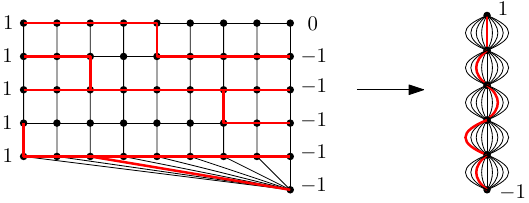}
    \caption{An illustration of the integral equivalence of Remark~\ref{rem::integral_equivalence_to_product_of_simplices} in the case of $n = 5$ and $m = 8$. The flow polytope of the graph on the right is integrally equivalent to $(\Delta_8)^5$.}
    \label{fig::integral_equivalence_to_product_of_simplices}
\end{figure}

It would be of interest to find the closed form for the generating function of $\left(v^{(n,m)}(\aaa,\bb)\right)_{m\geq 1}$ for special cases like the following ones:
\begin{itemize}
    \item[(i)] the staircase shape $\aaa={\bf 1}$ and $\bb={\bf 0}$ so that $\theta(\aaa,\bb)=\delta_{n}$,
    \item[(ii)]  the {\em reverse hook shape} $\aaa=(k,0^{n-1})$ and $\bb=(0,k-1,0^{n-2})$ so that $\theta(\aaa,\bb)=k^n/(k-1)^{n-1}$,
    \item[(iii)] and the {\em zigzag shape} $\aaa={\bf 1}$ and $\bb=(0,0,1^{n-2})$ so that $\theta(\aaa,\bb)=\delta_{n}/\delta_{n-2}$.
\end{itemize}
See Tables~\ref{tab:genfunctions5} and \ref{tab:genfunctions5abvert} for some preliminary data.

\begin{remark}
It is well known that the {\em Eulerian polynomial} $\sum_{k=1}^n A(n,k)x^{k-1} \in \mathbb{N}[x]$ is real-rooted \cite[\S 4.6]{Eulerian_book}. Based on Proposition~\ref{prop: vertices Eulerian case}, it is natural to ask if the numerators of the closed forms of the generating functions for $v^{(n,m)}(\aaa,\bb)$ have nonnegative coefficients or are real-rooted. However, neither of this is true in general. See for example the numerator $1+x-x^2+x^3$ in the entry $\chi(\aaa)=(0,0,1,1,1)$ in Table~\ref{tab:genfunctions5}.
\end{remark}

\begin{remark} \label{rem:reciprocity}
The data in Tables~\ref{tab:genfunctions5} and \ref{tab:genfunctions5abvert} indicates a difference between the degree $d$ of the polynomial $v^{(n,m)}(\aaa,\bb)$ and the degree of the numerator $(1-x)^{d+1}\cdot \sum_{m\geq 0} v^{(n,m)}(\aaa,\bb)x^m$. This in turn indicates that the polynomials $v^{(n,m)}(\aaa,\bb)$ have zeros at some initial negative values (see \cite[Thm. 3.18]{CCD}). It would be of interest to give an interpretation for the value of these polynomials at negative integers.
\end{remark}

\subsection{Generating functions and polynomiality of $f_d^{(n,m)}(\aaa,\bb)$}

In Section~\ref{subsec:generating functions}, we studied the generating functions for the sequence  $\left(v^{(n,m)}(\aaa,\bb)\right)_{m\geq 1}$ of the number of vertices of $\mathcal{F}_{G(n,m)}(\aaa,\bb)$ with $n$ fixed. Also, in  Corollary~\ref{lem:positivity other basis} we showed that the number  $v^{(n,m)}(\aaa,\bb)$  of vertices of $\mathcal{F}_{G(n,m)}(\aaa,\bb)$ is a polynomial in $m$ and expands with nonnegative integer coefficients in the binomial basis $\binom{m+1}{k}$. One could study similar questions for the number $f_d^{(n,m)}(\aaa,\bb)$ of $d$-dimensional faces of $\mathcal{F}_{G(n,m)}(\aaa,\bb)$ for fixed $n$ and $d$. The first named author will study these questions in future work, and in particular will show that $f_d^{(n,m)}(\aaa,\bb)$ is a polynomial in $m$.

\subsection{Other recursions}

We note that there are other interesting recursions for $v^{(n,m)}_{\textup{unsplit}}(\aaa,\bb)$ by fixing the flow on other columns in Lemma~\ref{lem:vertexmult}. For example, by fixing the flow on the penultimate column, we get 

\begin{align*}
v^{(n,m)}_{\textup{unsplit}}(\aaa,\bb)&=\sum_{\substack{\uu\in \NN^n: \, \aaa \,\trianglerighteq\, \uu \textup{ and}\\ \chi(\aaa) \,\trianglerighteq\, \chi(\uu)}}v^{(n,m-2)}_{\textup{unsplit}}(\aaa,\uu)\cdot v^{(n,1)}_{\textup{unsplit}}(\uu,\bb).
\end{align*}

From \cite{Pitman_Stanley_1999}, letting $sgn(\uu)=(c_{\uu,1}, \ldots, c_{\uu,k})$, we have that $v^{(n,1)}_{\textup{unsplit}}(\uu, \mathbf{0})=\prod_{i=1}^k (c_{\uu,i}+1)$. Moreover, as seen in Theorem~\ref{thm:matrixvertices}, we have that $v^{(n,m-2)}_{\textup{unsplit}}(\aaa,\uu)=((A_n)^{m-1})_{\chi(\aaa),\chi(\uu)}$. Therefore, $$v^{(n,m)}(\aaa)=\sum_{\substack{\uu\in \NN^n: \, \aaa \,\trianglerighteq\, \uu \textup{ and}\\ \chi(\aaa) \,\trianglerighteq\, \chi(\uu)}} \left(\prod_{i=1}^k (c_{\uu,i}+1)\right)\cdot ((A_n)^{m-1})_{\chi(\aaa),\chi(\uu)}.$$

\subsection{Volume and lattice points of $\PS_n^m(\aaa,\bb)$}

The volume and Ehrhart polynomial of the polytope $\PS_n^m(\aaa,\bb) \equiv \mathcal{F}_{G(n, m)}(\textbf{a},\textbf{b})$ will be discussed in a follow-up paper \cite{PSvolume}.

\section*{Acknowledgements} We were inspired by \cite{LMStD1,LMStD2} to extend a link between skew plane partitions with bounded parts and flow polytopes. We thank Jos\'e Cruz, Rafael Gonz\'alez D'Le\'on, Richard Stanley, Martha Yip, and the anonymous referees for helpful comments and suggestions. We also thank Jos\'e Cruz for doing examples that lead us to the current form of Theorem~\ref{thm:leading term v straight shape} from a previous statement involving shifted skew shapes.
This work was facilitated by computer experiments using Sage \cite{sagemath}, its algebraic combinatorics features developed by the Sage-Combinat community \cite{Sage-Combinat}. MH was supported by a Lee-SIP REU, WD and AHM were partially supported by NSF grant  DMS-1855536 and DMS-22030407, and AR was partially supported by NSF grant DMS-2054404.

An extended abstract of this work appeared as \cite{genPS-extended-abstract}.

\printbibliography

\appendix
\section{Proof of characterization of vertices of flow polytopes}\label{sec:appendixA}
\renewcommand{\thesubsection}{\Alph{subsection}}
\numberwithin{theorem}{subsection}
\numberwithin{equation}{subsection}

In this section, we give a self-contained proof of Theorem~\ref{char: vertices G as forests} from \cite{GalloSodini}. We need the following lemmas.

\begin{lemma}
\label{lem::a_b_valid_forest_unique_flow}
Let $G$ be a directed acyclic graph with netflow vector ${\bf a}$, and $F$ be a subgraph of $G$ (denoted as $F \subseteq G$) that is an ${\bf a}$-valid forest. Then $F$ is the support of a unique flow denoted as ${\bf x}_F$.
\end{lemma}

\begin{proof}
Given $G, {\bf a},$ and $F$ as in the statement of the lemma,  we denote by $T_1, \ldots , T_{\ell}$ the connected components of $F$ and by $m_1, \ldots, m_{\ell}$ the number of edges of $T_1, \ldots, T_{\ell}$ respectively. Given  ${\bf x} \in \RR^{|E(G)|}$ whose support is $F$, we denote by ${\bf \widetilde{x}}$ the vector obtained by removing the zero entries of $\mathbf{x}$. So we may write $ {\bf \widetilde{x}}:= (\widetilde{x}_{1,1}, \ldots , \widetilde{x}_{1, m_1}, \ldots, \widetilde{x}_{\ell,1}, \ldots , \widetilde{x}_{\ell, m_{\ell}})$ such that $\widetilde{x}_{i,j} > 0 $ for all $i,j$. Note that this means that $\widetilde{\xx}\in \mathbb{R}_{>0}^{|E(F)|}$.

Suppose ${\bf x}_0 \in \mathcal{F}_{G}({\bf a})$ is a  flow whose support is $F$ (note that such a flow exists since $F$ is an ${\bf a}$-valid graph). For any ${\bf x} \in \mathcal{F}_{G}({\bf a})$ with support $F$, then $M_F{\bf \widetilde{x} } = {\bf c}$, 
where $M_F$ is the (signed) incidence matrix of $F$ and ${\bf c}$ is the netflow vector of the vertices appearing in $F$ (ordered appropriately). Since we know $\widetilde{{\bf x}}_0$ is one solution, we know from linear algebra that the complete solution set to the system $M_F\widetilde{\bf x} = {\bf c}$ is equal to
$\{ \widetilde{{\bf x}} + \widetilde{{\bf x}}_0 \, | \, M_F \widetilde{{\bf x}} =0 \}$. 
However, $M_F$ has rank $|E(F)|$  since $F$ is a forest. This is the number of columns of our matrix, and  hence the linear transformation determined by $M_F$ is injective. Hence  the equation $M_F \widetilde{{\bf x}} = 0$  has the trivial solution set $\{ \widetilde{{\bf x}} = {\bf 0}\}$, making ${\bf x}_0$ the unique valid flow whose support is $F$.
\end{proof}

\begin{remark}
\label{rem::value_of_x_e_in_a_b_valid_forest}
Note that while Lemma \ref{lem::a_b_valid_forest_unique_flow} only proves the existence of such an ${\bf x}_F$, we can say exactly what the component $x_e$ is for $e \in E(F)$. Indeed, since $F$ is a forest, $e$ is a bridge and hence if $H_i$ and $H_j$  are the connected components that $e$ joins, it follows by conservation of flow that:
$$x_e = \left|\sum_{v \in V(H_i)} \text{netflow}(v) \right| = \left|\sum_{v \in V(H_j)}\text{netflow}(v)\right|.$$
\end{remark}

\begin{lemma}
\label{lem::a_b_valid_forest_is_edge_minimal}
Let $F$ be an ${\bf a}$-valid forest. Then there does not exist a proper subgraph $F' \subsetneq F$  on the same vertex set as $F$ such that $F'$ is an ${\bf a}$-valid forest.
\end{lemma}

\begin{proof}
Let $F$ be an ${\bf a}$-valid forest with connected components $T_1, \ldots , T_{\ell_1}$ and assume towards a contradiction that $F'$ is an ${\bf a}$-valid graph which is a proper subgraph $F' \subsetneq F$ with the same vertex set as $F$. Then in particular $F'$ has strictly more connected components than $F$, say $T'_1, \ldots, T'_{\ell_2}$ where $\ell_2>\ell_1$. Let ${\bf x}$ (respectively ${\bf x'}$) be the unique flow corresponding to $F$ (respectively $F'$) according to Lemma \ref{lem::a_b_valid_forest_unique_flow}. Now ${\bf x'}$ being a valid flow means that the netflow over each connected component is $0$. That is, for each $i$ from $1, \ldots , \ell_2$, we have
$$\sum_{v \in V(T'_i)} \text{netflow}(v) = 0.$$
Let $e$ be an edge of $E(F) \setminus E(F')$, and suppose without loss of generality that $e$ joins the components $T'_j$ and $T'_k$. Then $E(T'_j) \cup E(T'_k) \cup \{ e \}$ are the edges of a connected subgraph of some  $T_i$ in $F$, so from Remark \ref{rem::value_of_x_e_in_a_b_valid_forest} it follows that:
$$x_e  = \left|\sum_{v \in V(T'_j)} \text{netflow}(v) \right| = 0.$$
This contradicts the fact that $x_e > 0$ for $e \in E(F)$.
\end{proof}

\begin{proof}[Proof of Theorem~\ref{char: vertices G as forests}]
First suppose that ${\bf x}$ is a vertex of $\mathcal{F}_{G}({\bf a})$, but not an  ${\bf a}$-valid forest. Then in particular the support of the flow  ${\bf x}$ is not a forest, hence it contains an (undirected) cycle. In this case, we claim that ${\bf x}$ may be written as a convex combination ${\bf x} = \frac{1}{2}{\bf s} + \frac{1}{2}{\bf t}$ with ${\bf s}, {\bf t} \in \mathcal{F}_{G}({\bf a})$. To prove this claim, let $C$ be the undirected cycle of $G$ from above. We fix a planar embedding of $G$ in order to discuss the orientation of $C$, and denote by $v_{1}, \ldots , v_{{\ell}}$ the vertices of $C$ in clockwise order.  
We then define flows ${\bf s}$ and ${\bf t}$ by:
$$s_{e} := \begin{cases}x_{e} & \text{if $e  \not \in E(C)$} \\
x_{e} + \epsilon & \text{if $e = (v_i, v_j)$ and $i < j$}\\
x_{e} - \epsilon & \text{if $e = (v_i, v_j) $ and $ i > j$}\end{cases}$$
and:
$$t_{e} := \begin{cases}x_{e} & \text{if $e  \not \in E(C)$} \\
x_{e} - \epsilon & \text{if $e = (v_i, v_j)$ and $i < j$}\\
x_{e} + \epsilon & \text{if $e = (v_i, v_j) $ and $ i > j$}\end{cases}$$
for some $\epsilon > 0 $  small enough (specifically, we need $\epsilon < \min \{x_e \mid e \in C\}$). See Figure~\ref{fig:cycles PS}. Then for every edge $e$ in $G$ it follows that $s_e + t_e = 2x_e$, and so ${\bf x} = \frac{1}{2}{\bf s} + \frac{1}{2}{\bf t}$ as claimed. Hence ${\bf x}$ is not a vertex---a contradiction.

\begin{figure}[h!]
\centering
    \includegraphics[scale=0.2]{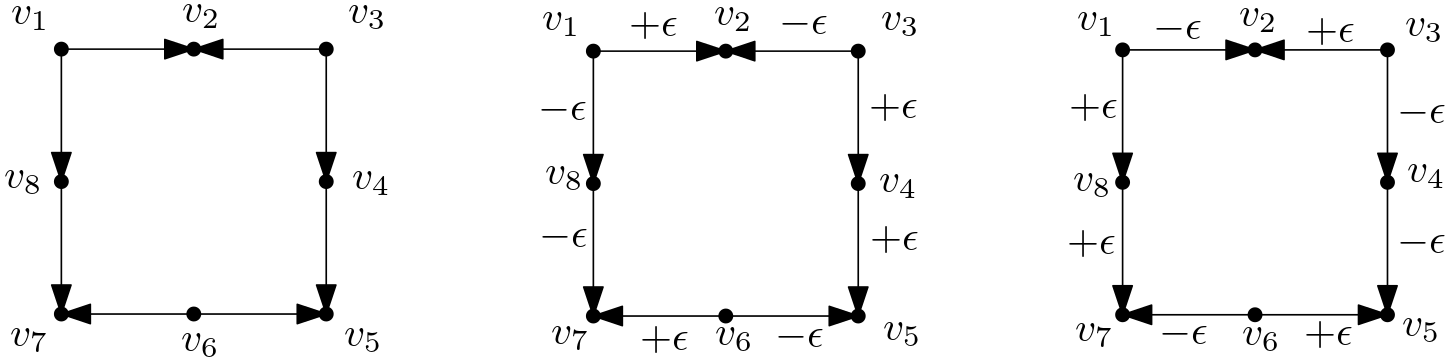}
    \caption{An example of adding and subtracting $\epsilon$ from the flows along edges in a cycle of the support of  ${\bf x}$ (left) to form ${\bf s}$ (middle) and ${\bf t}$ (right), as done in the proof of Theorem~\ref{char: vertices G as forests}. }
    \label{fig:cycles PS}
\end{figure}

For the other direction, suppose ${\bf x}$ is a flow whose support is an ${\bf a}$-valid forest $F_{{\bf x}}$, but which is not a vertex of $\mathcal{F}_{G}({\bf a})$. Then there exists $ 0 < \lambda < 1$ and ${\bf y}, {\bf z} \in \mathcal{F}_G({\bf a})$ such that ${\bf  x} = \lambda {\bf y} + (1 - \lambda ){\bf z}$. However, since the flow along any edge of $G$  is nonnegative, it follows that there is no cancellation among the components of ${\bf y}$ and ${\bf z}$, and hence the support of ${\bf y}$ and the support of ${\bf z}$ are both subsets of the support of ${\bf x}$. In particular, the support of ${\bf y}$ and the support of ${\bf z}$ are then both ${\bf a}$-valid forests as well, say $F_{{\bf y}}$ and $F_{{\bf z}}$ respectively. However, by Lemma~\ref{lem::a_b_valid_forest_is_edge_minimal}, we know that these supports cannot be proper subsets of $F_{{\bf x}}$. Hence it follows that 
$$F_{{\bf x}} = F_{{\bf y}} = F_{{\bf z}}$$
and by the uniqueness of flows corresponding to the support of an ${\bf a}$-valid forest given by Lemma~\ref{lem::a_b_valid_forest_unique_flow}, it follows that ${\bf x} = {\bf y} = {\bf z}$. Hence $\lambda \in \{0, 1\}$, a contradiction. 
\end{proof}

\newpage

\section{Tables of generating functions for vertices and unsplittable flows.} \label{appendix: tables}

\begin{table}[h!]
    \centering
$$\begin{array}{|c|c|c|}
\hline
\chi(\aaa) & d& (1-x)^{d+1}\cdot \sum_{m\geq 0} v^{5,m}(\aaa)x^m\\
\hline
\hline
(0,0,0,0,0) &0& 1\\
\hline
(0,0,0,0,1) &1 & 1\\
(0,0,0,1,0) &2 & 1\\
(0,0,1,0,0) &3 &1\\
(0,1,0,0,0) &4 &1\\
(1,0,0,0,0) &5 &1\\
\hline
(0,0,0,1,1) &3 &1\\
(0,0,1,0,1) &4 &1+x\\
(0,0,1,1,0) &5 &1+x^2\\
(0,1,0,0,1) &5 &1+2x\\
(0,1,0,1,0) &6 &1+2x+2x^2\\
(0,1,1,0,0) &7 &1+4x^2\\
(1,0,0,0,1) &6 &1+3x\\
(1,0,0,1,0) &7 &1+4x+4x^2\\
(1,0,1,0,0) &8 &1+3x+8x^2+2x^3\\
(1,1,0,0,0) &9 &1+10x^2+2x^3+x^4\\
\hline
(0,0,1,1,1) &6 & 1+x-x^2+x^3\\
(0,1,0,1,1) &7  &1+4x+2x^3\\
(0,1,1,0,1) &8  &1+3x+2x^2+6x^3\\
(0,1,1,1,0) &9 &1+2x+x^2+4x^3+4x^4\\
(1,0,0,1,1) &8  &1+7x+4x^2+4x^3\\
(1,0,1,0,1) &9 & 1+8x+13x^2+16x^3+4x^4\\
(1,0,1,1,0) &10  &1+7x+9x^2+19x^3+14x^4+4x^5\\
(1,1,0,0,1) &10  &1+5x+12x^2+30x^3+5x^4+3x^5\\
(1,1,0,1,0) &11  &1+6x+14x^2+38x^3+41x^4+6x^5+4x^6\\
(1,1,1,0,0) &12  &1+3x+10x^2+18x^3+53x^4+21x^5+2x^6+2x^7\\
\hline
(0,1,1,1,1) &10  &1+5x-5x^2+11x^3-4x^4+4x^5\\
(1,0,1,1,1) &11  &1+12x+10x^2+18x^3+13x^4+8x^5+4x^6\\
(1,1,0,1,1) &12  &1+11x+18x^2+56x^3+43x^4+39x^5+4x^6+4x^7\\
(1,1,1,0,1) &13  &1+10x+11x^2+52x^3+67x^4+104x^5+37x^6+0x^7+4x^8\\
(1,1,1,1,0) &14  &1+9x+5x^2+49x^3+19x^4+91x^5+79x^6+33x^7-4x^8+4x^9\\
\hline
(1,1,1,1,1) &15  &1+16x+4x^2+48x^3+62x^4+20x^5+88x^6+14x^7+37x^8-8x^9+4x^{10}\\
\hline 
\end{array}$$
    \caption{Generating functions $\sum_{m\geq 0}v^{5,m}(\aaa)x^m$.}
    \label{tab:genfunctions5}
\end{table}

\begin{table}[]
    \centering
$$\begin{array}{|c|c|c|}
\hline
\bb & d& (1-x)^{d+1}\cdot \sum_{m\geq 0} v^{5,m}({\bf 1}, \bb)x^m\\
\hline
\hline
(0,0,0,0,0) &15& 1+16x+4x^2+48x^3+62x^4+20x^5+88x^6+14x^7+37x^8-8x^9+4x^{10}\\
\hline
(0,0,0,0,1) &14& 1+20 x +68 x^{2}+134 x^{3}+151 x^{4}+191 x^{5}+180 x^{6}+123 x^{7}+64 x^{8}+4 x^{9}\\
(0,0,0,1,0) &13& 1 + 22x + 94x^2 + 126x^3 + 141x^4 + 184x^5 + 138x^6 + 96x^7 + 24x^8\\
(0,0,1,0,0) &12& 1+ 23x+82x^2+62x^3+88x^4+108x^5 + 48x^6 + 32x^7\\
(0,1,0,0,0) &11& 1+20x+ 30x^2+4x^3+57x^4 + 20x^6\\
(1,0,0,0,0) &10& 1+5x-5x^2+11x^3-4x^4+4x^5\\
\hline
(0,0,0,1,1) &12& 1+27 x +179 x^{2}+420 x^{3}+413 x^{4}+229 x^{5}+183 x^{6}+128 x^{7}+27 x^{8}+x^{9}\\
(0,0,1,0,1) &11& 1+29x+180x^2+349x^3+256x^4+161x^5+140x^6 + 48x^7 + 3x^8\\
(0,0,1,1,0) &10& 1+29x+166x^2+250x^3+121x^4+89x^5 + 55x^6 + 7x^7\\
(0,1,0,0,1) &10& 1+25x+94x^2+98x^3+52x^4+46x^5+31x^6+ 3x^7\\
(0,1,0,1,0) &9& 1+26x+99x^2+72x^3+25x^4+38x^5 + 9x^6\\
(0,1,1,0,0) &8& 1+23x+54x^2+10x^3+15x^4+9x^5\\
(1,0,0,0,1) &9& 1+8x+10x^2+7x^3+2x^4+6x^5+x^6\\
(1,0,0,1,0) &8& 1+9x+12x^2+5x^4+3x^5\\
(1,0,1,0,0) &7& 1+8x+2x^2+3x^4\\
(1,1,0,0,0) &6& 1 + x - x^2 + x^3\\
\hline
(0,0,1,1,1) &9& 1+ 35x+241x^2+475x^3+291x^4+48x^5+x^6\\
(0,1,0,1,1) &8& 1+31x + 156x^2 + 196x^3 + 61x^4+3x^5\\
(0,1,1,0,1) &7& 1 + 28x + 100x^2+72x^3+9x^4\\
(0,1,1,1,0) &6& 1+25x+67x^2+27x^3\\
(1,0,0,1,1) &7& 1 + 12x + 28x^2+14x^3+x^4\\
(1,0,1,0,1) &6&  1+11x+15x^2+3x^3\\
(1,0,1,1,0) &5& 1+10x+9x^2\\
(1,1,0,0,1) &5& 1+3x+x^2\\
(1,1,0,1,0) &4& 1+3x\\
(1,1,1,0,0) &3& 1\\
\hline
(0,1,1,1,1) &5& 1+26x+66x^2+26x^3+x^4\\
(1,0,1,1,1) &4& 1+11x+11x^2+x^3\\
(1,1,0,1,1) &3& 1+4x+x^2\\
(1,1,1,0,1) &2& 1+x\\
(1,1,1,1,0) &1& 1\\
\hline
(1,1,1,1,1) &0& 1\\ 
\hline 
\end{array}$$
    \caption{Generating functions $\sum_{m\geq 0}v^{5,m}({\bf 1}, \bb)x^m$.}
    \label{tab:genfunctions5abvert}
\end{table}

\begin{table}[]
    \centering
$$\begin{array}{|c|c|c|}
\hline
\bb & d& (1-x)^{d+1}\cdot \sum_{m\geq 0} v_\textup{unsplit}^{5,m}({\bf 1}, \bb)x^m\\
\hline
\hline
(0,0,0,0,0) &15 &1+16x+4x^2+48x^3+62x^4+20x^5+88x^6+14x^7+37x^8-8x^9+4x^{10}\\
\hline
(0,0,0,0,1) &14 &1+16x+4x^2+48x^3+62x^4+20x^5+88x^6+14x^7+37x^8-8x^9+4x^{10}\\
(0,0,0,1,0) &13& 1+16x+4x^2+48x^3+62x^4+20x^5+88x^6+14x^7+37x^8-8x^9+4x^{10}\\
(0,0,1,0,0) &12& 1+15x+2x^2+32x^3+71x^4-37x^5+96x^6-8x^7+4x^8\\
(0,1,0,0,0) &11& 1+12x-4x^2+22x^3+28x^4-20x^5+36x^6-12x^7+3x^8\\
(1,0,0,0,0) &10& 1+5x-5x^2+11x^3-4x^4+4x^5\\
\hline
(0,0,0,1,1) &12& 1+2x+16x^2-9x^3+57x^4-8x^5+22x^6+33x^7-8x^8+4x^9\\
(0,0,1,0,1) &11& 1 + 9x- 2x^2+36x^3+17x^4+x^5+52x^6-8x^7+4x^8\\
(0,0,1,1,0) &10& 1+3x+26x^3-11x^4+23x^5+14x^6-2x^7+2x^8\\
(0,1,0,0,1) &10& 1+10x-6x^2+26x^3+10x^4-3x^5+21x^6-7x^7+2x^8\\
(0,1,0,1,0) &9& 1+8x-8x^2+30x^3-8x^4+14x^5+6x^6-2x^7+x^8\\
(0,1,1,0,0) &8& 1+3x-2x^2+10x^3+x^4+3x^5\\
(1,0,0,0,1) &9& 1+5x-5x^2+11x^3-4x^4+4x^5\\
(1,0,0,1,0) &8& 1+5x-5x^2+11x^3-4x^4+4x^5\\
(1,0,1,0,0) &7& 1+4x-4x^2+6x^3\\
(1,1,0,0,0) &6& 1 + x - x^2 + x^3\\
\hline
(0,0,1,1,1) &9& 1-3x+8x^2-5x^4+15x^5-4x^6+2x^7\\
(0,1,0,1,1) &8& 1+2x^2+7x^3-5x^4+11x^5-3x^6+x^7\\
(0,1,1,0,1) &7& 1+x-2x^2+8x^3-x^4+2x^5\\
(0,1,1,1,0) &6& 1-x+4x^3-x^4+x^5\\
(1,0,0,1,1) &7& 1-x+5x^2-4x^3+4x^4\\
(1,0,1,0,1) &6&  1+2x-2x^2+4x^3\\
(1,0,1,1,0) &5& 1 + 2x^3\\
(1,1,0,0,1) &5& 1 + x - x^2 + x^3\\
(1,1,0,1,0) &4& 1+x-x^2+x^3\\
(1,1,1,0,0) &3& 1\\
\hline
(0,1,1,1,1) &5& 1-3x+4x^2-2x^3+x^4\\
(1,0,1,1,1) &4& 1-2x+2x^2\\
(1,1,0,1,1) &3& 1-x+x^2\\
(1,1,1,0,1) &2& 1\\
(1,1,1,1,0) &1& 1\\
\hline
(1,1,1,1,1) &0& 1\\ 
\hline 
\end{array}$$
    \caption{Generating functions $\sum_{m\geq 0}v^{5,m}_\textup{unsplit}({\bf 1}, \bb)x^m$.}
    \label{tab:genfunctions5ab}
\end{table}

\end{document}